\documentclass[a4paper,12pt]{article}

\usepackage{comment,amsmath,amsfonts,dsfont,amssymb,amsthm,mathrsfs,stmaryrd,enumitem}
\usepackage{tikz}
\usepackage[utf8]{inputenc}
\usepackage{centernot}
\usepackage{hyperref}
\usepackage[toc,page]{appendix} 
\newtheorem{thm}{Theorem}[section]
\newtheorem{defi}[thm]{Definition}
\newtheorem{prop}[thm]{Proposition}
\newtheorem{cor}[thm]{Corollary}

\newtheorem{lem}[thm]{Lemma}
\numberwithin{equation}{section}
\usepackage[final]{showkeys}
\usepackage{fullpage}

\newcommand\nlongleftrightarrow{
  \mathrel{\ooalign{$\longleftrightarrow$\cr
    \hskip 0pt plus 0.8fill $\arrownot$ \hskip 0pt plus 2fill\cr}}
}

\newcommand\oset{\big\{\,}
\newcommand\cset{\,\big\}}

\everymath{\displaystyle}
\begin{document}
\title{A new look at the interfaces in percolation}
\author{
\qquad Rapha\"el Cerf\footnote{
\noindent
D\'epartement de math\'ematiques et applications, Ecole Normale Sup\'erieure,
CNRS, PSL Research University, 75005 Paris.
\newline
Laboratoire de Math\'ematiques d'Orsay, Universit\'e Paris-Sud, CNRS, Universit\'e
Paris--Saclay, 91405 Orsay.}
\hskip 70pt Zhou Wei\footnotemark[1]
}
\maketitle
\begin{abstract}
We propose a new definition of the interface in the context of the Bernoulli percolation model.
We construct a coupling between two percolation configurations, one which is a standard percolation
configuration, and one which is a percolation configuration conditioned on a disconnection event.
We define the interface as the random set of the edges where these two configurations differ.
We prove that, inside a cubic box $\Lambda$, the interface between the top and the bottom of the box
is typically localised within a distance of order
$(\ln |\Lambda|)^2$ of the set of the pivotal edges. 
We prove also that, in our dynamical coupling, the typical speed of the pivotal edges remains bounded
as the box $\Lambda$ grows.
\end{abstract}

\section{Introduction}
At the macroscopic level, the interface between two pure phases seems to be deterministic. In fact, 
such an interface obeys a minimal action principle: it minimizes the surface tension between the two phases
and it is close to the solution of a variational problem. This can be seen as an empirical law, derived
from the observation at the macroscopic level.
This law has been justified from a microscopic point of view
in the context of the Ising model \cite{MR1181197}. 
One starts with a simple model
of particles located on a discrete lattice. There are two types of particles, which have a slight tendency
to repel each other. In the limit where the number of particles tends to $\infty$, at low temperatures,
the system presents a phenomenon of phase segregation, with the formation of interfaces between two
pure phases. On a suitable scale, these interfaces converge towards deterministic shapes, 
a prominent example being the Wulff crystal of the Ising model, which is the typical shape of the 
Ising droplets.
Although the limit is deterministic on the macroscopic level, the interfaces are intrinsically random objects
and their structure is extremely complex.
In two dimensions, the fluctuations of the Ising interfaces were precisely analysed in the DKS theory,
with the help of cluster expansions \cite{DH,MR1181197}. 
In higher dimensions, there is essentially one result on the fluctuations of the interfaces, due to
Dobrushin \cite{D72Gibbs}, which says that horizontal interfaces stay localised at low temperatures.
When dealing with interfaces in the Ising model, the first difficulty is to get a proper definition 
of the interface itself. The usual way is to start with the Dobrushin type boundary conditions, that is
a box with pluses on its upper half boundary and minuses on its lower half boundary. This automatically
creates an Ising configuration in the box with a microscopic interface between the pluses
and the minuses which separates the upper half and the lower half of the box.
Yet it is still not obvious how one should define the interface in this case, because 
several such microscopic interfaces exist, and a lot of different choices are possible. 
Dobrushin, Kotecky and Shlosman \cite{MR1181197} introduced a splitting rule between contours, which leads to pick up one particular microscopic
interface. The potential problem with this approach is that the outcome is likely to include
microscopic interfaces which are not necessarily relevant, for instance interfaces between opposite
signs which would have been present anyway, and which are not induced by the Dobrushin boundary conditions.

Our goal here is to propose a new way to look at the random interfaces, in any dimension $d\geq 2$.
We start our investigation in the framework of the Bernoulli percolation model, for several reasons.
First, the probabilistic structure of the percolation model is simpler than the one of the Ising model.
Another reason is that
the Wulff theorem in dimensions three was first derived for the percolation model \cite{MR1774341}
and then extended to the Ising model \cite{MR1724851,Cerf2000Wulff}.
A key fact was that the definition of the surface tension is much simpler for the percolation model
than for the Ising model. This leads naturally to hope that the probabilistic structure of the interfaces 
should be easier to apprehend as well in the percolation model.
Finally, in the context of percolation, one sees directly which edges are essential or not in
an interface: these are the pivotal edges. There is no corresponding notion in the Ising model.
For all these reasons, it seems wise to try to develop
a probabilistic description of random interfaces in the framework of Bernoulli percolation.

In this paper,
we consider the Bernoulli bond percolation model with a parameter $p$ close to $1$.
Interfaces in a cubic box $\Lambda$
are naturally created when the configuration is conditioned on the event that
the top $T$ and the bottom $B$ of the box are disconnected. From now onwards, this event is 
denoted by
$$\oset T\nlongleftrightarrow B \cset.$$
Our goal is to gain some understanding on the typical configurations realizing such a disconnection event.
To do so, we build a coupling between two percolation configurations $X,Y$ in the box $\Lambda$ such that:

\noindent $\bullet$ The edges in $X$ are i.i.d., open with probability $p$ and closed 
with probability $1-p$.

\noindent $\bullet$ The distribution of $Y$ is the distribution of the Bernoulli percolation conditioned on
$\oset T\nlongleftrightarrow B \cset$.

\noindent $\bullet$ Every edge open in $Y$ is also open in $X$.

\noindent 
We define then the random interface between the top $T$ and the bottom $B$ of the
box $\Lambda$ as the random set $\mathcal{I}$ of the edges where $X$ and $Y$ differ:
$$\mathcal{I}\,=\,\big\{\,e\subset \Lambda: X(e)\text{ is open},\,\,Y(e)\text{ is closed}\,\big\}\,.$$
Among these edges, some are essential for the disconnection between $T$ and $B$ to occur. These
edges are called pivotal and they are denoted by
$\mathcal{P}$:
$$\mathcal{P}\,=\,\bigg\{\,e\in\mathcal{I}:
\begin{array}{c}
	\text{the opening of $e$ in $Y$ would create}\\
\text{an open connection between $T$ and $B$}
\end{array}
\,\bigg\}\,.$$
When conditioning on the disconnection between $T$ and $B$, a lot of pivotal edges are created.
Yet another collection of edges which are not essential for the disconnection event turn out 
to be closed as well. Therefore it becomes extremely difficult to understand the effect of the 
conditioning on the distribution by looking at the conditioned probability measure alone.
This is why we build a coupling and we define the interface as the set of the edges where the two
percolation configurations differ.
The set $\mathcal{P}$ of the pivotal edges can be detected by a direct inspection of the conditioned
configuration, but not the interface $\mathcal{I}$. 
Our main result provides a quantitative control on the interface
$\mathcal{I}$ with respect to the set $\mathcal{P}$ of the pivotal edges.
We denote by $\mu_p$ the coupling probability measure between the configurations $X$ and $Y$. The precise construction of $\mu_p$ is done in section~\ref{new.mono}.
We denote by $d$ the usual Euclidean distance on $\mathbb{R}^d$, by $\Lambda$ a cubic box with sides
parallel to the axis of $\mathbb{Z}^d$, and by $|\Lambda|$ the cardinality of $\Lambda\cap \mathbb{Z}^d$.
\begin{thm}\label{new.main2}
There exists $\tilde{p}< 1$ and $\kappa>0$, such that, for $p\geqslant \tilde{p}$, any $c\geqslant 1$ and any box $\Lambda$ satisfying $|\Lambda|\geqslant \max\left\lbrace (cd)^{cd^2},3^{6d}\right\rbrace$, 
$$\mu_p \Big(\exists e\in \mathcal{P}\cup\mathcal{I},d\left(e,\Lambda^c\cup\mathcal{P}\setminus\{e\}\right)
	\geqslant \kappa c^2\ln^2 |\Lambda| \Big)\leqslant \frac{1}{|\Lambda|^c}.
$$
\end{thm}
\noindent The typical picture which emerges from 
theorem~\ref{new.main2} is the following. In the configuration conditioned on the event
$\oset T\nlongleftrightarrow B \cset$, there is a set $\mathcal{P}$ of pivotal edges.
These are the edges having one extremity connected by an open path to the top $T$ and
the other extremity connected by an open path to the bottom $B$.
Because of the conditioning, compared to the i.i.d. configuration,
some additional edges are closed, but they are typically within a distance of order
	$(\ln |\Lambda|)^2$ of the set $\mathcal{P}$ of the pivotal edges. 
	The edges which are further away from $\mathcal{P}$ behave as in the ordinary unconditioned
	percolation.
	Therefore the interface $\mathcal{I}$ is strongly localised around the set $\mathcal{P}$ of the
	pivotal edges. The interface is a dust of closed edges pinned around the pivotal edges.

	Our next endeavour was to obtain a conditional version of theorem~\ref{new.main2}.
	More precisely, we would like to estimate the conditional probability
$$\mu_p \Big( e\in \mathcal{P}\cup\mathcal{I}\,\Big|\,d
\big(e,\mathcal{P}\setminus\{e\}\big)
	\geqslant \kappa(\ln |\Lambda|)^2 \Big)\,.$$
We did not really succeed so far, however we are able to control the interface conditionally on the distance to a cut.
Before stating our result, let us recall the definition of a cut.
\begin{defi} A set $S$ of edges 
	separates the top $T$ and the bottom $B$ 
	in $\Lambda$ 
	if
every deterministic path of edges from $T$ to $B$ in $\Lambda$ intersects $S$.
\noindent A cut $C$ between $T$ and $B$ in $\Lambda$ is a set of edges which separates $T$ and $B$ in $\Lambda$ and
which is minimal for the inclusion.
\end{defi} 
A cut $C$ is closed in the configuration $Y$ if all the edges of $C$ are closed in $Y$.
We denote by $\mathcal{C}$ the collection of the closed cuts present in $Y$. Since
$Y$ realizes the event 
$\oset T\nlongleftrightarrow B \cset$, the collection $\mathcal{C}$ is not empty.
\begin{thm}\label{new.main1}
We have the following inequality:
\begin{multline*}
\exists \tilde{p} <1 \quad\exists \kappa>0 \quad \forall p\geqslant \tilde{p}\quad\forall c\geqslant 2 \quad  \forall \Lambda\quad \ln|\Lambda|>4+c+2dc^2+12(2\kappa d)^d \\
\forall e \in \Lambda \quad d(e,\Lambda^c)\geqslant \kappa c^2\ln^2|\Lambda|\\ \mu_p\left(e\in \mathcal{I}\,\Big|\,
	\exists C\in \mathcal{C}, d(e,C)\geqslant \kappa c^2\ln^2|\Lambda|\right) \leqslant \frac{1}{|\Lambda|^c}.
\end{multline*}
\end{thm}

Let us explain briefly how we build the coupling probability measure $\mu_p$, as well as
the strategy for proving theorem~\ref{new.main2}. 
Conditioning on the event
$\oset T\nlongleftrightarrow B \cset$ creates non trivial correlations between the edges, and
there is no simple tractable formula giving for instance the conditional distribution of a finite
set of edges. Yet a standard application of the FKG inequality yields that, for any increasing event
$A$, we have
$$P_p\big(A\,\big|\, T\nlongleftrightarrow B\big)\,\leq\,P_p(A)\,.$$ 
Thus the product measure $P_p$ stochastically dominates the conditional measure
$P_p(\cdot\,\big|\, T\nlongleftrightarrow B\big)$. Strassen's theorem tells us that there
exists a monotone coupling between these two probability measures.
In order to derive quantitative estimates on the differences between the coupled configurations, we 
build our coupling measure as the invariant measure of a dynamical process. This method of coupling is standard, for instance it is used in the proof
of Holley inequality (see chapter 2 of~\cite{MR1379156}). Our contribution is to study some specific properties
of this coupling in the context of percolation, and to relate it to the geometry
of the interfaces. To do so, we consider
the classical dynamical percolation process in the box $\Lambda$, see~\cite{MR2762676}. Since we always work in a finite box, 
we use the following discrete time version. We start with an initial configuration $X_0$.
At each step, we choose one edge uniformly at random, and we update its state with a coin of parameter $p$.
Of course all the random choices are independent. 
The resulting process is denoted by $(X_t)_{t\in\mathbb{N}}$.
Obviously the invariant probability measure of 
$(X_t)_{t\in\mathbb{N}}$ is the product measure $P_p$ and the process
$(X_t)_{t\in\mathbb{N}}$ is reversible with respect to $P_p$.
Next, we duplicate the initial configuration $X_0$, thereby
getting 
a second configuration $Y_0$.
We use the same random variables as before to update this second configuration, with one essential difference.
In the second configuration, we prohibit the opening of an edge if this opening creates a connection 
between the top $T$ and the bottom $B$. This mechanism ensures that $X_t$ is always above $Y_t$.
Moreover, a classical result on reversible Markov chains ensures that the invariant probability measure of 
the process
$(Y_t)_{t\in\mathbb{N}}$ is the conditional probability measure 
$P_p(\cdot\,\big|\, T\nlongleftrightarrow B\big)$. 
Our coupling probability measure $\mu_p$ is defined as the invariant probability measure of the 
process
$(X_t,Y_t)_{t\in\mathbb{N}}$.
In the case of the Ising model, 
where one has access to an explicit formula for the equilibrium measure, 
one usually derives results on the dynamics (for instance the Glauber dynamics) 
from results on the Ising Gibbs measure.
We go here in the reverse direction: we use our dynamical construction to derive results on the
equilibrium measure $\mu_p$.

For the proof, we consider the stationary process $(X_t,Y_t)_{t\in\mathbb{N}}$ starting from its equilibrium distribution $\mu_p$.
We fix a time $t$ and we estimate the probability that
the configuration $(X_t,Y_t)$ 
realizes the event appearing in the statement of theorem~\ref{new.main2}.
We distinguish the case of edges in the interface which are pivotal or not.
For pivotal edges, we shall prove the following 
	slightly stronger result.
\begin{prop} \label{new.distpivot}
There exists $\tilde{p}<1$ and $\kappa>1$ such that, for $p\geqslant \tilde{p}$, and for any $c \geqslant 1$ and any box $\Lambda$ satisfying $|\Lambda|> 3^{6d}$, we have
$$ P_p \Big(\exists e\in \mathcal{P}, d(e,\Lambda^c\cup\mathcal{P}\setminus \{e\})\geqslant \kappa c\ln|\Lambda|
	\,\Big|\,
 T\nlongleftrightarrow B 
	\Big)\leqslant \frac{1}{|\Lambda|^c}.
$$
\end{prop}
\noindent The proof of this proposition relies on the BK inequality. 
We consider next the case of an edge $e$ in the interface which is not pivotal.
Such an edge $e$ can be opened at any time in the configuration $Y$. 
Therefore, unless it becomes pivotal again,
it cannot stay for a long time in the interface. 
In addition, 
before becoming part of the interface $\mathcal{I}$, 
the edge $e$ must have been pivotal. Indeed, non--pivotal
edges in the process 
$(Y_t)_{t\in\mathbb{N}}$ evolve exactly as in the process
$(X_t)_{t\in\mathbb{N}}$.
We look backwards in the past at the last time when the edge $e$ was still pivotal.
As said before, this time must be quite close from $t$. However, at time $t$, it turns out
that 
the set of the pivotal edges is quite far from $e$.
We conclude that
the set of the pivotal edges must have moved away from $e$ very fast.
%
To estimate the
probability of a fast movement of the set $\mathcal{P}$, we derive an estimate on the speed of the
set of the pivotal edges, which is stated in proposition~\ref{new.vconst}. This estimate is at the heart
of the argument. It relies on the construction of specific space--time paths, 
which describe how the influence
of the conditioning propagates in the box.
If a space--time path travels over a long distance in a short time, then this implies
that a certain sequence of closing events has occurred, and we estimate the corresponding
probability. This estimate is delicate, because the closed space--time path can take
advantage of the pivotal edges which remain closed thanks to the conditioning. The computation
relies again on the BK inequality, this time applied to the space--time paths.

The statement of theorem~\ref{new.main2} naturally prompts several questions.
First,
the results presented here hold only for values of $p$ sufficiently close to $1$, 
because the proofs rely on Peierls arguments.
\smallskip

\noindent
{\bf Question 1}.
Is it possible to prove an analogous
result throughout the supercritical regime $p>p_c$?
\smallskip

\noindent
Proposition~\ref{new.distpivot} shows that, typically, 
each pivotal edge is within a distance of order
	$\ln |\Lambda|$ of another pivotal edge. Of course, we would like to understand better 
	the random set $\mathcal{P}$.
\smallskip

\noindent
{\bf Question 2}.
What else can be said about the structure of the set $\mathcal{P}$?

\noindent This question is essential to understand the fluctuation of the interfaces. 

\smallskip
\noindent
{\bf Question 3}.
Is it possible to replace 
	$(\ln |\Lambda|)^2$ by
	$\ln |\Lambda|$ in the statement of 
theorem~\ref{new.main2}?

\noindent
Since there is no square in the logarithm appearing in proposition~\ref{new.distpivot}, we
suspect that it should also be the case in the statement of theorem~\ref{new.main2}.
Despite serious efforts, we did not manage to remove the square in the logarithm so far. However,
we obtained a partial result in this direction: we managed to improve the control of the speed of
the pivoltal edges.
More precisely, we prove the following result.
\begin{thm}\label{eps.main}
There exists $\tilde{p}<1$ such that for $p\geqslant\tilde{p}$, $c\geqslant 1$ and any box $\Lambda$ satisfying $|\Lambda|\geqslant e^{2d^2c}$, we have
$$P_\mu\left( d_H^{2d c\ln|\Lambda|}\Big(\bigcup_{r\in [t-c|\Lambda|\ln\Lambda,t]}\mathcal{P}_r,\bigcup_{s\in[t,t+c|\Lambda|\ln\Lambda]}\mathcal{P}_s\Big)\geqslant 2d c\ln|\Lambda|\right)\leqslant \frac{1}{|\Lambda|^{c-3}}.
$$
\end{thm}
\noindent
To control the distance between pivotal edges at two different times, we need to control the speed of the pivotal edges. 
In the proof of theorem~\ref{new.main2}, we obtain a control on the speed of these edges during a time interval of order $|\Lambda|$. 
If we apply these results
on a time interval of length $2dc|\Lambda|\ln|\Lambda|$, we can bound the distance of displacement of
the pivotal edges by $\ln^2|\Lambda|$ instead of
$\ln|\Lambda|$. To remove the square, we need a new ingredient compared to the previous argument. 
We shall obtain a speed estimate on a time interval of order $|\Lambda|\ln|\Lambda|$ by studying a new type of space-time path which connects a pivotal edge at time $t$ to an edge of $\mathcal{P}_s\cup\mathcal{I}_s$ at a time $s<t$. The length of this new type of space-time paths has an exponential decay property during a time interval of order $|\Lambda|\ln|\Lambda|$. As a drawback, we have to replace $\mathcal{P}$ by $\mathcal{P}\cup\mathcal{I}$ due to the construction of this new space-time path. As a consequence, we can only study the distance between a pivotal edge and the union of the pivotal edges on a time interval in the past. 
\smallskip

\noindent
Ultimately, we would like to gain some understanding on the Ising interfaces.
The natural road to transfer percolation results towards the Ising model is to use the FK percolation
model. However, there are several difficulties to overcome in order to adapt the proof to FK percolation.
First, we use the BK inequality twice in the proof, and this inequality is not available in the FK model.
Second, the dynamics for the FK model is more complicated.
\smallskip

\noindent
{\bf Question 4}.
Does theorem~\ref{new.main2} extend to the FK percolation model?
\smallskip

\noindent
Suppose that the answer to question 4 is positive. It is not obvious
to transcribe 
theorem~\ref{new.main2} in the Ising context. For instance, the pivotal edges, which can be detected
by visual inspection of a percolation configuration, are hidden inside the associated Ising configuration.
\smallskip

\noindent
{\bf Question 5}.
What is the counterpart of
theorem~\ref{new.main2} for Ising interfaces?
\smallskip

\noindent
The questions 4 and 5 are addressed in \cite{lowTIsingLocal}.

The paper is organized as follows.
In section~\ref{new.mono}, we define precisely the model and the notations. Beyond the classical percolation
definitions, this section contains the definition of the space--time paths and the graphical construction of the
coupling.
Section~\ref{new.iso} is devoted to the proof of
proposition~\ref{new.distpivot}. 
In section~\ref{new.speed}, we prove the central result on the control of the speed of the set of the pivotal edges.
Then, the theorem~\ref{new.main2} is proved in section~\ref{new.th1}.
In section \ref{new.speedI}, we improve the results obtained in section~\ref{new.speed} and we prove the theorem~\ref{new.main1} in section~\ref{new.th2}.
Then we continue our study to prove theorem \ref{eps.main}.
In section~\ref{eps.constp}, we construct the new space-time path which will be used in the proof. In section \ref{eps.speedcontrol}, we control the distance between $\mathcal{P}_t\cup\mathcal{I}_t$ and $\mathcal{P}_{t+s}$ for $s\leqslant |\Lambda|\ln|\Lambda|$ with the help of this space-time path. Finally, the proof of theorem~\ref{eps.main} is presented in the section \ref{eps.proofmain}.

\paragraph{Acknowledgement.} We warmly thank Jean-Baptiste Gouéré for his attentive reading
and for numerous constructive comments which were essential to improve the presentation of the results and the clarity of the proofs. 


\section{The model and notations}
\label{new.mono}
\subsection{Geometric definitions}
We give standard geometric definitions.
\paragraph{The edges $\mathbb{E}^d$.} The set of edges $\mathbb{E}^d$ is the set of the pairs $\{x,y \}$ of points of $\mathbb{Z}^d$ which are at Euclidean distance 1. 

\paragraph{The box $\Lambda$.} We will mostly work in a closed box $\Lambda$ centred at the origin. We denote by $T$ the top side of $\Lambda$ and by $B$ its bottom side.

\paragraph{The separating sets.} Let $A,B$ be two subsets of $\Lambda$. We say that a set of edges $S\subset \Lambda$ separates $A$ and $B$ if no connected subset of $\Lambda\cap\mathbb{E}^d\setminus S$ intersects both $A$ and $B$. Such a set $S$ is called a separating set for $A$ and $B$. We say that a separating set is minimal if there does not exist a strict subset of $S$ which separates $A$ and $B$.

\paragraph{The cuts.} We say that $S$ is a cut if $S$ separates $T$ and $B$, and $S$ is minimal for the inclusion.

\paragraph{The usual paths.} We say that two edges $e$ and $f$ are neighbours if they have one endpoint in common. A usual path is a sequence of edges $(e_,\dots,e_n)$ such that for $1\leqslant i <n$, the edge $e_i$ and $e_{i+1}$ are neighbours.

\paragraph{The $*$-paths.} In order to study the cuts in any dimension $d \geqslant 2$, we use $*$-connectedness on the edges as in~\cite{Pisztora1996Surface}. We consider the supremum norm on $\mathbb{R}^d$:
$$\forall x = (x_1,\dots,x_d)\in \mathbb{R}^d\qquad\parallel x\parallel_\infty = \max_{i = 1,\dots,d}|x_i|.
$$
For $e$ an edge in $\mathbb{E}^d$, we denote by $m_e$ the center of the unit segment associated to $e$. We say that two edges $e$ and $f$ of $\mathbb{E}^d$ are $*$-neighbours if $\parallel m_e-m_f\parallel_\infty\leqslant 1$. A $*$-path is a sequence of edges $(e_1,\dots,e_n)$ such that, for $1\leqslant i <n$, the edge $e_i$ and $e_{i+1}$ are $*$-neighbours.
\subsection{The dynamical percolation.} 
We define the dynamical percolation and the space-time paths.
\paragraph{Percolation configurations.} A percolation configuration in $\Lambda$ is a map from the set of the edges included in $\Lambda$ to $\{0,1\}$. An edge $e\subset \Lambda$ is said to be open in a configuration $\omega$ if $\omega(e) =1$ and closed if $\omega(e) = 0$. For two subsets $A,B$ of $\Lambda$ and a configuration $\omega\in \Omega$, we denote by $A\overset{\omega}{\longleftrightarrow} B$ the event that there is an open path between a vertex of $A$ and a vertex of $B$ in the configuration $\omega$.
\paragraph{Probability measures.} We denote by $P_p$ the law of the Bernoulli bond percolation in the box $\Lambda$ with parameter $p$. The probability $P_p$ is the probability measure on the set of bond configurations which is the product of the Bernoulli distribution $(1-p)\delta_0+p\delta_1$ over the edges included in $\Lambda$. We define $P_\mathcal{D}$ as the probability measure $P_p$ conditioned on the event $\oset T\nlongleftrightarrow B\cset$, i.e.,
$$P_\mathcal{D}(\cdot) = P_p\left(\,\cdot \,\big|\,T\nlongleftrightarrow B\,\right).
$$
\paragraph{Probability space.} Throughout the paper, we assume that all the random variables used in the proofs are defined on the same probability space $\Omega$. For instance, this space contains the random variables used in the graphical construction presented below, as well as the random variables generating the initial configurations of the Markov chains. We denote simply by $P$ the probability measure on $\Omega$.

\paragraph{Graphical construction.}
We now present a graphical construction of the dynamical percolation in the box $\Lambda$. We build a sequence of triplets $(X_t,E_t,B_t)_{t\in \mathbb{N}}$, where $X_t$ is the percolation configuration in $\Lambda$ at time $t$, $E_t$ is a random edge in the box $\Lambda$ and $B_t$ is a Bernoulli random variable. The sequence $(E_t)_{t\in \mathbb{N}}$ is an i.i.d. sequence of edges, with uniform distribution over the edges included in $\Lambda$. The sequence $(B_t)_{t\in \mathbb{N}}$ is an i.i.d. sequence of Bernoulli random variables with parameter $p$. The sequence $(E_t)_{t\in \mathbb{N}}$ and $(B_t)_{t\in \mathbb{N}}$ are independent. The process $(X_t)_{t\in \mathbb{N}}$ is built iteratively as follows. At time $0$, we start from the configuration $X_0$, which might be random. At time $t$, we change the state of $E_t$ to $B_t$ and we set 
$$\forall t\geqslant 1  \qquad X_t(e) = \left\lbrace\begin{array}{cc}
X_{t-1}(e) & \text{if } E_t \neq e\\
B_t & \text{if } E_t = e
\end{array}\right..
$$
The process $(X_t)_{t\in \mathbb{N}}$ is the dynamical percolation process in the box $\Lambda$.

\paragraph{The space-time paths.} We introduce the space-time paths which generalise both the usual paths and the $*$-paths to the dynamical percolation. A space-time path is a sequence of pairs, called time-edges, $(e_i,t_i)_{1\leqslant i \leqslant n}$, such that, for $1\leqslant i \leqslant n-1$, we have either $e_i= e_{i+1}$, or $(e_i$, $e_{i+1}$ are neighbours and $t_i = t_{i+1})$. We say that a space-time path $(e_i,t_i)_{1\leqslant i \leqslant n}$ is during a time interval $[s,t]$ if for all $1\leqslant i\leqslant n$, we have $t_i\in[s,t]$. We define also space-time $*$-paths, by using edges which are $*$-neighbours in the above definition. For $s,t$ two integers, we define 
$$ s\wedge t = \min(s,t), \qquad s\vee t = \max(s,t).
$$
A space-time path $(e_i,t_i)_{1\leqslant i \leqslant n}$ is open in the dynamical percolation process $(X_t)_{t\in \mathbb{N}}$ if
$$\forall i\in \oset 1,\dots ,n \cset \quad X_{t_i}(e_i) = 1$$ and
$$\forall i\in \oset 1,\dots ,n-1 \cset\quad e_i = e_{i+1}\quad\Longrightarrow \quad\forall t\in[t_i\wedge t_{i+1},t_i\vee t_{i+1}] \quad X_t(e_i) = 1.
$$
In the same way, we can define a closed space-time path by changing $1$ to $0$ in the previous definition. In the remaining of the article, we use the abbreviation STP to design a space-time path. Moreover, unless otherwise specified, the closed paths (and the closed STPs) are defined with the relation $*$ and the open paths (and the open STPs) are defined with the usual relation. This is because the closed paths come from the cuts, while the open paths come from existing connexions. 

\subsection{The interfaces by coupling.} 
We propose a new way of defining the interfaces by coupling two processes of dynamical percolation. We start with the graphical construction $(X_t,E_t,B_t)_{t\in \mathbb{N}}$ of the dynamical percolation. We define a further process $(Y_t)_{t\in\mathbb{N}}$ as follows: at time 0, we set $X_0=Y_0$, and for all $t\geqslant 1$, we set
$$ \forall e \subset \Lambda\quad Y_t(e) = \left\lbrace \begin{array}{cl}
Y_{t-1}(e) & \text{if }\quad e \neq E_t\\
0 &\text{if }\quad e = E_t\text{ and } B_t = 0\phantom{T\overset{Y_{t-1}^{E_t}}{\nlongleftrightarrow} B}\\
1 & \text{if }\quad e=E_t, B_t = 1 \text{ and } T\overset{Y_{t-1}^{E_t}}{\nlongleftrightarrow} B\\
0 & \text{if }\quad e=E_t, B_t = 1 \text{ and } T\overset{Y_{t-1}^{E_t}}{\longleftrightarrow} B
\end{array} \right.,
$$
where, for a configuration $\omega$ and an edge $e$, the notation $\omega^e$ means the configuration obtained by opening $e$ in $\omega$. Typically, we start with a configuration $Y_0$ realizing the event $\oset T\nlongleftrightarrow B\cset$, but this is not mandatory in the above definition. An illustration of this dynamics is given in the figure~\ref{new.fig:deco}.
We denote by $P_\mathcal{D}$ the equilibrium distribution of the process $(Y_t)_{t\in\mathbb{N}}$.
\begin{figure}[ht] 
\centering
\includegraphics[width = 10cm]{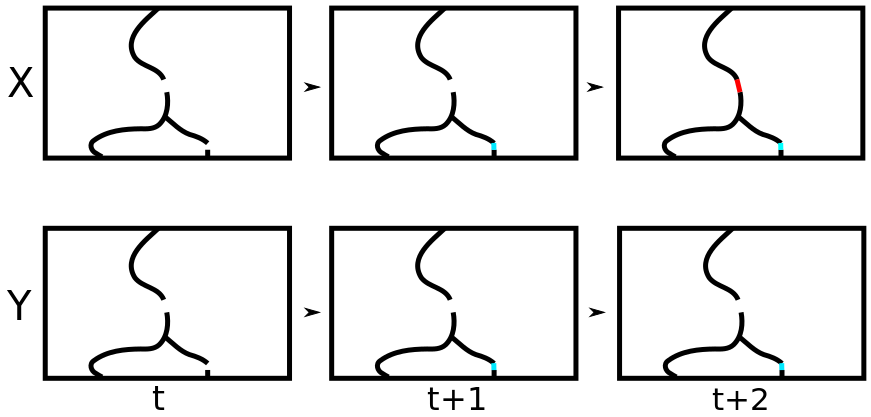}
\caption{A coupling of the process $(X_t,Y_t)_{t\in\mathbb{N}}$. At time $t+1$ we try to open the blue edge and at time $t+2$, we try to open the red edge.}
\label{new.fig:deco}
\end{figure}
Before opening a closed edge $e$ at time $t$, we verify whether this will create a connexion between $T$ and $B$. If it is the case, the edge $e$ stays closed in the process $(Y_t)_{t\in\mathbb{N}}$ but can be opened in the process $(X_t)_{t\in\mathbb{N}}$, otherwise the edge $e$ is opened in both processes $(X_t)_{t\in\mathbb{N}}$ and $(Y_t)_{t\in\mathbb{N}}$. On the contrary, the two processes behave similarly for the edge closing events since we cannot create a new connexion by closing an edge. The set of the configurations satisfying $\oset T\nlongleftrightarrow B \cset$ is irreducible and the process $(X_t)_{t\in\mathbb{N}}$ is reversible. Therefore, the process $(Y_t)_{t\in\mathbb{N}}$ is the dynamical percolation conditioned to satisfy the event $\oset T\nlongleftrightarrow B \cset$. According to the lemma 1.9 of~\cite{Kelly:2011:RSN:2025239}, the invariant probability measure of $(Y_t)_{t\in \mathbb{N}}$ is $P_\mathcal{D}$, the probability $P_p$ conditioned by the event $\oset T\nlongleftrightarrow B\cset$, i.e., 
$$P_\mathcal{D} (\cdot)= P_p(\cdot \,|\, T\nlongleftrightarrow B) .
$$

Suppose that we start from a configuration $(X_0,Y_0)$ belonging to the set 
$$\mathcal{E} = \oset (\omega_1,\omega_2)\in \{0,1\}^{\mathbb{E}^d\cap \Lambda}\times \{T\nlongleftrightarrow B\}\,:\, \forall e\subset \Lambda\quad \omega_1(e)\geqslant \omega_2(e)\cset.
$$
The set $\mathcal{E}$ is irreducible and aperiodic. In fact, each configuration of $\mathcal{E}$ communicates with the configuration where all edges are closed. The state space $\mathcal{E}$ is finite, therefore the Markov chain $(X_t,Y_t)_{t\in\mathbb{N}}$ admits a unique equilibrium distribution $\mu_p$. We denote by $P_\mu$ the law of the process $(X_t,Y_t)_{t\in\mathbb{N}}$ starting from a random initial configuration $(X_0,Y_0)$ with distribution $\mu_p$.
We now present a definition of the interface between $T$ and $B$ based on the previous coupling.
\begin{defi}The interface at time $t$ between $T$ and $B$, denoted by $\mathcal{I}_t$, is the set of the edges in $\Lambda$ that differ in the configurations $X_t$ and $Y_t$, i.e.,
$$ \mathcal{I}_t = \big\{\,e\subset\Lambda : X_t(e) \neq Y_t(e)\,\big\}.
$$
\end{defi}
\noindent The edges of $\mathcal{I}_t$ are open in $X_t$ but closed in $Y_t$ and the configuration $X_t$ is above the configuration  $Y_t$. We define next the set $\mathcal{P}_t$ of the pivotal edges for the event $\{T\nlongleftrightarrow B\}$ in the configuration $Y_t$.
\begin{defi}
The set $\mathcal{P}_t$ of the pivotal edges in $Y_t$ is the collection of the edges in $\Lambda$ whose opening would create a connection between $T$ and $B$, i.e.,
$$\mathcal{P}_t = \big\{ \,e\subset\Lambda : T\overset{Y_t^e}{\longleftrightarrow} B\,\big\}.
$$
\end{defi}
We define finally the set $\mathcal{C}_t$ of the cuts in $Y_t$.
\begin{defi}
The set $\mathcal{C}_t$ of the cuts in $Y_t$ is the collection of the cuts in $ \Lambda$ at time $t$.
\end{defi}

\section{The isolated pivotal edges}\label{new.iso}
In this section, we will show the proposition~\ref{new.distpivot}. We first investigate the structure of the set of the cuts. In a configuration $\omega$ realizing the event $\{T\nlongleftrightarrow B\}$, we will identify two separating sets $S^+$ and $S^-$. We construct $S^+$ by considering the open cluster 
$$O(T) = \oset x\in \mathbb{Z}^d\cap \Lambda\,:\, x\overset{\omega}{\longleftrightarrow} T\cset.$$ 
We consider the set $O(T)^c = \mathbb{Z}^d\setminus O(T)$. As $\mathbb{Z}^d\setminus \Lambda$ is $*$-connected, there exists only finitely many $*$-connected components of $O(T)^c$ and exactly one of them is of infinite size. We denote these components by $G,H_1,\dots,H_k$ where $G$ is the unique infinite component. We set 
$$O'(T) = O(T) \cup H_1\cup \dots\cup H_k.
$$
The set $O'(T)$ is $*$-connected and has no holes. For a $*$-connected set $A\subset \mathbb{Z}^d$, we define the external boundary of $A$, denoted by $\partial^{ext}A$, as 
$$ \partial^{ext}A = \oset \{x,y\}\in \mathbb{E}^d\,:\, x\in A, y\notin A\cset. 
$$
We then define $S^+$ as the subset of $\partial^{ext}O'(T)$ consisting of the edges of $\partial^{ext}O'(T)$ which are included in $\Lambda$.
In a similar way, we define $S^-$ by replacing $T$ by $B$ in the previous construction. Each of the two sets contains a cut. An illustration of these two separating sets can be found in the figure~\ref{new.fig:cuts}.

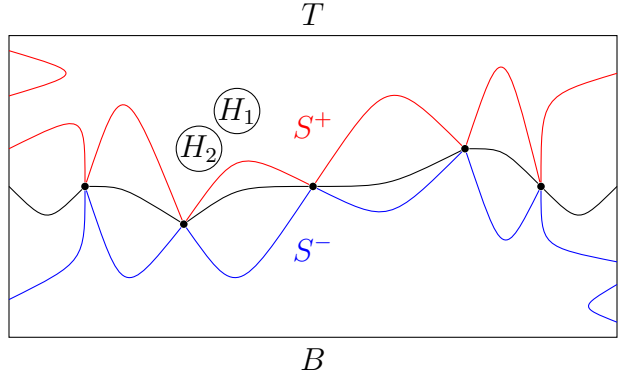
\begin{figure}[ht]
\center
\begin{tikzpicture}
\draw (0,0) rectangle (8,4);
\node[below] at (4,0) {$B$}; 
\node[above] at (4,4) {$T$}; 
\node[fill,circle,inner sep = 1pt] at (1,2) (a) {};
\node[fill,circle,inner sep = 1pt] at (2.3,1.5) (b) {};
\node[fill,circle,inner sep = 1pt] at (4,2) (c) {};
\node[fill,circle,inner sep = 1pt] at (6,2.5) (d) {};
\node[fill,circle,inner sep = 1pt] at (7,2) (e) {};
\draw[red] (0,2.5) .. controls (1,3) .. (a) .. controls (1.5,3.5) .. (b) .. controls (3,2.5) .. (c) .. controls (5,3.5) .. (d) .. controls (6.5,4) .. (e) .. controls (7,3.2) .. (8,3.5);
\draw (2.5,2.5) circle (0.3);
\node at (2.5,2.5) {$H_2$};
\draw (3,3) circle (0.3);
\node at (3,3) {$H_1$};
\draw[red] (0,3.2) .. controls (1,3.5) .. (0,3.8);
\draw[blue] (0,0.5) .. controls (1,1) .. (a) .. controls (1.5,0.5) .. (b) .. controls (3,0.5) .. (c) .. controls (5,1.5) .. (d) .. controls (6.5,1) .. (e) .. controls (7,1.2) .. (8,1);
\draw (0,2) .. controls (0.5,1.5) .. (a) .. controls (1.5,2) .. (b) .. controls (3,2) .. (c) .. controls (5,2) .. (d) .. controls (6.5,2.5) .. (e) .. controls (7.5,1.5) .. (8,2);
\draw[blue] (8,0.7) .. controls (7.5,0.4) .. (8,0.2);
\node[above,red] at (4,2.5)  {$S^+$};
\node[below,blue] at (4,1.5) {$S^-$};
\end{tikzpicture}
\caption{The sets $S^+$ (red) and the set $S^-$ (blue).}\label{new.fig:cuts}
\end{figure}

\begin{lem}\label{new.connect}
The sets $\partial^{ext} O'(T)$ and $\partial^{ext}O'(B)$ are $*$-connected.
\end{lem}
\noindent This result is a direct consequence of the first point in lemma 2.1 in~\cite{Pisztora1996Surface}. We also mention the lemma 2.23 in~\cite{MR876084} for a similar result on the set of vertices and a shorter argument presented in \cite{0711.1713}.  
We explain next the relation between the sets $S^+$, $S^-$ and $ \mathcal{P}$.
\begin{lem} \label{new.capcut}The set $\mathcal{P}$ of the pivotal edges is the intersection between $S^+$ and $S^-$.
\end{lem} 
\begin{proof}
We have the inclusion $\mathcal{P}\subset S^+\cap S^-$ since all the pivotal edges are in all the cuts. Both $S^+$ and $S^-$ contain a cut. We consider next an edge $e$ in $S^+\cap S^-$. Since $S^+$ consists of the boundary edges of $O(T)$, there is an open path between $T$ and $e$. The same result holds for $S^-$. Therefore, there is a path between $T$ and $B$ whose edges other than $e$ are open. By opening $e$, we realise the event $\{T\longleftrightarrow B\}$. In other words, the edge $e$ is included in $\mathcal{P}$. We conclude that $S^+\cap S^- \subset \mathcal{P}$.
\end{proof}

We also need a combinatoric result on the $*$-connectedness in dimension $d$.
\begin{lem}
\label{new.count}In the $d$-dimensional lattice, the number of $*$-neighbours of an edge $e$ is 
$$\alpha(d) = 3^d+4(d-1)3^{d-2}-1.$$
\end{lem}
\begin{proof}
An $*$-neighbour edge $f$ of $e$ is either parallel to $e$ or belongs to the $(d-1)$-cube centred at a vertex of $e$ and of side-length $2$ perpendicular to $e$. For the edges that are parallel to $e$, the distance between their centres is $1$ and there are $3^d-1$ such edges. A $(d-1)$-cube of side-length $2$ has $2(d-1)3^{d-2}$ edges. Hence there are $
3^d+4(d-1)3^{d-2}-1$
$*$-neighbours of $e$.
\end{proof}

We now prove the proposition~\ref{new.distpivot}. The main idea of the proof is to observe a long closed path outside of a cut whenever a pivotal edge is isolated. We use then the BK inequality and we conclude with the help of classical arguments of exponential decay.

\begin{proof}[Proof of proposition~\ref{new.distpivot}]
Since there is a pivotal edge which is at distance more than $1$ from the others, there is a cut which contains at least one non pivotal edge. By lemma~\ref{new.capcut}, this cut is not included in $S^-\cap S^+$, thus there are at least two distinct cuts in the configuration.
Let $e$ be an edge of $\mathcal{P}$ which is at distance at least $\kappa c\ln|\Lambda|$ from $\Lambda^c\cup\mathcal{P} \setminus \{e\}$. Let $e'$ be the pivotal edge which is nearest to $e$ or one of them if there are several.  By lemma~\ref{new.connect}, there is a closed $*$-path included in $\partial^{ext}O(T)$ between $e$ and $e'$. This path might exit from the box $\Lambda$, since $\partial^{ext}O(T)$ is defined as the external boundary of $O'(T)$, where $O'(T)$ is seen as a subset of $\mathbb{Z}^d$, not of $\Lambda$.
However, since $e$ is at distance at least $\kappa c\ln|\Lambda|$ from $\mathcal{P}\setminus \{e\}$ and from $\Lambda^c$, the initial portion of the closed $*$-path from its origin until it has travelled a distance $\kappa c\ln|\Lambda|$ is inside the box $\Lambda$, it consists of closed edges which are not pivotal, and therefore, by lemma \ref{new.capcut}, it is also disjoint from the set $S^-$.
Let us denote by $\mathcal{E}(e)$ the event:
$$ \mathcal{E}(e) = \left\lbrace \kern-1.5pt\begin{array}{c}\text{there exists a closed }* \text{-path starting at one }*\text{-neighbour of }e\\
\text{which travels a distance at least }\kappa c\ln|\Lambda|-2d\end{array} \kern-1.5pt\right\rbrace.
$$
From the previous discussion, we conclude that 
$$\oset e\in \mathcal{P}, d(e,\Lambda^c\cup \mathcal{P}\setminus \{e\})\geqslant \kappa c\ln|\Lambda|\cset\cap\oset T\nlongleftrightarrow B\cset  \quad \subset \quad\mathcal{E}(e)\circ \oset T\nlongleftrightarrow B\cset,
$$
where $\circ$ means the disjoint occurrence. Therefore, we have the following inequality:
$$
P_\mathcal{D} \Big(e\in \mathcal{P}, d(e, \Lambda^c\cup\mathcal{P}\setminus \{e\})\geqslant \kappa c\ln|\Lambda|\big) \leqslant P_\mathcal{D}\Big( \mathcal{E}(e)\circ \{T\nlongleftrightarrow B\}\Big).
$$
By the definition of $P_\mathcal{D}$, we have
$$P_\mathcal{D}\Big( \mathcal{E}(e)\circ \{T\nlongleftrightarrow B\}\Big)= \frac{P_p\Big( \mathcal{E}(e)\circ \{T\nlongleftrightarrow B\}\Big)}{P_p\Big(T\nlongleftrightarrow B\Big)}.
$$
Note that the event $\mathcal{E}(e)$ and $\oset T\nlongleftrightarrow B\cset$ are both decreasing.
Applying the BK inequality (see~\cite{grimmett1999percolation}), we get 
$$P_\mathcal{D} \Big(e\in \mathcal{P}, d(e, \Lambda^c\cup\mathcal{P}\setminus \{e\})\geqslant \kappa c\ln|\Lambda|\Big)\leqslant P_p\big(\mathcal{E}(e)\big).
$$
The closed $*$-path in the event $\mathcal{E}(e)$ starts at a neighbour of $e$ and travels a distance at least $\kappa c\ln |\Lambda|-2d$. By this, we mean that there is an Euclidean distance at least $\kappa c\ln|\Lambda|-2d$ from one endpoint of the first edge of the path to one endpoint of the last edge of the path. The distance between the centres of two $*$-neighbouring edges is at most $d$, therefore the number of edges in such a path is at least $$
\frac{1}{d}(\kappa c\ln|\Lambda| -2d -1).$$
We assume that $|\Lambda|\geqslant 3^{6d}$ and we choose $\kappa>1$, whence, for $c\geqslant 1$, $$\kappa c\ln|\Lambda| -2d-1 \geqslant \frac{\kappa c}{2}\ln|\Lambda|.
$$
Hence $$
P_p\big(\mathcal{E}(e)\big)\leqslant (1-p)^{ \displaystyle\frac{\kappa c}{2d}\ln|\Lambda|}\alpha(d)^{\displaystyle\frac{\kappa c}{2d}\ln|\Lambda|}.
$$
We then sum the probability over all the edges $e$ in $\Lambda$. We obtain
$$P_\mathcal{D}\Big(\exists e\in \mathcal{P}, d(e, \Lambda^c\cup\mathcal{P}\setminus \{e\})\geqslant \kappa c\ln|\Lambda|\Big)\leqslant d|\Lambda|^{\displaystyle 1+\frac{\kappa c}{2d}\ln\Big((1-p)\alpha(d)\Big)}.
$$
We choose $\tilde{p}<1$ such that $(1-\tilde{p})\alpha(d)<1$.  There exists a $\kappa>0$ such that, for any $p\geqslant \tilde{p}$ and any $c>0$, we have
$$d|\Lambda|^{\displaystyle 1+\frac{\kappa c}{2d}\ln\Big((1-p)\alpha(d)\Big)}\leqslant \frac{1}{|\Lambda|^c},
$$
and we obtain the desired inequality.
\end{proof}
We state now a corollary of the proposition~\ref{new.distpivot} which controls the distance between any cut present in the configuration and the set $\mathcal{P}$. 
\begin{cor}  There exists $\tilde{p}<1$ and $\kappa>1$ such that, for $p\geqslant \tilde{p}$, for any constant $c\geqslant 1$ and any box $\Lambda$ satisfying $|\Lambda|>3^{6d}$, the following inequality holds:
$$P_\mathcal{D} \Big(\exists C\in \mathcal{C},\exists e\in C, d(e,\mathcal{P}\cup\Lambda^c)\geqslant \kappa c\ln|\Lambda| \Big)\leqslant \frac{1}{|\Lambda|^c}.
$$\label{new.distcut}
\begin{proof}
Let $C$ be a cut and let $e$ be an edge of $C$ such that $d(e,\mathcal{P}\cup\Lambda^c)\geqslant \kappa c\ln |\Lambda|$. There exists a closed $*$-path included in $C$ which connects $e$ to a pivotal edge $f$. Within a distance less than $\kappa c\ln|\Lambda|$ from $e$, there is no pivotal edge. By stopping the path at the first pivotal edge that it encounters or at the first edge intersecting the boundary of $\Lambda$, we obtain a path $(e_1,\dots,e_n)$ without pivotal edge. Suppose that this path encounters the set $S^+$ or the set $S^-$. Let $e_j$ be the first edge of the path which is in $S^+\cup S^-$. By lemma~\ref{new.capcut}, the edge $e_j$ doesn't belong to $S^+\cap S^-$. Without loss of generality, we can suppose that $e_j\in S^+\setminus S^-$. We concatenate $(e_1,\dots,e_j)$ and a closed path in $\partial^{ext}O'(T)$ from $e_j$ to a pivotal edge or to an edge on the boundary of $\Lambda$. We obtain a closed path disjoint from $S^-$. We reuse the same techniques as in the proof of~\ref{new.distpivot} and we obtain the desired result.
\end{proof}
\end{cor}
We shall also study the case where there is no pivotal edge in a configuration.
\begin{prop} \label{new.nopivot}There exists a constant $\tilde{p}<1$, such that, 
$$ \forall p\geqslant \tilde{p}\quad \forall \Lambda\qquad P_\mathcal{D}\left(\mathcal{P} = \emptyset\right)\leqslant d|\Lambda|\exp\left(-D\right),
$$
where $D$ is the diameter of $T$ (or $B$).
\end{prop}
\begin{proof}
Suppose that $\mathcal{P}$ is empty. By lemma~\ref{new.capcut}, the set $S^+$ and the set $S^-$ are then disjoint. Each of them contains a cut. Therefore, there are two disjoint closed $*$-paths travelling a distance at least $|\Lambda|^{1/d}$. By the same reasoning as in the proof of proposition~\ref{new.distpivot}, the $P_\mathcal{D}$ probability of this event can be bounded by 
$$ P_p\left(\exists \gamma \text{ closed path}\subset \Lambda, \gamma \text{ travels a distance at least }D\right).
$$
Since there are at least $D/d$ edges in such a path $\gamma$, this probability is less than 
$$ d|\Lambda|\big(\alpha(d)(1-p)\big)^{\displaystyle\frac{D}{d}}.
$$
There exists $\tilde{p}<1$ such that, for all $\Lambda$, we have
$$\forall p\geqslant \tilde{p}\qquad d|\Lambda|\big(\alpha(d)(1-p)\big)^{\displaystyle\frac{D}{d}}\leqslant d|\Lambda|\exp\left(-D\right).
$$
This yields the desired inequality.
\end{proof}

\section{Speed of the cuts}\label{new.speed}
We state now the crucial proposition which gives a control on the speed of the cuts. 
\begin{prop} \label{new.vconst}
There exists $\tilde{p}<1$, such that for $p\geqslant \tilde{p}$, for any $\ell\geqslant 2$, $t\in \mathbb{N}$, $s\in \oset 0,\dots,|\Lambda|\cset$ and any edge $e\subset \Lambda$ at distance more than $\ell$ from $\Lambda^c$, 
$$
P_\mu\left(\begin{array}{c|c}\begin{array}{c} e\in \mathcal{P}_{t+s}\\\forall r\in [t,t+s] \quad\mathcal{P}_r \neq \emptyset\end{array} &\begin{array}{c} \exists c_t \in \mathcal{C}_t,\, d(e,c_t)\geqslant\ell
\end{array} \end{array}\right)\leqslant \exp(-\ell).
$$
\end{prop}
\noindent To prove this result, we will construct a STP associated to the movement of the pivotal edges and then show that the probability to have such a long STP decreases exponentially fast as the length of the path grows. 
\subsection{Construction of the STP}
We start by defining some properties of a STP. In the rest of the paper, unless otherwise specified, all the closed paths (and the closed STPs) are defined with the relation $*$ and the open paths (and the open STPs) are defined with the usual relation.
\begin{defi} A STP $(e_1,t_1),\dots,(e_n,t_n)$ is increasing (respectively decreasing) if
$$ t_1\leqslant \dots \leqslant t_n \,(\text{resp. } t_1\geqslant\dots \geqslant t_n).
$$
If a STP is increasing or decreasing, we say that it is monotone.
\end{defi}
\begin{figure}[ht]
\centering{
\resizebox{60mm}{!}{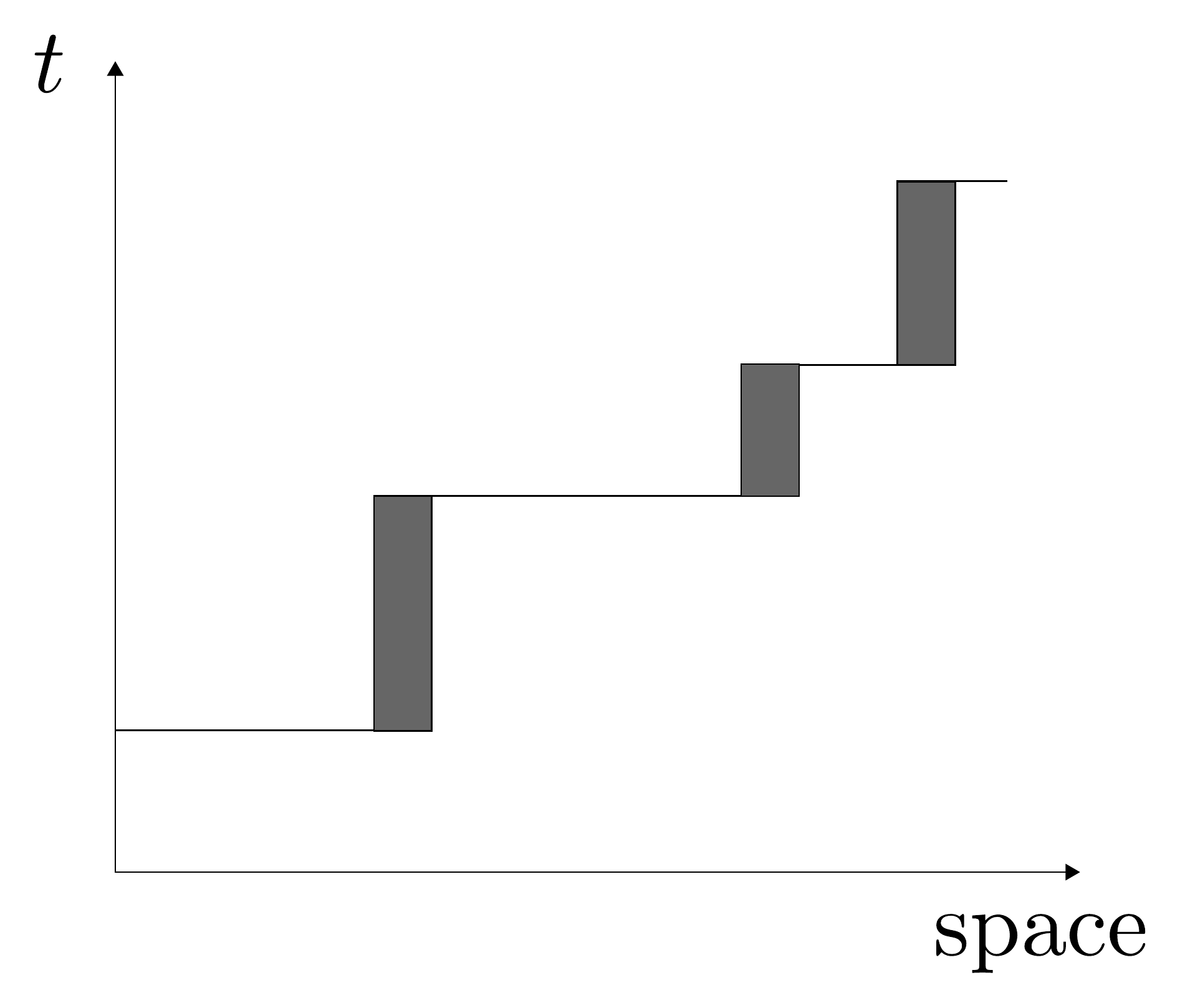}
\caption{An increasing STP with its time change intervals in gray}
\label{new.fig:stc_c}
}
\end{figure}

\begin{defi} A closed STP $(e_1,t_1),\dots,(e_n,t_n)$ in $X$ (respectively $Y$) is called simple if each edge is visited only once or it is opened at least once between any two consecutive visits, i.e., for any $i,j$ in $\oset 1,\dots, n\cset$ such that $|i-j|\neq 1$,
$$ (e_i = e_j\quad t_i<t_j)\quad \Longrightarrow \quad \exists s\in ]t_i,t_j]\quad X_{s}(e_i) = 1\quad  (\text{resp. }Y_{s}(e_i) = 1).
$$
\end{defi}
\begin{figure}[ht]
\centering{
\resizebox{60mm}{!}{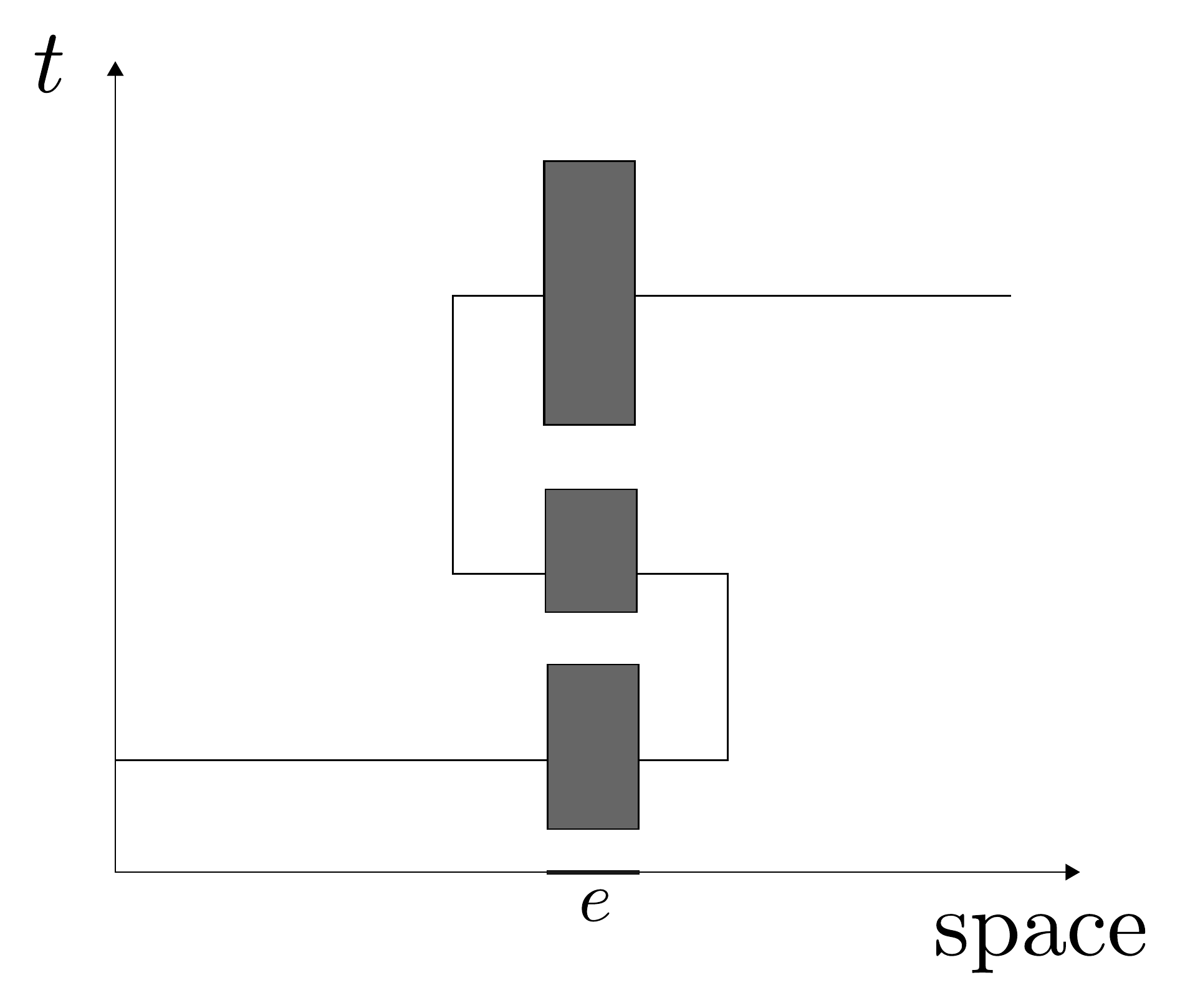}
\caption{A simple STP, intervals of closure of the edge $e$ in gray}
\label{new.fig:stc_s}
}
\end{figure}
\noindent We show next that two pivotal edges occurring at different times are connected through a monotone simple STP closed in $Y$. 
\begin{prop}\label{new.stc} Let $s$ and $t$ be two times such that $s<t$. We suppose that $\mathcal{P}_r$ is not empty for all $r\in [s,t]$. Let $f\in \mathcal{P}_s$ and $e\in \mathcal{P}_t$. Then there exists a decreasing simple STP $\gamma$ closed in $Y$ from $(e,t)$ to $(f,s)$ or a decreasing simple STP closed in $Y$ from $(e,t)$ to $(g,\alpha)$ where $g$ is an edge meeting the boundary $\partial\Lambda$ of $\Lambda$ and $\alpha\in [s,t]$.
\end{prop}
\begin{proof}
By lemma~\ref{new.connect}, the edges of $\mathcal{P}_t$ are connected by a $*$-path which might possibly exit from $\Lambda$, but whose edges included in $\Lambda$ are closed in $Y_t$. We consider the function $\theta(t)$ giving the time when the oldest edge of $\mathcal{P}_t$ appeared, i.e.,
$$ \theta(t) = \min_{\varepsilon \in \mathcal{P}_t}\min\big\{ \,r\leqslant t: \varepsilon \in \mathcal{P}_r, \varepsilon \notin \mathcal{P}_{r-1}, \forall \alpha\in[r,t] ,\varepsilon\in \mathcal{P}_\alpha \big\}.
$$
We denote by $e_1$ one of the edges realizing the minimum $\theta(t)$. We claim that $\theta(t)< t$. Indeed, suppose first that an edge closes at time $t$. Then a pivotal edge cannot be created at time $t$ and all the pivotal edges present at time $t$ were also pivotal at time $t-1$. Therefore $\theta(t)\leqslant t-1< t$. Suppose next that an edge opens at time $t$. Let us consider an edge $\varepsilon$ of $\mathcal{P}_{t-1}$, which is assumed to be not empty. At time $t-1$, there is one open path which connects $\varepsilon$ to $T$ and another one which connects $\varepsilon$ to $B$. Since one edge opens at time $t$, these two paths remain open at time $t$. Therefore $\varepsilon$ is still pivotal at time $t$. We have thus $\mathcal{P}_{t-1}\subset\mathcal{P}_t$ and it follows that $\theta(t)\leqslant t-1 <t$. We have proved that $\theta(t)< t$.

If $\theta(t)\leqslant s$, we consider the STP obtained by connecting the path between $ (e,t),(e_1,t)$ and the path between $(e_1,s),(f,s)$ with a time change from $t$ to $s$ on the edge $e_1$. If this STP does not encounter $\partial\Lambda$ then it answers the question. If it encounters $\partial\Lambda$, then we stop the STP at the first edge intersecting $\partial\Lambda$, we obtain a STP satisfying the second condition of the proposition. Suppose now that $\theta(t)> s$. We consider the edge $e_1$ at time $\theta(t)$. By construction, the edge $e_1$ belongs to $\mathcal{P}_{\theta(t)}$. Moreover, using lemma~\ref{new.connect}, $(e_1,\theta(t))$ is connected to $(e,t)$ by a STP consisting of a closed path at time $t$ and a time change from $t$ to $\theta(t)$ on the edge $e_1$.
We take $(e_1,\theta(t))$ as the new starting point. We repeat the procedure above and we obtain a sequence of times $(\theta(t),\theta(\theta(t)),\dots,\theta^{(i)}(t),\dots)$ by defining iteratively
$$
\theta^{i+1}(t) = \min_{\varepsilon\in \mathcal{P}_{\theta^{i}(t)}} \oset r \leqslant \theta^i(t): \, \varepsilon\in \mathcal{P}_r,\,\varepsilon\notin\mathcal{P}_{r-1},\forall \alpha \in [r,\theta^i(t)],\varepsilon\in \mathcal{P}_\alpha\cset.
$$
For each index $i$, we choose an edge $e_i\in\mathcal{P}_{\theta^{i-1}(t)}$ which becomes pivotal at time $\smash{\theta^{(i)}(t)}$. From the argument above, we obtain a strictly decreasing sequence $$t> \theta(t)>\dots> \theta^i(t).$$ Therefore, there exists an index $k$ such that $$\theta^{k+1}(t)\leqslant s<\theta^k(t).$$ 
For $i\in \{0,\dots,k-1\}$, the edge-time $(e_i,\theta^i(t))$ is connected to $(e_{i+1},\theta^{i+1}(t))$ by a decreasing STP $\gamma_i$ which is closed in $Y$. By concatenating these STPs, we obtain a decreasing STP between $(e,t)$ and $(e_k,\theta^{k}(t))$. At time $\theta^k(t)$, there exists also a closed path $\rho$ between $e_k$ and $e_{k+1}$. We stop the time change at $s$ on the edge $e_{k+1}$ in order to arrive at an edge of $\mathcal{P}_s$. By lemma~\ref{new.connect}, there is a closed path $\rho$ between $e_{k+1}$ and $f$ at time $s$. Therefore, the STP
$$(e,t),\gamma_0,(e_1,\theta(t)),\gamma_1,\dots,\gamma_{k-1},(e_k,\theta^k(t)),\rho,(e_{k+1},\theta^k(t)),(e_{k+1},s),\rho,(f,s)
$$
is decreasing, closed in $Y$ and it connects $(e,t)$ and $(f,s)$. If this STP exits the box $\Lambda$, then the initial portion starting from $e$ until the first edge intersecting $\partial\Lambda$ satisfies the second condition of the proposition.

In order to obtain a STP which is simple in $Y$, we consider the following iterative procedure to modify a path. Let us denote by $(e_i,t_i)_{0\leqslant i\leqslant N}$ the STP obtained previously. Starting with the edge $e_0$, we examine the rest of the edges one by one. Let $i\in \oset 0,\dots,N\cset$. Suppose that the edges $e_0,\dots,e_{i-1}$ have been examined and let us focus on $e_i$. We encounter three cases:
\begin{itemize}[leftmargin = 0.4cm]
\item For every index $j\in \{i+1,\dots,N\}$, we have $e_j\neq e_i$. Then, we don't modify anything and we start examining the edge $e_{i+1}$.
\item There is an index $j\in\{i+1,\dots,N\}$ such that $e_i = e_j$, but for the first index $k>i+1$ such that $e_i = e_k$, there is a time $\alpha\in ]t_k,t_i[$ when $Y_\alpha(e_i) = 1$. Then we don't modify anything and we start examining the next edge $e_{i+1}$.
\item There is an index $j\in\{i+1,\dots,N\}$ such that $e_i = e_j$ and for the first index $k> i+1$ such that $e_i = e_k$, we have $Y_\alpha(e_i) = 0$ for all $\alpha\in ]t_k,t_i[$. In this case, we remove all the time-edges whose indices are strictly between $i$ and $k$. We then have a simple time change between $t_i$ and $t_k$ on the edge $e_i$. We continue the procedure from the index $e_k$.
\end{itemize}
The STP becomes strictly shorter after every modification, and the procedure will end after a finite number of modifications. We obtain in the end a simple path in $Y$. Since the procedure doesn't change the order of the times $t_i$, we still have a decreasing path.
\end{proof}

\subsection{The BK inequality applied to a STP}
Before embarking in technicalities, let us discuss the differences between the processes $(X_t)_{t\in\mathbb{N}}$ and $(Y_t)_{t\in\mathbb{N}}$. Let $(e,t)$ be a closed time-edge in $Y$. Since there is no constraint in the process $X$, the edge $e$ can be open in the configuration $X_t$. If the edge $e$ is open in $X_t$, then it belongs to $\mathcal{I}_t$. Now let us consider a time $t$ for which $E_t = e$ and $B_t = 0$. Closing an edge doesn't create an open path between $T$ and $B$, thus the edge $e$ will be closed in both $X_t$ and $Y_t$. On the contrary, for a time $t$ such that $B_t = 1$ and $E_t = e$ is pivotal at time $t-1$, the edge $e$ can be opened in the process $(X_t)_{t\in\mathbb{N}}$ but it remains closed in the process $(Y_t)_{t\in\mathbb{N}}$. Now let us consider the STP constructed in proposition~\ref{new.stc}. Since the STP is closed in $Y$, each edge $e$ visited by the path is either closed at time $s$ or there is a time $r\in\, ]s,t]$ when $E_r = e$ and $B_r = 0$. In fact, since the STP is simple, then each edge is reopened and closed between two successive visits of the STP. Our first goal is to introduce the necessary notation in order to keep track of all the closing events implied by the STP.

We shall define the space projection of a STP. Given $k\in\mathbb{N}^*$ and a sequence $\Gamma = (e_i)_{1\leqslant i \leqslant k}$ of edges, we say that it has length $k$, which we denote by $\mathrm{length}(\Gamma)= k$, and we define its support
$$\mathrm{support}(\Gamma) = \oset e\subset \Lambda\,:\, \exists i\in \{1,\dots,k\}\quad e_i = e\cset.
$$
Let $\gamma = (e_i,t_i)_{1\leqslant i\leqslant n}$ be a simple STP, the space projection of $\gamma$ is obtained by removing one edge in every time change in the sequence $(e_i)_{1\leqslant i \leqslant n}$. More precisely, let $m$ be the number of time changes in $\gamma$. We define the function $\phi:\{1,\dots,n-m\}\rightarrow \mathbb{N}$ by setting $\phi(1) = 1$ and 
$$
\forall i\in \{\, 1,\dots, n-m\,\}\quad \phi(i+1) = \left\lbrace \begin{array}{ccc}
\phi(i)+1 &\text{if}& e_{\phi(i)} \neq e_{\phi(i)+1}\\
\phi(i)+2 &\text{if}& e_{\phi(i)} = e_{\phi(i)+1}
\end{array} \right..
$$
The sequence $(e_{\phi(i)})_{1\leqslant i \leqslant n-m}$ is called the space projection of $\gamma$, denoted by 
$\mathrm{Space}(\gamma)$. We say that $\mathrm{length}(\mathrm{Space}(\gamma))$ is the length of the STP $ \gamma$, denoted also by $\mathrm{length}(\gamma)$.
We shall distinguish $\mathrm{Space}(\gamma)$ from the support of $\gamma$, denoted by $\mathrm{support}(\gamma)$, which we define as:
$$\mathrm{support}(\gamma) = \mathrm{support }(\mathrm{Space}(\gamma)).
$$ 
We say that a sequence of edges $\Gamma = (e_i)_{1\leqslant i \leqslant k}$ is visitable if there exists a STP $\gamma$ such that $\mathrm{Space}(\gamma) = \Gamma$.

We prove next a key inequality to control the number of $\mathrm{closing}$ events along a simple STP.
\begin{prop}\label{new.BKFermeutre}
Let $\Gamma$ be a visitable sequence of edges and $[s,t]$ a time interval. For any $ k\in \oset 0,\dots, |\mathrm{support}(\Gamma)|\cset$ and any percolation configuration $y$ such that there are exactly $k$ closed edges in $\mathrm{support}(\Gamma)$,
we have the following inequality:
$$P\left(\kern-15pt\begin{array}{c|c}\begin{array}{c}\exists \gamma\text{ decreasing closed}\\\text{ simple STP in }Y 
\text{during }[s,t]\\ \text{such that }\mathrm{Space}(\gamma) = \Gamma 
 \end{array}&Y_s = y\kern-5pt\end{array}\right)\leqslant 
\left(\frac{(t-s)(1-p)}{|\Lambda|}\right)^{\mathrm{length}(\Gamma)-k}.
$$
\end{prop}
\begin{proof}
We denote by $n$ the length of $\Gamma$ and $(e_1,\dots,e_n)$ the sequence $\Gamma$.
We consider a STP $\gamma$ such that $\mathrm{Space}(\gamma)=\Gamma$. Since $\gamma$ is closed, all the edges of $\Gamma$ are closed at time $s$ or become closed after $s$. For an edge $e\in \mathrm{support}(\Gamma)$, we denote by $v(e)$ the number of times that $\Gamma$ visits $e$:
$$v(e) = \left|\oset j\in \{1,\dots,n\}\,:\, e_j = e\cset \right|.
$$
Since $\gamma$ is simple, between two consecutive visits, there exists a time when the edge $e$ is open, as illustrated in the figure~\ref{new.fig:vitesse}.

\begin{figure}[ht]
\centering{
\resizebox{100mm}{!}{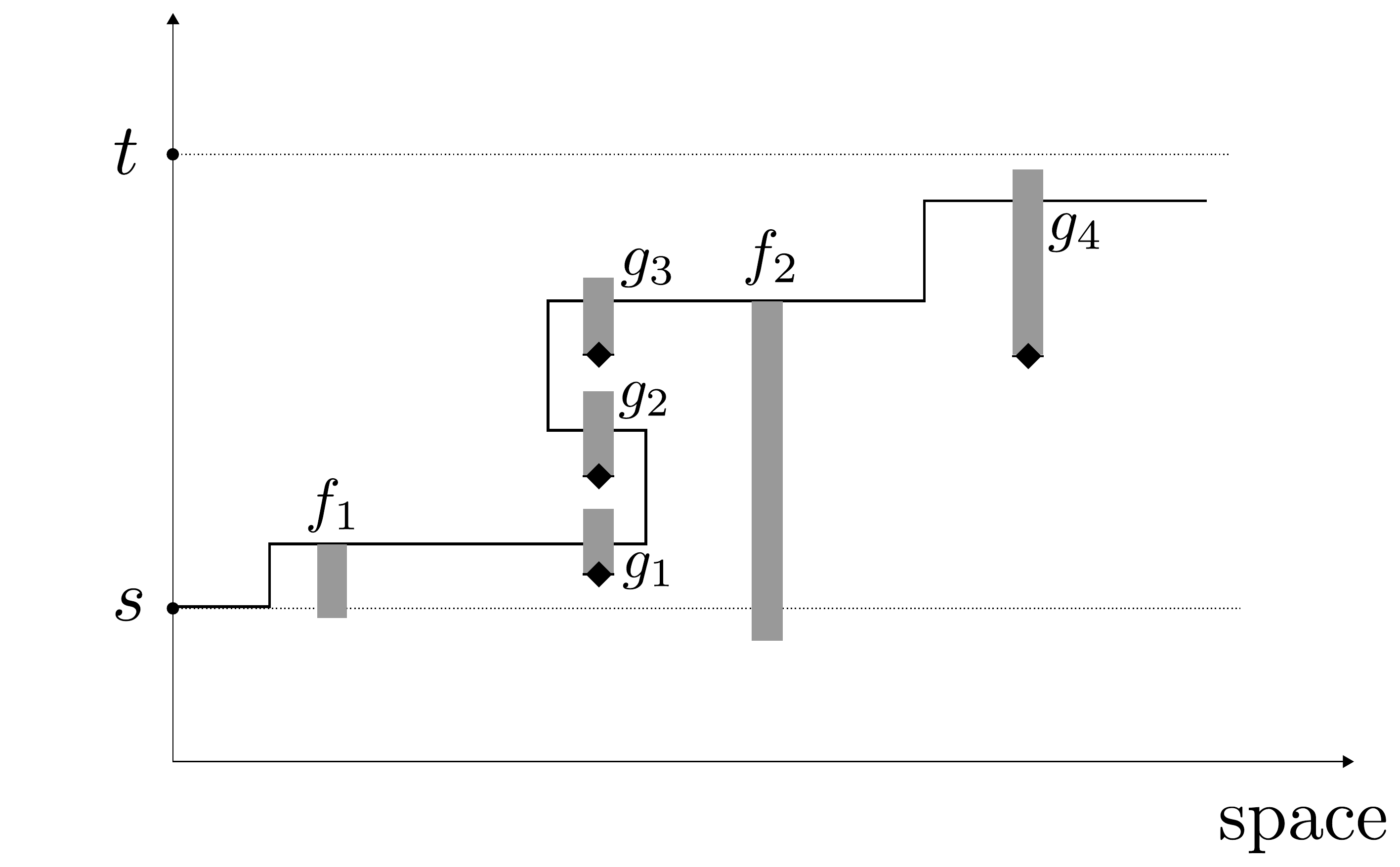}
\caption{The edges $f_1,f_2$ are closed at time $s$. The edges $g_1,g_2,g_3$ and $g_4$ closes after $s$. We see that $g_1 = g_2 = g_3$ and the simplicity of the path implies that the edge opens and closes between two consecutive visits.}
\label{new.fig:vitesse}
}
\end{figure} 
\noindent For each edge $e$ visited by $\gamma$, we distinguish two cases according to the configuration $Y_s$. If $Y_s(e) = 1$, there is a time between $s$ and the first visit when $e$ becomes closed and the edge $e$ closes at least $v(e)$ times during the time interval $]s,t]$. If $Y_s(e) = 0$, then, between the time $s$ and the first visit of $e$, it can remain closed and the edge $e$ becomes closed at least $v(e)-1$ times during $]s,t]$. Notice that the numbers $v(e)$ depend on the sequence $\Gamma$. The probability in the proposition is therefore less than or equal to 
\begin{equation}\label{new.probferme}
P\left(\begin{array}{c|c}\begin{array}{c}\left\lbrace\begin{array}{c}\forall e\in \mathrm{support}(\Gamma)\quad y(e) = 1 \\ e \text{ closes at least }v(e) \text{ times during } ]s,t]
\end{array}\right\rbrace \bigcap \\
\left\lbrace\begin{array}{c}\forall f\in \mathrm{support}(\Gamma)\quad y(f) = 0 \\ f \text{ closes at least }v(f)-1 \text{ times during } ]s,t]
\end{array}\right\rbrace \end{array}
&Y_s = y\end{array}\right).
\end{equation}
Notice that for any edge $e$ such that $y(e) = 1$ (respectively $y(e) = 0$), the event 
$$\oset e \text{ closes at least }v(e) \,(\text{resp. }v(e)-1)\text{ times during } ]s,t]
\cset
$$ depends on the collection of random variables
$$F(e) = \oset (E_r,B_r)\,:\,s < r\leqslant t, \, E_r = e,\, B_r = 0\cset.$$ Therefore these events are independent of the event $\oset Y_s = y\cset$ which depends on $(X_0,Y_0)$ and $\oset(E_r,B_r)\,:\, r \leqslant s\cset$. The probability in~\eqref{new.probferme} is thus equal to
$$P\left(\begin{array}{c}\left\lbrace\begin{array}{c}\forall e\in \mathrm{support}(\Gamma)\quad y(e) = 1 \\ e \text{ closes at least }v(e) \text{ times during } ]s,t]
\end{array}\right\rbrace \bigcap \\
\left\lbrace\begin{array}{c}\forall f\in \mathrm{support}(\Gamma)\quad y(f) = 0 \\ f \text{ closes at least }v(f)-1 \text{ times during } ]s,t]
\end{array}\right\rbrace \end{array}
\right).
$$
For any edge $e\in \mathrm{support}(\Gamma)$, we define $J(e)$ as the set of indices 
$$J(e) = \oset  s<j\leqslant t\,:\, E_j = e, B_j = 0 \cset.$$ We notice that the sets $\big(J(e),{e\in \mathrm{support}(\Gamma)}\big)$ are pairwise disjoint subsets of $\mathbb{N}^*$. By the BK inequality applied to the random variables $(E_t,B_t)_{t\in\mathbb{N}}$, the probability in~\eqref{new.probferme} is less than 
\begin{multline*}\prod_{e\in \mathrm{support}(\Gamma), y(e)=1}P\big( e \text{ closes at least }v(e) \text{ times during } ]s,t] \big)\\ \times
\prod_{f\in \mathrm{support}(\Gamma), y(f)=0}P\big(f \text{ closes at least }v(f)-1 \text{ times during } ]s,t]
\big).
\end{multline*}
We obtain therefore 
\begin{multline}\label{new.prodferme}P\left(\begin{array}{c|c}\begin{array}{c}\exists \gamma\text{ decreasing closed simple STP}\\\text{ in }Y 
\text{during }[s,t]\quad \mathrm{Space}(\gamma) = \Gamma 
 \end{array}&Y_s = y\end{array}\right)\\ 
\leqslant \prod_{e\in \mathrm{support}(\Gamma), y(e)=1}P\left( e \text{ closes }v(e) \text{ times during } ]s,t] \right)\\ \times
\prod_{f\in \mathrm{support}(\Gamma), y(f)=0}P\left(f \text{ closes }v(f)-1 \text{ times during }]s,t]
\right).
\end{multline}
For any edge $e\in \mathrm{support}(\Gamma)$ and any $m\in \mathbb{N}$, we have
\begin{multline*}P\big( e \text{ closes at least }m \text{ times during } ]s,t] \big) \\
\leqslant P\left(\begin{array}{c}\exists J\subset\oset s+1,\dots,t\cset\quad |J| =m\\ \forall j \in J\qquad E_j = e\quad B_j = 0\end{array}\right)\\
\leqslant \left(\frac{(t-s)(1-p)}{|\Lambda|}\right)^m.
\end{multline*}
We use this inequality in~\eqref{new.prodferme} and we obtain 
\begin{multline*}
P\left(\begin{array}{c|c}\begin{array}{c}\exists \gamma\text{ decreasing closed simple STP}\\\text{ in }Y 
\text{ during }[s,t]\quad \mathrm{Space}(\gamma) = \Gamma 
 \end{array}&Y_s = y\end{array}\right)\\ 
\leqslant \left(\frac{(t-s)(1-p)}{|\Lambda|}\right)^{\displaystyle \sum_{e\in \mathrm{support}(\Gamma), y(e)=1}\kern-20pt v(e)+\sum_{f\in \mathrm{support}(\Gamma), y(f)=0}\kern-20pt v(f)-1}\\ = \left(\frac{(t-s)(1-p)}{|\Lambda|}\right)^{n-k}.
\end{multline*}
This is the desired result.
\end{proof}

\subsection{Proof of proposition~\ref{new.vconst}}
Our goal is to study the speed of a cut during a time interval of size $|\Lambda|$. We start by using the results in the previous section to control the length of a STP far from a cut.
\begin{prop}\label{new.estSTC}
Let $\ell$ be a positive constant, $\Gamma$ be a visitable sequence of edges starting from an edge $e$ such that $\Gamma$ travels a distance less than $\ell$ and $t$ be a time. For $s\in\mathbb{N}$, we have the following inequality:
\begin{multline*}P_\mu\left(\kern-15pt\begin{array}{c|c}\begin{array}{c}\exists \gamma\text{ decreasing closed}\\\text{ simple STP in }Y 
\text{during }[t,t+s]\\ \text{such that }\mathrm{Space}(\gamma) = \Gamma 
 \end{array} \kern-5pt&\begin{array}{c}\exists c_t \in \mathcal{C}_t\\ d(e,c_t)\geqslant\ell\end{array}\kern-5pt\end{array}\right)\\ \leqslant 
\left(\left(1+\frac{s}{|\Lambda|}\right)(1-p)\right)^{\mathrm{length}(\Gamma)}.
\end{multline*}
\end{prop}

\begin{proof}
We start by rewriting the conditional probability in the proposition as
\begin{multline*}
P_\mu\left(\kern-15pt\begin{array}{c|c}\begin{array}{c}\exists \gamma\text{ decreasing closed}\\\text{ simple STP in }Y 
\text{during }[t,t+s]\\ \text{such that }\mathrm{Space}(\gamma) = \Gamma 
 \end{array} \kern-5pt&\exists c_t \in \mathcal{C}_t,\, d(e,c_t)\geqslant\ell\kern-5pt\end{array}\right)= \\
 \frac{P_\mu\left(\kern-5pt\begin{array}{ccc}\left\lbrace\begin{array}{c}\exists \gamma\text{ decreasing closed}\\\text{ simple STP in }Y 
\text{during }[t,t+s]\\ \text{such that }\mathrm{Space}(\gamma) = \Gamma 
 \end{array} \right\rbrace\kern-5pt&\bigcap&\left\lbrace\begin{array}{c}\exists c_t \in \mathcal{C}_t\\ d(e,c_t)\geqslant\ell\end{array}\right\rbrace\kern-5pt\end{array}\right)}{P_\mu\left(\begin{array}{c}\exists c_t \in \mathcal{C}_t\\ d(e,c_t)\geqslant\ell\end{array}\right)}.
\end{multline*}
We denote by $(e_i,t_i)_{1\leqslant i \leqslant N}$ the time-edges of $\gamma$. Let $n$ be the length of $\gamma$. We consider the case where there are exactly $k$ edges of $\mathrm{support}(\gamma)$ that are closed at time $t$. We shall estimate the following probability:
\begin{equation}\label{new.YFermC}
P_\mu\left(\begin{array}{c}\left\lbrace\begin{array}{c}\exists \gamma \text{ simple decreasing STP}\\ \text{closed in }Y\text{ of length }n\\
\gamma \text{ starts at }(e,t+s) \text{ and ends after }t\\\mathrm{Space}(\gamma) = \Gamma\end{array}\right\rbrace\bigcap \\
\left\lbrace\begin{array}{c}
  \exists F\subset \mathrm{support}(\gamma)\quad |F| =k\\
 \forall f\in F\quad Y_t(f) = 0\\
 \forall f\in \mathrm{support}(\gamma)\setminus F\quad Y_t(f) = 1\end{array}\right\rbrace \bigcap\\
 \Big\{
\exists c_t \in \mathcal{C}_t, d(e,c_t) \geqslant \ell
\Big\}\end{array}
\right).
\end{equation}
We consider the set $M(k)$ of the configurations defined as 
$$
M(k) = \left\lbrace \begin{array}{cl}\omega\,:&\begin{array}{l} \exists F\subset \mathrm{support}(\Gamma),\, |F| = k\\
\forall f \in F\quad \omega(f) = 0\\
\forall f \in \mathrm{support}(\Gamma)\setminus F\quad \omega(f) = 1\\
\exists C\in \mathcal{C} \quad d(e,C)\geqslant \ell\end{array}\end{array}\right\rbrace.
$$
The probability in~\eqref{new.YFermC} is bounded from above by 
$$\sum_{y\in M(k)}P_\mu\left( \begin{array}{c|c}\begin{array}{c}\exists \gamma \text{ decreasing simple closed STP},\\ \mathrm{length}(\gamma) = n,\quad \mathrm{Space}(\gamma) = \Gamma,\\ \gamma \text{ starts at }(e,t+s) \text{ and ends after }t\end{array}& Y_t = y\end{array}\right)P_\mu(Y_t = y).
$$
By proposition~\ref{new.BKFermeutre}, for any $y\in M(k)$, we have
\begin{equation}\label{new.pFerm}P_\mu\left( \kern-8 pt\begin{array}{c|c} \begin{array}{c}\exists \gamma \text{ decreasing simple closed STP},\\ \mathrm{length}(\gamma) = n,\quad \mathrm{Space}(\gamma) = \Gamma,\\ \gamma \text{ starts at }(e,t+s) \text{ and ends after }t\end{array}\kern-5 pt& Y_t = y\end{array}\kern-5pt\right)\leqslant\left(\frac{s}{|\Lambda|}(1-p)\right)^{n-k}.
\end{equation}
We compute now the probability $P_\mu\Big(Y_t\in M(k)\Big)$. Notice that 
$$P_\mu(Y_t\in M(k))\leqslant P_\mu\left( \begin{array}{c} \exists F\subset \mathrm{support}(\Gamma) \quad |F| = k\\ \forall f\in F \quad Y_t(f) = 0\\
\exists c_t \in \mathcal{C}_t\quad d(e,c_t) \geqslant \ell\end{array}\right).
$$ 
The event in the last probability depends only on the configuration $Y_t$. Since the initial configuration $(X_0,Y_0)$ is distributed according to $\mu_p$, so is the couple $(X_t,Y_t)$. The configuration $Y_t$ is distributed according to the second marginal distribution $P_\mathcal{D}$. We have therefore
$$P_\mu\left(\kern-2pt \begin{array}{c} \exists F\subset \mathrm{support}(\Gamma) \quad |F| = k\\ \forall f\in F \quad Y_t(f) = 0\\
\exists c_t \in \mathcal{C}_t\quad d(e,c_t) \geqslant \ell\end{array}\right) = P_\mathcal{D}\left(\kern-2pt\begin{array}{c} \exists F\subset \mathrm{support}(\Gamma) \quad |F| = k\\ \forall f\in F \quad f \text{ closed}\\
\exists C \in \mathcal{C}\quad d(e,C) \geqslant \ell\end{array} \right).
$$
By the definition of $P_\mathcal{D}$, we have 
\begin{multline}P_\mathcal{D}\left(\begin{array}{c} \exists F\subset \mathrm{support}(\Gamma) \quad |F| = k\\ \forall f\in F \quad f \text{ closed}\\
\exists C \in \mathcal{C}\quad d(e,C) \geqslant \ell\end{array}\right) \\= P_p\left(\begin{array}{c|c}\begin{array}{c} \exists F\subset \mathrm{support}(\Gamma) \quad |F| = k\\ \forall f\in F \quad f \text{ closed}\\
\exists C \in \mathcal{C}\quad d(e,C) \geqslant \ell\end{array}&T\nlongleftrightarrow B\end{array}\right) \\
= \frac{P_p\left(\left\lbrace \kern-5pt\begin{array}{c} \exists F\subset \mathrm{support}(\Gamma) \quad |F| = k\\ \forall f\in F \quad f \text{ closed}\end{array}\kern-5pt\right\rbrace \bigcap\left\lbrace\kern-4pt\begin{array}{c}\exists C \in \mathcal{C}\\ d(e,C) \geqslant \ell\kern-4pt\end{array}\right\rbrace\bigcap\oset T\nlongleftrightarrow B\cset\right)}{P_p\Big(T\nlongleftrightarrow B\Big)} .\label{new.pYt}
\end{multline}
The existence of a cut implies the event $\oset T\nlongleftrightarrow B\cset$, thus we can rewrite the numerator as
\begin{equation*}P_p\left(\left\lbrace \begin{array}{c} \exists F\subset \mathrm{support}(\Gamma) \quad |F| = k\\ \forall f\in F \quad f \text{ closed}\end{array}\right\rbrace\bigcap\big\{\exists C\in\mathcal{C}, d(e,C) \geqslant \ell\big\}\right).
\end{equation*}
The edges of $\mathrm{support}(\Gamma)$ are at distance less than $\ell$ from the edge $e$ and the event $\oset \exists C\in\mathcal{C}, d(e,C) \geqslant \ell\cset$ depends on the edges at distance more than $\ell$ from~$e$. It follows that the two events in the previous probability are independent and we have 
\begin{multline*}
P_p\left(\left\lbrace \begin{array}{c} \exists F\subset \mathrm{support}(\Gamma) \quad |F| = k\\ \forall f\in F \quad f \text{ closed}\end{array}\right\rbrace\bigcap\big\{\exists C\in \mathcal{C}, d(e,C) \geqslant \ell\big\}\right) \\= P_p\left(\begin{array}{c} \exists F\subset \mathrm{support}(\Gamma) \quad |F| = k\\ \forall f\in F \quad f \text{ closed}\end{array}\right)P_p\Big(\exists C\in\mathcal{C}, d(e,C)\geqslant \ell\Big).
\end{multline*}
Replacing the numerator in~\eqref{new.pYt} by this product, we obtain
\begin{multline*}
P_p\left(\kern-5pt\begin{array}{c|c}\left\lbrace \begin{array}{c} \kern-5pt\exists F\subset \mathrm{support}(\Gamma) \quad |F| = k\\ \forall f\in F \quad f \text{ closed}\end{array}\kern-5pt\right\rbrace\bigcap\big\{\exists C \in \mathcal{C}, d(e,C) \geqslant \ell\big\}\kern-2pt&\kern-2ptT\nlongleftrightarrow B\end{array}\kern-5pt\right)\\
= \frac{P_p\left(\begin{array}{c} \exists F\subset \mathrm{support}(\Gamma) \quad |F| = k\\ \forall f\in F \quad f \text{ closed}\end{array}\right)P_p\Big(\exists C\in\mathcal{C}, d(e,C)\geqslant \ell\Big)}{P_p\Big(T\nlongleftrightarrow B\Big)}\\
= P_p\left(\begin{array}{c} \exists F\subset \mathrm{support}(\Gamma) \quad |F| = k\\ \forall f\in F \quad f \text{ closed}\end{array}\right)P_\mathcal{D}\Big(\exists C\in\mathcal{C}, d(e,C)\geqslant \ell\Big).
\end{multline*}
Since the edges of $F$ are distinct, we have
$$ P_p\left(\begin{array}{c} \exists F\subset \mathrm{support}(\Gamma) \quad |F| = k\\ \forall f\in F \quad f \text{ closed}\end{array}\right) \leqslant \binom{|\mathrm{support}(\Gamma)|}{k} (1-p)^k.
$$
Combined with~\eqref{new.pFerm}, we obtain that, for $\Gamma$ and $k$ fixed, the probability in~\eqref{new.YFermC}  is bounded from above by
$$ \binom{|\mathrm{support}(\Gamma)|}{k}\left(\frac{s}{|\Lambda|}\right)^{n-k}(1-p)^{n}P_\mathcal{D}\Big(\exists C \in\mathcal{C}, d(e,C) \geqslant \ell\Big).
$$
We sum on the number $k$ from $0$ to $|\mathrm{support}(\Gamma)|$, and we recall that $$|\mathrm{support}(\Gamma)|\leqslant n.$$ We have therefore
\begin{multline*}
P_\mu\left(\kern-5pt\begin{array}{ccc}\left\lbrace\kern-5pt\begin{array}{c}\exists \gamma \text{ simple closed decreasing STP},\\
\mathrm{length}(\gamma) = n,\quad \mathrm{Space}(\gamma)  = \Gamma,\\
\gamma \text{ starts at }(e,t+s)\text{ and ends after } t\end{array}\kern-5pt\right\rbrace\kern-5pt&\bigcap &\kern-5pt\left\lbrace\kern-5pt\begin{array}{c}
\exists c_t \in \mathcal{C}_t \\ d(e,c_t) \geqslant \ell
\kern-5pt\end{array}\right\rbrace\end{array}
\kern-5pt\right)\\
\leqslant \sum_{0\leqslant k \leqslant \ell/2d}\binom{n}{k}\left(\frac{s}{|\Lambda|}\right)^{n-k}(1-p)^{n}P_\mathcal{D}\Big(\exists C \in\mathcal{C}, d(e,C) \geqslant \ell\Big) \\
= \left(\left(1+\frac{s}{|\Lambda|}\right)(1-p) \right)^{n}P_\mathcal{D}\Big(\exists C \in\mathcal{C}, d(e,C) \geqslant \ell\Big).
\end{multline*}
Since the second marginal of $P_\mu$ is $P_\mathcal{D}$, we have 
$$P_\mu\left(\begin{array}{c}\exists c_t \in \mathcal{C}_t\\ d(e,c_t)\geqslant\ell\end{array}\right) = P_\mathcal{D}\Big(\exists C \in\mathcal{C}, d(e,C) \geqslant \ell\Big).
$$
This yields the inequality in the proposition.
\end{proof}
We can now estimate the probability that the set of the pivotal edges moves fast. To do so, we study the STP constructed in proposition~\ref{new.stc} and we use the previous results. 
\begin{proof}[Proof of proposition~\ref{new.vconst}] We rewrite the conditional probability appearing in the proposition as
$$\frac{P_\mu\Big(\{e\in \mathcal{P}_{t+s}\}\,\cap\, \{\forall r\in [t,t+s] \quad\mathcal{P}_r \neq \emptyset\}\,\cap\,\{\exists c_t \in \mathcal{C}_t, d(e,c_t) \geqslant \ell \}\Big)}{P_\mu\Big(\exists c_t \in \mathcal{C}_t, d(e,c_t) \geqslant \ell\Big)}.
$$
Let us estimate the probability in the numerator. By proposition~\ref{new.stc}, there exists a closed decreasing simple STP $\gamma$ inside of $\Lambda$ which connects $(e,t+s)$ to either an edge of $\mathcal{P}_t$ at time $t$ or to an edge intersecting the boundary of $\Lambda$ after time $t$. In both cases, this STP travels a distance at least $\ell$ because all the edges of $\mathcal{P}_t$ are included in the cuts and $e$ is at distance more than $\ell$ from $\Lambda^c$. Since the STP is a $*$-STP, the distance between two consecutive edges is at most $d$, and the length of the STP is at least $(\ell-1)/d$. We denote by $(e_i,t_i)_{1\leqslant i \leqslant N}$ the time-edges of $\gamma$. Let $n$ be the first index such that the STP
$$(e_1,t_1),\dots,(e_n,t_n)
$$
is longer than $\ell/2d$, i.e.,
$$ n = \inf\left\{\,k\geqslant 1 : \mathrm{length}\Big((e_1,t_1),\dots,(e_k,t_k)\Big)\geqslant \frac{\ell}{2d}\,\right\}.
$$
We set $\Gamma = \mathrm{Space}((e_i,t_i)_{1\leqslant i \leqslant n})$ and we denote $\Gamma = (f_i)_{i\in I}$. We have the following inequality:
\begin{multline*}
P_\mu\left(\begin{array}{c|c}\begin{array}{c} e\in \mathcal{P}_{t+s}\\\forall r\in [t,t+s] \quad\mathcal{P}_r \neq \emptyset\end{array} &\begin{array}{c} \exists c_t \in \mathcal{C}_t,\, d(e,c_t)\geqslant\ell
\end{array} \end{array}\right)\\ \leqslant \sum_{\Gamma}P_\mu\left(\begin{array}{c|c}\begin{array}{c}\exists \gamma \text{ simple closed decreasing STP},\\
\mathrm{length}(\gamma) = \ell/2d,\quad \mathrm{Space}(\gamma)  = \Gamma,\\
\gamma \text{ starts at }(e,t+s) \text{ and ends after }t\end{array}\kern-5pt&\begin{array}{c}
\exists c_t \in \mathcal{C}_t, \\ d(e,c_t) \geqslant \ell
\end{array}\end{array}
\right).
\end{multline*}
By proposition \ref{new.estSTC}, each term in the sum is less than 
$$\left(\left(1+\frac{s}{|\Lambda|}\right)(1-p) \right)^{\ell/2d}.
$$
We sum next on all the possible choices of $\Gamma$. By lemma~\ref{new.count}, we have
\begin{multline*}P_\mu\left(\begin{array}{c|c}\begin{array}{c} e\in \mathcal{P}_{t+s}\\\forall r\in [t,t+s] \quad\mathcal{P}_r \neq \emptyset\end{array} &\begin{array}{c} \exists c_t \in \mathcal{C}_t,\, d(e,c_t)\geqslant\ell
\end{array} \end{array}\right)\\ \leqslant \left(\alpha(d)\left(1+\frac{s}{|\Lambda|}\right)(1-p) \right)^{\ell/2d}.
\end{multline*}
There is a constant $\tilde{p}<1$ such that, for all $p\geqslant \tilde{p}$, $s\leqslant |\Lambda|$ and $\ell\geqslant 2$,
$$\left(\alpha(d)\left(1+\frac{s}{|\Lambda|}\right)(1-p) \right)^{\ell/2d}\leqslant e^{-\ell}. 
$$
We have obtained the result stated in the proposition~\ref{new.vconst}.
\end{proof}

\section{The localisation around pivotal edges}
\label{new.th1}
We start by stating a corollary of proposition~\ref{new.vconst}.
Recall at first that the Hausdorff distance between two subsets $A$ and $B$ of $\mathbb{R}^d$, denoted by $d_H(A,B)$, is
$$ d_H(A,B) = \max \,\big\{\,\sup_{a\in A}d(a,B), \sup_{b\in B}d(b,A)\,\big\}.
$$
For $A$ a subset of $\mathbb{R}^d$ and $r>0$, we define the neighbourhood
$$\mathcal{V}(A,r) = \oset x\in \mathbb{R}^d\,:\, d(x,A)<r\cset.$$
The Hausdorff distance is also equal to 
$$ \inf\,\oset r\geqslant 0 \,:\,A\subset\mathcal{V}(B,r),\quad B\subset\mathcal{V}(A,r)\cset.
$$
For $\ell\geqslant 0$, we consider two subsets $A,B$ of $\Lambda$ and we define a semi-distance between two such subsets, denoted by $d_H^\ell(A,B)$, adapted to our study, by
$$ d_H^\ell(A,B) = \inf\,\left\lbrace r\geqslant 0 \,:\begin{array}{c}A\setminus \mathcal{V}(\Lambda^c,\ell)\subset\mathcal{V}(B,r)\\ B\setminus \mathcal{V}(\Lambda^c,\ell)\subset\mathcal{V}(A,r)\end{array}\right\rbrace.
$$
Notice that $d_H^\ell$ is a semi-distance, in fact the triangle inequality is not satisfied. However, the following lemma allow us to compare $d_H^\ell$ with the Hausforff distance and provides us an alternative to the triangle inequality.
\begin{lem}\label{new.ineqtri}
For two subsets $A,B$ of $\Lambda$ and for all $\ell\geqslant 0$, we have 
$$d_H^\ell(A,B)\vee\ell\geqslant d_H(A\cup\Lambda^c,B\cup\Lambda^c).
$$
\end{lem}
\begin{proof}
Let $A,B$ be two subsets of $\Lambda$, and let us set $d_1 = d_H^\ell(A,B)$. We claim that 
$$A\cup\Lambda^c\subset \mathcal{V}\big(\,B\cup\Lambda^c, d_1\vee \ell\,\big).
$$ 
Let $x\in A\cup \Lambda^c$, we will show that 
$x$ belongs to $\mathcal{V}\big(\,B\cup\Lambda^c, d_1\vee \ell\,\big).
$ We distinguish two cases. If $x\in\mathcal{V}(\Lambda^c,\ell)$, then  we have 
$$x\in\mathcal{V}\big(\,B\cup\Lambda^c, \ell\,\big)\subset \mathcal{V}\big(\,B\cup\Lambda^c, d_1\vee \ell\,\big).
$$
In the other case, if $x\in A\setminus \mathcal{V}(\Lambda^c,\ell)$ and by the definition of $d_H^\ell$, we have 
$$ x\in \mathcal{V}(B,d_1)\subset \mathcal{V}(B\cup\Lambda^c,d_1)\subset\mathcal{V}(B\cup\Lambda^c,d_1\vee \ell).
$$
By exchanging $A$ and $B$, we have $$ B\cup\Lambda^c\subset \mathcal{V}\big(\,A\cup\Lambda^c, d_1\vee \ell\,\big).$$
By the definition of $d_H$, we obtain the desired claim, which in turn proves the lemma.
\end{proof}
\begin{prop}\label{new.dH}
We have the following result:
\begin{multline*}\exists \tilde{p}<1 \quad \exists \kappa>1 \quad \forall p\geqslant \tilde{p}\quad \forall c \geqslant 1\quad\forall \Lambda\quad|\Lambda|\geqslant(cd)^{cd^2}\quad \forall t \geqslant 0 \\  P_\mu\Big(
\exists s\leqslant |\Lambda|\quad d^{\kappa c\ln|\Lambda|}_H(\mathcal{P}_t,\mathcal{P}_{t+s})\geqslant \kappa c\ln |\Lambda|
 \Big)\leqslant \frac{10d}{|\Lambda|^c}.
\end{multline*}
\end{prop}
\begin{proof}
We fix $s\in \oset 1,\dots, |\Lambda|\cset$. By the definition of the distance $d_H^\ell$, we have, for any $\kappa>1$,
\begin{multline}P_\mu\Big(d_H^{\kappa c\ln|\Lambda|}(\mathcal{P}_t,\mathcal{P}_{t+s})\geqslant \kappa c\ln |\Lambda|\Big)\leqslant \\
P_\mu\Big(\mathcal{P}_{t+s}\setminus \mathcal{V}(\Lambda^c,\kappa c\ln|\Lambda|)\nsubseteq \mathcal{V}(\mathcal{P}_t,\kappa c\ln|\Lambda|)\Big)\\+ P_\mu\Big(\mathcal{P}_t\setminus \mathcal{V}(\Lambda^c,\kappa c\ln|\Lambda|)\nsubseteq \mathcal{V}(\mathcal{P}_{t+s},\kappa c\ln|\Lambda|)\Big).\label{new.pH}
\end{multline}
Since the two probabilities in the sum depend only on the process $Y$, which is reversible, they are in fact equal to each other. We shall estimate the first probability. We discuss first the case where there is a time $r\in\oset t,\dots,t+s\cset$ when $\mathcal{P}_r = \emptyset$. By proposition \ref{new.nopivot}, there is a $\tilde{p}<1$ such that, for $p\geqslant \tilde{p}$ and all $\Lambda$,
$$\forall r \in \mathbb{N}\qquad P_\mathcal{D}\left(\mathcal{P}_r = \emptyset\right) \leqslant d|\Lambda|\exp\left(-D\right),
$$
where $D$ is the diameter of $T$.
By summing over the time $r$, we have
\begin{equation}\label{new.probvide}P_\mathcal{D}\left(\exists r\in [t,t+s]\quad\mathcal{P}_r = \emptyset\right)\leqslant d|\Lambda|^2\exp\left(-D\right).
\end{equation}
We now consider the case where there exists always at least one pivotal edge during the time interval $[t,t+s]$. 
We can then apply proposition \ref{new.vconst} with an $\ell$ which will be determined later. There exists $\tilde{p}<1$ such that for $p\geqslant\tilde{p}$, for $t\geqslant 0$, and for any $s\leqslant |\Lambda|$ and $e$ an edge such that $d(e,\Lambda^c)\geqslant \ell$,
$$P_\mu\left( \begin{array}{c|c} \begin{array}{c}e\in\mathcal{P}_{t+s}\\ \forall r\in [t,t+s]\quad \mathcal{P}_r \neq \emptyset\end{array}&\exists c_t\in \mathcal{C}_t,d(e,c_t)\geqslant \ell\end{array}\right) \leqslant e^{-\ell}.
$$
Let us fix $t\geqslant0$, $s\leqslant |\Lambda|$ and $e$ an edge such that $d(e,\Lambda^c)\geqslant \ell$. The previous inequality implies that
$$P_\mu\left(\begin{array}{c} e\in\mathcal{P}_{t+s}\\ \forall r\in [t,t+s], \mathcal{P}_r \neq \emptyset\\ \exists c_t\in \mathcal{C}_t\quad d(e,c_t)\geqslant \ell\end{array}\right) \leqslant e^{-\ell}.
$$
In order to replace $c_t$ by $\mathcal{P}_t$ in the last probability, we use the corollary~\ref{new.distcut}. At the time $t$, the configuration $Y_t$ follows the distribution $P_\mathcal{D}$. Therefore, there exists $\tilde{p}<1$ and a $\kappa' >1 $ such that for $p\geqslant \tilde{p}$, for all $c\geqslant 1$ and all $\Lambda$ such that $|\Lambda|\geqslant 3^{6d}$, we have
$$P_\mu \left(\begin{array}{c}\exists C\in\mathcal{C}_t\quad\exists f\in C\\
 d(f,\Lambda^c\cup \mathcal{P}_t\setminus \{f\})\geqslant \kappa' c\ln|\Lambda|\\
 \forall r\in [t,t+s]\quad \mathcal{P}_r \neq \emptyset\end{array} \right)\leqslant \frac{1}{|\Lambda|^c}.
$$
From now onwards, we suppose that $p$ is larger than the three previous $\tilde{p}$. Let $c>0$ be fixed and let $\kappa'$ be associated to $c$ as above. We distinguish two cases to control the following probability:
\begin{multline*}
P_\mu\big( e\in \mathcal{P}_{t+s}, d(e,\mathcal{P}_t)\geqslant \kappa c\ln|\Lambda|,\forall r\in [t,t+s]\,\, \mathcal{P}_r \neq \emptyset\big)\leqslant \\ P_\mu\left(\begin{array}{c}e\in \mathcal{P}_{t+s},\quad d(e,\mathcal{P}_t)\geqslant \kappa c\ln|\Lambda|,\\ \forall C\in \mathcal{C}_t\quad\forall f\in C\setminus\mathcal{V}(\Lambda^c,\kappa' c\ln|\Lambda|)\\
 d(f,\mathcal{P}_t)< \kappa' c\ln|\Lambda|,\\
 \forall r\in [t,t+s]\quad\mathcal{P}_r \neq \emptyset
 \end{array}\right)\\
+P_\mu\Big( \exists C\in\mathcal{C}_t,\exists f\in C, d(f,\Lambda^c\cup \mathcal{P}_t\setminus \{f\})\geqslant \kappa' c\ln|\Lambda|\Big).
\end{multline*}
The second probability is less than $1/|\Lambda|^c$. Let us study the first probability. Since all the edges of a cut at time $t$ are either at distance less than $\kappa' c\ln|\Lambda|$ from $\Lambda^c$ or at distance less than $\kappa' c\ln|\Lambda|$ from $\mathcal{P}_t$ and the distance between $e$ and $\mathcal{P}_t\cup \Lambda^c$ is larger than $\kappa c\ln |\Lambda|$, then all the cuts at time $t$ are at distance more than $(\kappa-\kappa') c\ln|\Lambda|$ from $e$. Hence, for $\kappa> \kappa'$,
\begin{multline*}
P_\mu\left(\begin{array}{c}e\in \mathcal{P}_{t+s}\quad d(e,\Lambda^c\cup\mathcal{P}_t)\geqslant \kappa c\ln|\Lambda|\\ \forall C\in \mathcal{C}_t\quad\forall f\in C\setminus\mathcal{V}(\Lambda^c,\kappa' c\ln|\Lambda|)\\
 d(f,\mathcal{P}_t)< \kappa' c\ln|\Lambda|\\
 \forall r\in [t,t+s]\quad \mathcal{P}_r \neq \emptyset\end{array}\right)\leqslant \\
P_\mu\left(
\begin{array}{c}e\in \mathcal{P}_{t+s}\quad d(e,\Lambda^c)\geqslant \kappa c\ln|\Lambda|\\ \forall C\in \mathcal{C}_t\quad\forall f\in C\setminus\mathcal{V}(\Lambda^c,\kappa' c\ln|\Lambda|)\\
 d(f,e)> (\kappa-\kappa') c\ln|\Lambda|\\
 \forall r\in [t,t+s]\quad \mathcal{P}_r \neq \emptyset\end{array}
\right)\leqslant \\
 P_\mu\left(\begin{array}{c} e\in\mathcal{P}_{t+s}\\ \exists c_t\in \mathcal{C}_t\quad d(e,c_t)\geqslant (\kappa-\kappa')c\ln|\Lambda|\\ \forall r\in [t,t+s]\quad \mathcal{P}_r \neq \emptyset\end{array}\right) \leqslant \frac{1}{|\Lambda|^{(\kappa-\kappa')c}}.
\end{multline*}
We choose now $\kappa = \kappa'+1$, and we get
$$P_\mu\big( e\in \mathcal{P}_{t+s},d(e,\mathcal{P}_t)\geqslant \kappa c\ln|\Lambda|,\forall r\in [t,t+s]\,\, \mathcal{P}_r \neq \emptyset\big) \leqslant \frac{2}{|\Lambda|^c}.
$$
We sum over $e$ in $\Lambda$ and $s\in \oset 1,\dots,|\Lambda|\cset$ to get
$$P_\mu\left(\begin{array}{c}\exists s\leqslant |\Lambda|, \exists e\in \mathcal{P}_{t+s}\\ d(e,\Lambda^c\cup\mathcal{P}_t)\geqslant \kappa c\ln|\Lambda| \\
\forall r\in [t,t+s]\quad \mathcal{P}_r \neq \emptyset \end{array}\right) \leqslant \frac{4d}{|\Lambda|^{c-2}}.
$$
We add the probability in~\eqref{new.probvide} and we obtain

$$P_\mu\left(\begin{array}{c}
\exists s\leqslant |\Lambda|, \exists e\in \mathcal{P}_{t+s}\\ d(e,\Lambda^c\cup\mathcal{P}_t)\geqslant \kappa c\ln|\Lambda|
\end{array}  \right)\leqslant \frac{4d}{|\Lambda|^{c-2}}+d|\Lambda|^2\exp\left(-D\right).
$$ 
This is the first probability in~\eqref{new.pH} and we conclude that
$$P_\mu\left(\begin{array}{c}
\exists s\leqslant |\Lambda|\\d_H^\Lambda(\mathcal{P}_t,\mathcal{P}_{t+s})\geqslant \kappa c\ln |\Lambda|
\end{array}  \right)\leqslant \frac{8d}{|\Lambda|^{c-2}}+2d|\Lambda|^2\exp\left(-D\right).
$$ 
For all box $\Lambda$ such that $|\Lambda|\geqslant (cd)^{cd^2}$, we have 
$$\exp\left(-D\right)\leqslant \frac{1}{|\Lambda|^c}.
$$
Therefore, for all $\Lambda$ such that $|\Lambda|\geqslant (cd)^{cd^2}$, we have
$$\frac{8d}{|\Lambda|^{c-2}}+2d|\Lambda|^2\exp\left(-D\right) \leqslant \frac{10d}{|\Lambda|^{c-2}}.
$$
In order to obtain $1/|\Lambda|^c$, we replace $c$ by $c+2$, since $(c+2)/c\leqslant 3$ for $c\geqslant 1$, we have, for $|\Lambda|\geqslant \max\left\lbrace (cd)^{cd^2},3^{6d}\right\rbrace$,
\begin{multline*}P_\mu\left(\begin{array}{c}
\exists s\leqslant |\Lambda|\\d_H^\Lambda(\mathcal{P}_t,\mathcal{P}_{t+s})\geqslant 3\kappa c\ln |\Lambda|
\end{array}  \right)\\ \leqslant P_\mu\left(\begin{array}{c}
\exists s\leqslant |\Lambda|\\d_H^\Lambda(\mathcal{P}_t,\mathcal{P}_{t+s})\geqslant \kappa (c+2)\ln |\Lambda|
\end{array}  \right)\leqslant \frac{10d}{|\Lambda|^c}.
\end{multline*}
This yields the desired inequality.
\end{proof}
We now complete the proof of the theorem~\ref{new.main2}.
\begin{proof}[Proof of theorem~\ref{new.main2}]
Let us fix an edge $e$ in $\Lambda$ and a time $t$. We distinguish the cases where $e\in\mathcal{I}_t\setminus \mathcal{P}_t$ and $e\in \mathcal{P}_t$. If $e\in \mathcal{P}_t$, then we use the proposition~\ref{new.distpivot}. We consider now the case where $e\in\mathcal{I}_t\setminus \mathcal{P}_t$. We consider the last time $\tau$ when $e$ was pivotal,
$$\tau = \max\,\oset 0\leqslant s < t\,:\,e \in \mathcal{P}_s, e\notin \mathcal{P}_{s+1}\cset.
$$
The edge $e$ has not been modified between $\tau$ and $t$. Let $c\geqslant1$. We have
$$
P_\mu(t-\tau \geqslant c|\Lambda|\ln|\Lambda|)\leqslant P_\mu\left(\begin{array}{c}\forall r\in [t-c|\Lambda|\ln|\Lambda|,t]\\
E_r \neq e\end{array}\right)\leqslant \frac{1}{|\Lambda|^c}.
$$
We consider now the case where $t-\tau< c|\Lambda|\ln|\Lambda|$. We split the interval $[\tau,t]$ into subintervals of length $|\Lambda|$ and we set
$$t_i = \tau+ i|\Lambda|,\quad 0\leqslant i < \frac{t-\tau}{|\Lambda|} \quad\text{and}\quad t_{\lfloor (t-\tau)/|\Lambda|\rfloor+1} = t.$$
\begin{figure}[ht]
\centering{
\resizebox{120mm}{!}{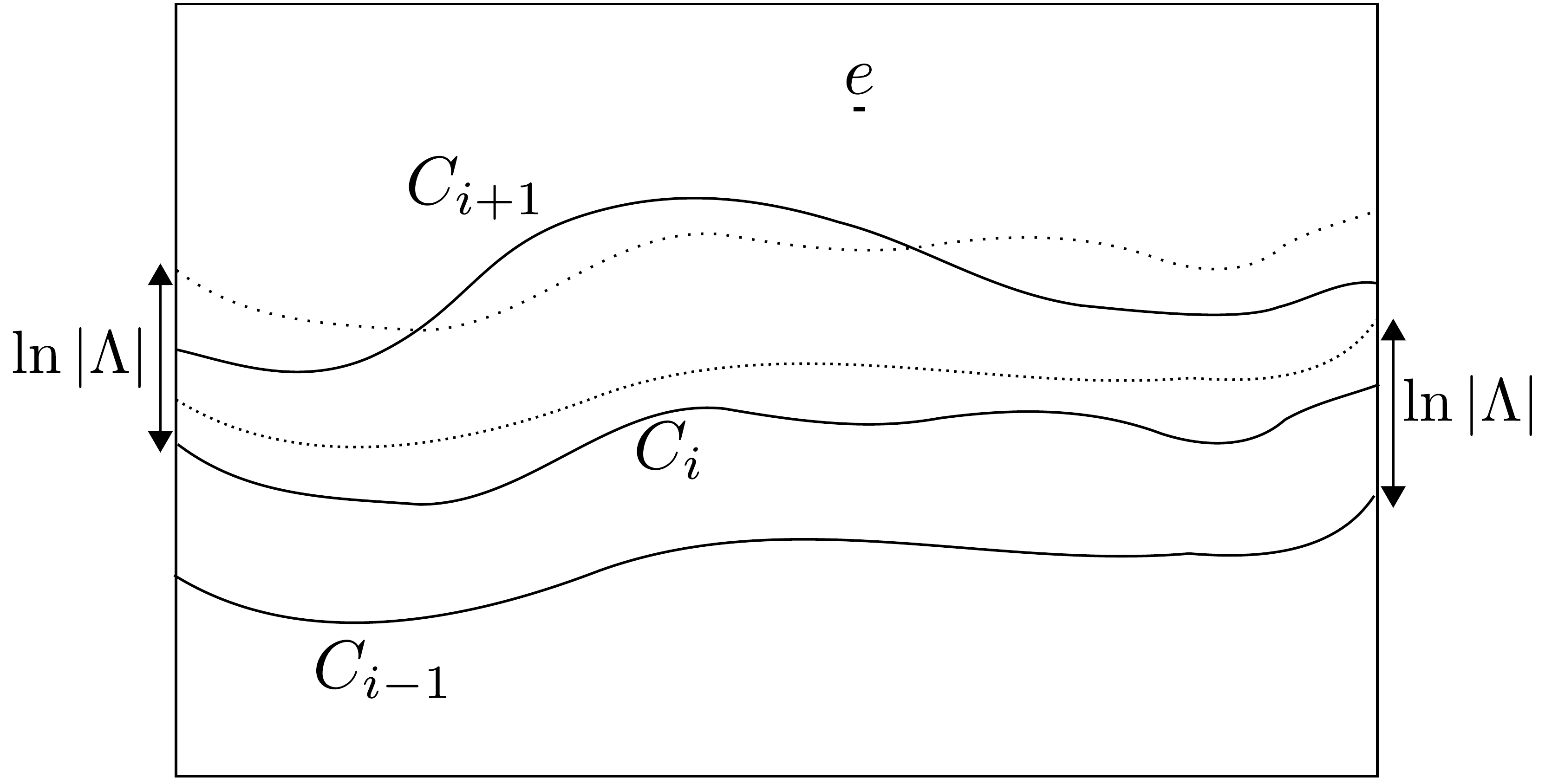}
\caption{The cut $C_i$ at time $t_i$ is at distance less than $\ln |\Lambda|$ from $C_{i-1}$ and the cut $C_{i+1}$ is at distance more than $\ln |\Lambda|$ from $C_i$.}
\label{new.fig:deplacement}
}
\end{figure}
According to proposition~\ref{new.dH}, there exists $\tilde{p}<1$ and $\kappa'>1$, such that
\begin{multline*}\forall p\geqslant \tilde{p}\quad \forall c\geqslant 1 \quad\forall |\Lambda|\geqslant (cd)^{cd^2} \quad\forall j\geqslant 0 \\ P_\mu\Big( d_H^{\kappa' c\ln|\Lambda|}(\mathcal{P}_{t_j},\mathcal{P}_{t_{j+1}})\geqslant \kappa'c\ln|\Lambda|\Big)\leqslant \frac{10d}{|\Lambda|^c}.
\end{multline*}
Let $c\geqslant 1$. We suppose that $$d(e,\mathcal{P}_t\cup\Lambda^c)\geqslant 2\kappa'c^2(\ln|\Lambda|)^2>(c\ln|\Lambda|+1)\kappa'c\ln|\Lambda|.$$ 
We have by lemma \ref{new.ineqtri}, as illustrated in the figure~\ref{new.fig:deplacement},
\begin{multline*}\sum_{0\leqslant i<(t-\tau)/|\Lambda|} d_H^{\kappa'c\ln|\Lambda|}(\mathcal{P}_{t_i},\mathcal{P}_{t_{i+1}})\vee\kappa'c\ln|\Lambda|\\ \geqslant \sum_{0\leqslant i<(t-\tau)/|\Lambda|} d_H(\mathcal{P}_{t_i}\cup\Lambda^c,\mathcal{P}_{t_{i+1}}\cup\Lambda^c)\\ \geqslant d_H(\mathcal{P}_\tau\cup\Lambda^c,\mathcal{P}_t\cup\Lambda^c)\\\geqslant d(e,\mathcal{P}_t\cup\Lambda^c) \geqslant 2\kappa'c^2(\ln|\Lambda|)^2.
\end{multline*}
Necessarily, there is an index $0\leqslant j <c\ln|\Lambda|$ such that 
$$d_H^{\kappa'c\ln|\Lambda|}(\mathcal{P}_{t_j},\mathcal{P}_{t_{j+1}})\vee \kappa'c\ln|\Lambda|\geqslant 2\kappa'c\ln|\Lambda|.
$$
Therefore, we have
$$d_H^{\kappa'c\ln|\Lambda|}(\mathcal{P}_{t_j},\mathcal{P}_{t_{j+1}})\geqslant \kappa'c\ln|\Lambda|.
$$
By summing over $j$ from $0$ to $\lfloor c\ln|\Lambda|\rfloor$, we have 
$$P_\mu\Big(e\in \mathcal{I}_t,d(e,\mathcal{P}_t)\geqslant 2\kappa'c^2(\ln|\Lambda|)^2, t-\tau < c|\Lambda|\ln|\Lambda|\Big) \leqslant \frac{10d(c\ln|\Lambda|+1)}{|\Lambda|^c}.
$$
We set $\kappa = 2\kappa'$ and we obtain
\begin{multline*}
P_\mu\Big(e\in\mathcal{P}_t\cup\mathcal{I}_t,d(e,\mathcal{P}_t\setminus\{e\})\geqslant \kappa c^2(\ln|\Lambda|)^2\Big)\\ \leqslant P_\mu\Big(e\in \mathcal{P}_t,d(e,\mathcal{P}_t\setminus\{e\})\geqslant \kappa (c\ln|\Lambda|)^2\Big)\\
+P_\mu\Big( t-\tau \geqslant c |\Lambda|\ln|\Lambda|\Big)+P_\mu\Big(e\in\mathcal{I}_t, d(e,\mathcal{P}_t)\geqslant \kappa(c\ln|\Lambda|)^2, t-\tau \leqslant c\ln|\Lambda|\Big)\\
\leqslant \frac{2}{|\Lambda|^c}+\frac{10d(c\ln|\Lambda|+1)}{|\Lambda|^c}.
\end{multline*}
We sum over the edge $e$. For $\Lambda$ such that $|\Lambda|\geqslant (cd)^{cd^2}$, we have
\begin{multline*}P_\mu \Big(\exists e\in \mathcal{P}_t\cup\mathcal{I}_t,d\left(e,\mathcal{P}_t\setminus\{e\}\right)\geqslant \kappa(c\ln |\Lambda|)^2 \Big)\leqslant\\ \frac{4d+20d^2(c\ln|\Lambda|+1)}{|\Lambda|^{c-1}}\leqslant \kern-1pt\frac{1}{|\Lambda|^{c-2}}.
\end{multline*}
We apply this result with $c+2$ of $c$, since for $c\geqslant 1$, $(c+2)^2/c^2\leqslant 9$, we have 
$$9\kappa c^2\geqslant \kappa(c+2)^2.
$$
Therefore, we have
$$P_\mu \Big(\exists e\in \mathcal{P}_t\cup\mathcal{I}_t,d\left(e,\mathcal{P}_t\setminus\{e\}\right)\geqslant 9\kappa(c\ln |\Lambda|)^2 \Big)\leqslant \frac{1}{|\Lambda|^c}.
$$ This proves the statement of theorem \ref{new.main2}.
\end{proof}
\section{Speed estimations conditionned by the past}\label{new.speedI}

We derive further estimates on the speed of the pivotal edges which will be used in the proof of the theorem~\ref{new.main1}.
First, we give a corollary of the proposition~\ref{new.vconst}, which provides a control on the cuts, rather than the pivotal edges. 
\begin{cor} \label{new.vcconst}
We have the following inequality:
\begin{multline*}
\exists \tilde{p}<1\quad \forall p \geqslant \tilde{p}\quad \forall \Lambda \\ \forall \ell\geqslant 1\quad\forall e\in \Lambda\quad d(e,\Lambda^c)\geqslant \ell\quad\forall t>0\quad  \forall s\in \{0,\dots,|\Lambda|\}\\ P_\mu\Big(\exists C\in \mathcal{C}_{t+s},e\in C\,\big|\,\exists c_t\in \mathcal{C}_t,\, d(e,c_t)\geqslant\ell
\Big)\leqslant \exp(-\ell).
\end{multline*}
\end{cor}
\begin{proof}
We adapt the construction of the STP done in the proposition~\ref{new.stc}. We cannot use directly the STP constructed in proposition~\ref{new.stc} because between the times $t$ and $t+s$, the set of the pivotal edges can be empty. Therefore, we consider $\tau$ the last time before $t+s$ when $\mathcal{P}$ is empty, i.e.,
$$\tau = \sup \oset r \leqslant t+s \,:\, \mathcal{P}_{r} = \emptyset \cset.
$$
If $\tau\leqslant t$, the conditions of proposition~\ref{new.stc} are satisfied and there exists a closed decreasing simple STP starting from $(e,t+s)$ and ending after $t$ which travels a distance at least $\ell$. If $\tau = t+s$, since the edge $e$ is in a cut, there exists a closed $*$-path in $Y_{t+s}$ which connects $e$ to an edge intersecting the boundary of $\Lambda$. This path travels a distance at least $\ell$. If $t<\tau<t+s$, then we have $\mathcal{P}_r\neq \emptyset$ for $\tau< r \leqslant t+s$. According to proposition~\ref{new.stc}, there exists a STP from $(e,t+s)$ to an edge of $\mathcal{P}_{\tau+1}$ at time $\tau+1$ or an edge intersecting the boundary of $\Lambda$ after time $\tau+1$. If the STP ends at an edge intersecting the boundary, then it travels a distance at least $\ell$. If it ends at an edge of $\mathcal{P}_{\tau+1}$ at time $\tau+1$, then, at time $\tau+1$, there must be an edge which becomes open and creates the pivotal edges of $\mathcal{P}_{\tau+1}$ which are on a cut $C$ at time $\tau+1$. Notice that the cut $C$ existed already at time $\tau$ because all the edges of $C$ are closed. Therefore, there exists a decreasing closed STP which connects $(e,t+s)$ to an edge intersecting the boundary of $\Lambda$ at time $\tau$. We reapply the algorithm of modification described in the proof of proposition~\ref{new.stc} to obtain a simple STP. In all the cases above, we obtain a decreasing closed simple STP starting at $(e,t+s)$ which travels a distance at least $\ell$. We apply the same arguments as in the proof of proposition~\ref{new.vconst} in order to obtain the desired estimate.
\end{proof}

We wish to control the movement of the set of the cuts over a time interval. To achieve this goal, we will derive estimates for the appearance of a pivotal edge conditionally on the presence of a cut far away during a whole interval. In proposition \ref{new.vconst}, the conditioning gave information on one instant, not a whole interval. In the next lemma, we deal with a time interval of length $|\Lambda|$.
\begin{lem}\label{new.recinit}There exist $\tilde{p}<1$ and $\kappa>0$ such that for $p\geqslant \tilde{p}$, any $c\geqslant 1$, any integer $m\geqslant 1$, any $\Lambda$ such that $|\Lambda|\geqslant 2d$, any edge $e$ at distance more than $\kappa c\ln|\Lambda|$ from $ \Lambda^c$ and for $0<s\leqslant |\Lambda|\leqslant t$, we have
\begin{multline*}P_\mu\left(\kern -10 pt\begin{array}{c|c}\begin{array}{c}\exists C\in \mathcal{C}_{t+s}\\d(e,C)\leqslant (m-1) \kappa c\ln|\Lambda|\end{array}\kern-5pt&\kern -8 pt\begin{array}{c}
\forall r\in ]t-|\Lambda|,t]\\
\exists C_r\in \mathcal{C}_r\quad d(e,C_r)\geqslant m\kappa c\ln|\Lambda|\\
\exists C'\in \mathcal{C}_{t-|\Lambda|}\quad d(e,C')\geqslant (m+1)\kappa c\ln|\Lambda| \end{array}\end{array}\kern -10 pt\right)\\ \leqslant \frac{1}{|\Lambda|^c}.
\end{multline*}
\end{lem}
\begin{proof}
Let $\kappa$ be a positive constant which will be chosen at the end of the proof.
We reuse the construction of the STP in corollary~\ref{new.vcconst}: there exists a decreasing closed simple STP which connects $(e,t+s)$ to a pivotal edge at time $t$ or to an edge intersecting the boundary of $\Lambda$ at a time after $t$. Since the edge $e$ is at distance at least $\kappa c \ln|\Lambda|$ from $\mathcal{P}_t\cup \Lambda^c$, in both cases, there exists a decreasing closed simple STP $\gamma$ of length $(\kappa c\ln|\Lambda|)/2d$ starting from the time-edge $(e,t+s)$ and ending after $t$ which is strictly included in the box $\Lambda$. Let $\Gamma$ be the space projection of $\gamma$, i.e., $$\Gamma = \mathrm{Space}(\gamma) = (e_1,\dots,e_m).$$
We introduce the following events: 
$$D_1 = \oset \forall r\in ]t-|\Lambda|,t],\exists C_r\in\mathcal{C}_r, d(e,C_r)\geqslant m\kappa c\ln|\Lambda|\cset,
$$
$$\overline{D}_1 = \oset \exists C\in \mathcal{C}_{t-|\Lambda|}, d(e,C)\geqslant (m+1)\kappa c\ln|\Lambda|\cset,
$$
and
$$\mathcal{E}(t,s,\Gamma) =\left\lbrace\begin{array}{c}\exists \gamma \text{ simple closed decreasing STP}\\
\mathrm{length}(\gamma) = (\kappa c\ln|\Lambda|)/2d, \mathrm{Space}(\gamma)  = \Gamma\\
\gamma \text{ starts at an edge }(e',t+s)\\
d(e',e)\leqslant (m-1)\kappa c\ln|\Lambda|\\\text{ and ends after }t\end{array}\right\rbrace.
$$
As in the proof of proposition \ref{new.vconst}, the probability appearing in the proposition is less than
\begin{equation}\sum_\Gamma P_\mu\left(\kern-5pt\begin{array}{c|c}\mathcal{E}(t,s,\Gamma)&\kern-5pt\begin{array}{c}
D_1, \overline{D}_1 \end{array}\end{array}\kern-10pt\right),\label{new.n=1}
\end{equation}
where the sum is over the possible choices for $\Gamma$.
We fix a path $\Gamma$ and we condition each probability in the sum by the configuration at time $t$. Let $A$ be a subset of $\mathrm{support}(\Gamma)$, we denote by $M(A)$ the following set of configurations:
$$M(A) = \left\lbrace \begin{array}{cl}\omega\,:&\begin{array}{l} 
\forall f \in A\quad \omega(f) = 0\\
\forall f \in \mathrm{support}(\Gamma)\setminus A\quad \omega(f) = 1\\
\end{array}\end{array}\right\rbrace.
$$
Let $y$ be a configuration in $M(A)$ and let us start by estimating the probability
$$P_\mu\left(\kern-5pt\begin{array}{c|c}\mathcal{E}(t,s,\Gamma)&
Y_t = y, D_1, \overline{D}_1\end{array}\kern-5pt\right).
$$
By the Markov property, this probability is equal to 
$$P_\mu\left(\begin{array}{c|c}\mathcal{E}(t,s,\Gamma)&
Y_t = y\end{array}\right),
$$
and by proposition~\ref{new.BKFermeutre}, it is less than
\begin{equation}\left(\frac{s}{|\Lambda|}(1-p)\right)^{(\kappa c\ln|\Lambda|)/2d -|A|}.\label{new.aftert}
\end{equation}
Each term of the sum in~\eqref{new.n=1} can be written as 
\begin{multline*}\sum_{0\leqslant k \leqslant |\mathrm{support}(\Gamma)|}\smash{\sum_{\mbox{\scriptsize$\begin{array}{c}A\subset \mathrm{support}(\Gamma)\\|A|  = k\end{array}$}}}\sum_{y\in M(A)}P_\mu\left(\begin{array}{c|c}\mathcal{E}(t,s,\Gamma)&
Y_t = y, D_1, \overline{D}_1\end{array}\right)\times\\P_\mu\left(\kern-5pt\begin{array}{c|c}Y_t = y&D_1,\overline{D}_1\end{array}\kern-5pt\right).
\end{multline*}
Using~\eqref{new.aftert}, we see that each term in~\eqref{new.n=1} is less than
\begin{equation}\sum_{0\leqslant k \leqslant |\mathrm{support}(\Gamma)|}\left(\frac{s}{|\Lambda|}(1-p)\right)^{(\kappa c\ln|\Lambda|)/2d -k}\kern-20pt\sum_{\mbox{\scriptsize$\begin{array}{c}A\subset \mathrm{support}(\Gamma)\\|A|  = k\end{array}$}}\kern-20pt P_\mu\left(\kern-5pt\begin{array}{c|c}Y_t \in M(A)&D_1,\overline{D}_1\end{array}\kern-5pt\right).\vspace{-5pt}\label{new.pED}
\end{equation}
In the rest of the proof, we will calculate an upper bound of 
\begin{equation}\sum_{\mbox{\scriptsize$\begin{array}{c}A\subset \mathrm{support}(\Gamma)\\|A|  = k\end{array}$}}\kern-20ptP_\mu\left(\kern-5pt\begin{array}{c|c}Y_t \in M(A)&D_1,\overline{D}_1\end{array}\kern-5pt\right).\vspace{-5pt}\label{new.pMA}
\end{equation}
Notice that, for an edge $f\in \Gamma$, if there is a time $r\in [t-|\Lambda|,t]$ such that $E_r = f$, then, under the probability $P_p$, conditioned on $D_1,\overline{D_1}$, the state of $f$ at time $t$ is independent of the other edges of $\Gamma$ and it follows a Bernoulli variable of parameter $p$. On the contrary, if $E_r\neq f$ for all $r\in[t-|\Lambda|,t]$, then the state of $f$ at time $t$ is the same as at time $t-|\Lambda|$. For $A$ a subset of $\mathrm{support}(\Gamma)$ and $B$ a subset of $A$, we define the event $\mathrm{reset}(B,A)$ as
$$\mathrm{reset}(B,A) = \left\lbrace \begin{array}{c}\forall e \in B,\, \exists r\in [t-|\Lambda|,t], E_r = e\\
\forall r\in [t-|\Lambda|,t],E_r\notin A\setminus B\end{array}\right\rbrace.
$$
For each subset $A$, we partition the probability in~\eqref{new.pMA} according to the subset $B$ of $A$ for which the event $\mathrm{reset}(B,A)$ occurs, and we get
\begin{multline}\label{new.pMk}
\kern -10 pt\sum_{\mbox{\scriptsize$\begin{array}{c}A\subset \mathrm{support}(\Gamma)\\ |A| = k, B\subset A\end{array}$}}\kern -25pt P_\mu \big(Y_t\in M(A), \mathrm{reset}(B,A)\,\big|\, D_1,\overline{D}_1\big)\\ 
= \kern -20 pt\sum_{\mbox{\scriptsize$\begin{array}{c}A\subset \mathrm{support}(\Gamma)\\ |A| = k, B\subset A\end{array}$}}\kern -20pt P_\mu\left(\kern-5pt\begin{array}{c|c}
\begin{array}{c} \forall f \in B,\, Y_t(f) = 0\\
\forall f\in A\setminus B, Y_{t-|\Lambda|}(f) = 0\\
\mathrm{reset}(B,A)\end{array}\kern-5pt&\kern-5pt\begin{array}{c}D_1,\overline{D}_1\end{array}\end{array}\kern-5pt\right).
\end{multline}
We write 
$$P_\mu(\cdot) = \sum_{x_0,y_0}P_{x_0,y_0}\big(\cdot\big)\mu\big((x_0,y_0)\big)
,$$
where $P_{x_0,y_0}$ is the law of the process $(X_t,Y_t)_{t\in\mathbb{N}}$ starting from the initial configuration $(x_0,y_0)$. For each term we rewrite the conditioned probability as follows:
\begin{multline}
P_{x_0,y_0}\left(\kern-5pt\begin{array}{c|c}
\begin{array}{c} \forall f \in B,\, Y_t(f) = 0\\
\forall f\in A\setminus B, Y_{t-|\Lambda|}(f) = 0\\
\mathrm{reset}(B,A)\end{array}\kern-5pt&\kern-5pt\begin{array}{c}D_1,\overline{D}_1\end{array}\end{array}\kern-5pt\right) \\
= \frac{P_{x_0,y_0}\left(\left\{ \begin{array}{c} \forall f \in B,\, Y_t(f) = 0\\
\forall f\in A\setminus B, Y_{t-|\Lambda|}(f) = 0\\
\mathrm{reset}(B,A)\end{array}\right\}\bigcap D_1\bigcap\overline{D}_1\right)}{P_{x_0,y_0}(D_1,\overline{D}_1)}.
\label{new.1-pB}
\end{multline}
Starting from an initial configuration $(x_0,y_0)$, the process $(Y_t)_{t\in\mathbb{N}}$ is obtained by conditioning to stay in the configurations with disconnexion.
We can replace $P_\mu$ by $P_p$ in the previous fraction and the numerator can be written as
\begin{multline*}P_p\left(\left\{ \begin{array}{c} \forall f \in B,\, X_t(f) = 0\\
\forall f\in A\setminus B, X_{t-|\Lambda|}(f) = 0\\
\mathrm{reset}(B,A)\end{array}\right\}\bigcap D_1\bigcap\overline{D}_1\right)
\\
= P_p\left(\left\{ \begin{array}{c} \forall f \in B,\, B_{\tau(f)} = 0\\
\forall f\in A\setminus B, X_{t-|\Lambda|}(f) = 0\\
\mathrm{reset}(B,A)\end{array}\right\}\bigcap D_1\bigcap\overline{D}_1\right),
\end{multline*}
where the time $\tau(f)$ is the last time before $t$ when the edge $f$ is chosen, i.e.,
$$\tau(f) = \sup\oset s\leqslant t \,:\, E_s = f\cset.
$$
Let us fix a sequence of edges $\mathbf{e} =(e_1,\dots,e_t)$ and let us condition this last probability by the event
$$(E_1,\dots,E_t) = \mathbf{e}.
$$
We have
\begin{multline}P_p\left(\left\{ \begin{array}{c} \forall f \in B,\, B_{\tau(f)} = 0\\
\forall f\in A\setminus B, X_{t-|\Lambda|}(f) = 0\\
\mathrm{reset}(B,A)\end{array}\right\}\bigcap D_1\bigcap\overline{D}_1\right)\\ = 
\kern-15pt\sum_{\mathbf{e}\in\mathrm{reset}(B,A)}\kern-15ptP_p\left(\kern-5pt\begin{array}{c|c}\left\{\kern-5pt \begin{array}{c} \forall f \in B,\, B_{\tau(f)} = 0\\
\forall f\in A\setminus B, X_{t-|\Lambda|}(f) = 0\end{array}\kern-5pt\right\}\bigcap D_1\bigcap\overline{D}_1&(E_1,\dots,E_t) = \mathbf{e}\end{array}\kern-5pt\right)\\
\times P_p\big((E_1,\dots,E_t) = \mathbf{e}\big).\label{new.condE}
\end{multline}
Notice that on the event $\mathrm{reset}(B,A)$, for an edge $f\in B$, we have necessarily $\tau(f)>t-|\Lambda|$.  Therefore the event 
$\oset \forall f \in B,\, B_{\tau(f)} = 0 \cset$ depends on the set of variables $\oset B_s: e_s \in B , s >t-|\Lambda|\cset$. The events $$\oset \forall f\in A\setminus B, X_{t-|\Lambda|}(f) = 0\cset$$ and $\overline{D}_1$ depend on the variables $\oset B_s\,:\, s\leqslant t-|\Lambda|\cset $ and the event $D_1$ does not depend on the variables $\oset B_s : E_s \in B\cset$. All the events above are decreasing, by the BK inequality applied to the random variables $(B_s)_{s\in\mathbb{N}}$, we have 
\begin{multline*}
P_p\left(\kern-5pt\begin{array}{c|c}\left\{\kern-5pt \begin{array}{c} \forall f \in B,\, B_{\tau(f)} = 0\\
\forall f\in A\setminus B, X_{t-|\Lambda|}(f) = 0\end{array}\kern-5pt\right\}\bigcap D_1\bigcap\overline{D}_1&(E_1,\dots,E_t) = \mathbf{e}\end{array}\kern-5pt\right)\\
\leqslant 
P_p\big(\forall f \in B,\, B_{\tau(f)} = 0\,\big|\, (E_1,\dots,E_t) = \mathbf{e}\big)\\
\times P_p\left(\begin{array}{c|c}\begin{array}{c}\forall f\in A\setminus B, X_{t-|\Lambda|}(f) = 0\\
D_1\bigcap\overline{D}_1\end{array}& (E_1,\dots,E_t) = \mathbf{e}\end{array}\right)\\
\leqslant (1-p)^{|B|}P_p\left(\begin{array}{c|c}\begin{array}{c}\forall f\in A\setminus B, X_{t-|\Lambda|}(f) = 0\\
D_1\bigcap\overline{D}_1\end{array}& (E_1,\dots,E_t) = \mathbf{e}\end{array}\right).
\end{multline*}
We use this inequality in \eqref{new.condE} and we obtain
\begin{multline*}P_p\left(\left\{ \begin{array}{c} \forall f \in B,\, B_{\tau(f)} = 0\\
\forall f\in A\setminus B, X_{t-|\Lambda|}(f) = 0\\
\mathrm{reset}(B,A)\end{array}\right\}\bigcap D_1\bigcap\overline{D}_1\right)\\
\leqslant (1-p)^{|B|}P_p\left(\left\{ \begin{array}{c} 
\forall f\in A\setminus B, X_{t-|\Lambda|}(f) = 0\\
\mathrm{reset}(B,A)\end{array}\right\}\bigcap D_1\bigcap\overline{D}_1\right).
\end{multline*}
We replace the numerator in \eqref{new.1-pB} and we sum over the initial configurations, we have 
\begin{multline*}
P_\mu\left(\kern-5pt\begin{array}{c|c}
\begin{array}{c} \forall f \in B,\, Y_t(f) = 0\\
\forall f\in A\setminus B, Y_{t-|\Lambda|}(f) = 0\\
\mathrm{reset}(B,A)\end{array}\kern-5pt&\kern-5pt\begin{array}{c}D_1,\overline{D}_1\end{array}\end{array}\kern-5pt\right) \\
\leqslant (1-p)^{|B|}P_\mu\left(\begin{array}{c|c} \begin{array}{c} 
\forall f\in A\setminus B, X_{t-|\Lambda|}(f) = 0\\
\mathrm{reset}(B,A)\end{array}& D_1,\overline{D}_1\end{array}\right).
\end{multline*}
This last probability is less than 
\begin{equation}\label{new.A-B}\frac{P_\mu \left(\begin{array}{c|c}\begin{array}{c}\forall f\in A\setminus B\quad Y_{t-|\Lambda|}(f) = 0\\ \mathrm{reset}(B,A)\end{array}&\overline{D}_1 \end{array}\right)}{P_\mu(D_1|\overline{D}_1)}.
\end{equation}
Let us estimate separately the numerator and the denominator. In order to calculate the numerator, we use the notation $\overline{M}(A)$ defined as follows:
$$\overline{M}(A) = \oset \omega\,:\,
\forall f \in A\quad \omega(f) = 0
\cset.
$$ 
We obtain 
\begin{multline*}
P_\mu \left(\kern-10pt\begin{array}{c|c}\begin{array}{c}\forall f\in A\setminus B\quad Y_{t-|\Lambda|}(f) = 0\\ \mathrm{reset}(B,A)\end{array}\kern-5pt&\overline{D}_1 \end{array}\kern-5pt\right) \\= \kern-11pt\sum_{y\in \overline{M}(A\setminus B)}\kern-13pt P_\mu \Big( \mathrm{reset}(B,A),\, Y_{t-|\Lambda|} = y\,\big|\, \overline{D}_1\Big).
\end{multline*}
As in the proof of proposition~\ref{new.vconst}, we write
\begin{multline*}P_\mu\Big(\mathrm{reset}(B,A), Y_{t-|\Lambda|}=y\,\big|\, \overline{D}_1 \Big) = \\P_\mu\Big(\mathrm{reset}(B,A)\,\big|\, Y_{t-|\Lambda|} = y, \overline{D}_1  \Big)P_\mu \Big(Y_{t-|\Lambda|}=y\,\big|\, \overline{D}_1\Big).
\end{multline*}
Since the event $\mathrm{reset}(B,A)$ depends only on the variables $$\oset(E_r,B_r)\,:\,t-|\Lambda|< r\leqslant t\cset,$$ it is independent from $Y_{t-|\Lambda|}$ (and also from the event $\overline{D}_1$, as $\overline{D}_1$ is entirely determined by $Y_{t-|\Lambda|}$). We obtain
\begin{multline*}\sum_{y\in \overline{M}(A\setminus B)}P_\mu\Big(\mathrm{reset}(B,A), Y_{t-|\Lambda|}=y\,\big|\, \overline{D}_1 \Big) =\\ P_\mu\Big(\mathrm{reset}(B,A)\Big)P_\mu\Big(Y_{t-|\Lambda|} \in \overline{M}(A\setminus B)\,\big|\,\overline{D}_1\Big).
\end{multline*}
Let us estimate the last probability. Since the second marginal of $P_\mu$ is $P_\mathcal{D}$ and $P_\mathcal{D}(\cdot) = P_p(\cdot|T\nlongleftrightarrow B)$, we have
\begin{multline*}P_\mu\Big(Y_{t-|\Lambda|} \in \overline{M}(A\setminus B)\,\big|\,\overline{D}_1\Big)\\ = \frac{P_p\left(\begin{array}{c}\forall f\in A\setminus B\quad f\text{ closed}\\ \exists C\in \mathcal{C}, d(e,C)\geqslant (m+1)\kappa c\ln |\Lambda|\end{array}\right)}{P_\mathcal{D}\Big(\exists C\in \mathcal{C}, d(e,C)\geqslant (m+1)\kappa c\ln |\Lambda|\Big)P_p\Big(T\nlongleftrightarrow B\Big)}.
\end{multline*}
The event $$\oset \forall f\in A\setminus B\quad f\text{ closed}\cset$$ depends only on the edges at distance less than $(m-1/2)\kappa c\ln|\Lambda|$ from the edge~$e$, while the event $$\oset\exists C\in \mathcal{C}, d(e,C)\geqslant (m+1)\kappa c\ln |\Lambda|\cset$$ depends on the edges at distance larger than $(m+1)\kappa c\ln|\Lambda|$ from $e$. By independence, we have
\begin{multline*}P_p\left(\begin{array}{c}\forall f\in A\setminus B\quad f\text{ closed}\\ \exists C\in \mathcal{C}, d(e,C)\geqslant (m+1)\kappa c\ln |\Lambda|\end{array}\right)=\\ P_p\Big(\forall f\in A\setminus B\quad f\text{ closed}\Big)P_p\Big(\exists C\in \mathcal{C}, d(e,C)\geqslant (m+1)\kappa c\ln |\Lambda|\Big).
\end{multline*}
We obtain therefore
$$P_\mu\Big(Y_{t-|\Lambda|} \in \overline{M}(A\setminus B)\,\big|\,\overline{D}_1\Big)\leqslant P_p\Big(\forall f\in A\setminus B\quad f\text{ closed}\Big) = (1-p)^{|A\setminus B|}.
$$
We conclude that the numerator of~\eqref{new.A-B} is less than 
$$ (1-p)^{|A\setminus B|}P_\mu\Big(\mathrm{reset}(B,A)\Big).
$$
Now, we estimate the denominator in~\eqref{new.A-B}. In fact, this probability is equal to
$$1-P_\mu\Big(\exists s \in ]t-|\Lambda|,t], \forall C\in \mathcal{C}_s, d(e,C)< m\kappa c\ln|\Lambda|\,\big|\,\overline{D}_1\Big).
$$
By corollary~\ref{new.vcconst}, there exists a $\tilde{p}<1$ such that for $p\geqslant \tilde{p}$, for any $c,\kappa_1\geqslant 1$ and for any edge $e$ at distance more than $\kappa_1 c\ln|\Lambda|$ from $\Lambda^c$, we have
$$P_\mu\Big(\exists C\in \mathcal{C}_{t+s},e\in C\,\big|\,\exists c_t\in \mathcal{C}_t,\, d(e,c_t)\geqslant\kappa_1 c\ln|\Lambda|
\Big)\leqslant \frac{1}{|\Lambda|^{\kappa_1 c}}.$$
Since $(c+3)/c\leqslant 4$ for $c\geqslant 1$, there exists a $\kappa_1>1$, such that for any $c\geqslant 1$, 
$$\frac{1}{|\Lambda|^{\kappa_1 c}}\leqslant \frac{1}{|\Lambda|^{c+3}}.
$$
\begin{figure}[ht] 
\centering
\includegraphics[width = 10cm]{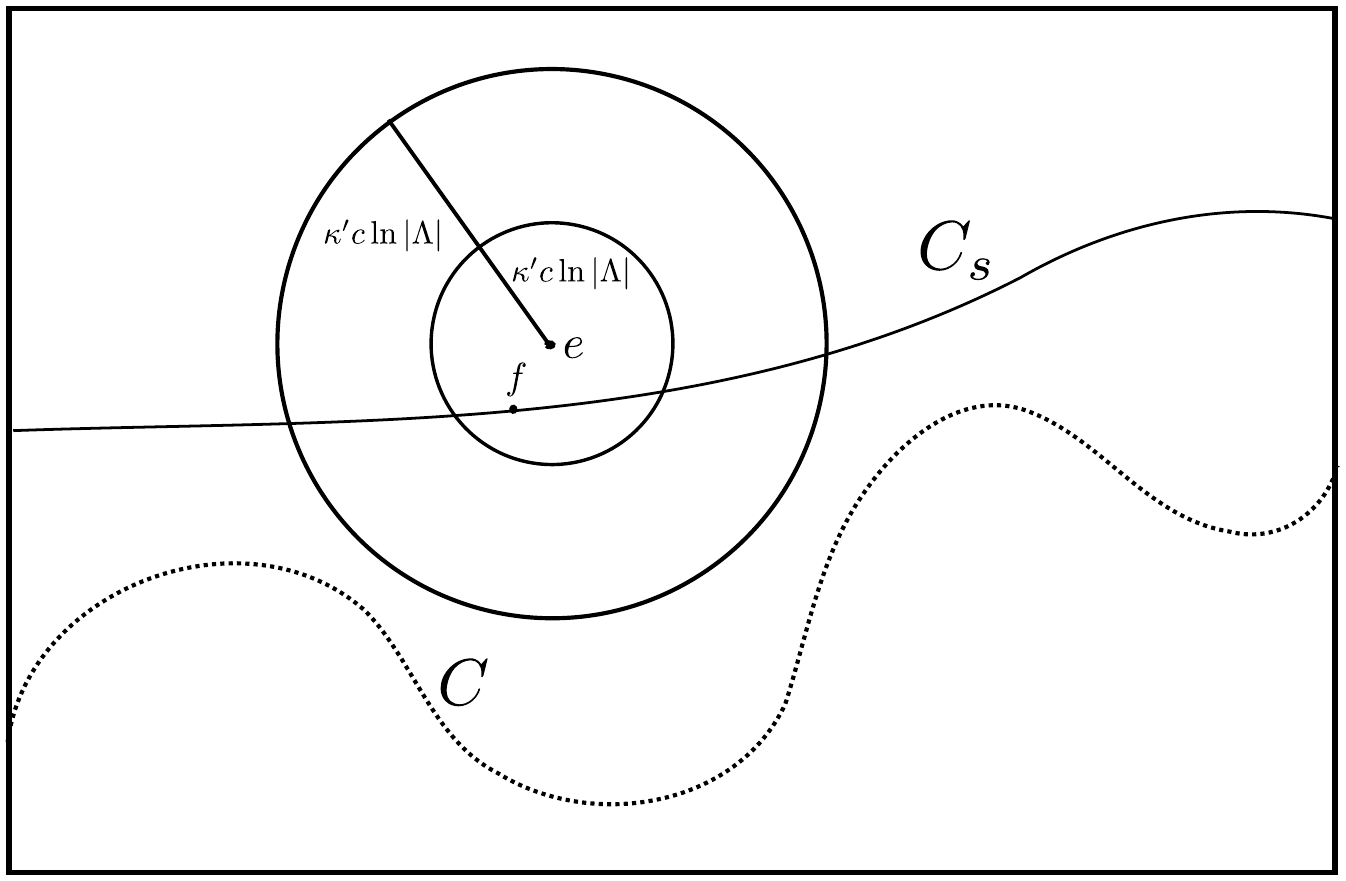}
\caption{The edge $f\in\mathcal{P}_s$ is at distance less than $\kappa' c \ln |\Lambda|$ from $e$ and the cut $C\in \mathcal{C}_{t-|\Lambda|}$ is at distance larger than $2\kappa'c \ln |\Lambda|$ from $e$. }
\label{new.fig:interval}
\end{figure}
\noindent We have therefore, as illustrated in the figure~\ref{new.fig:interval}, for $|\Lambda|\geqslant 2d$:
\begin{multline*}
P_\mu\left(\begin{array}{c|c}\begin{array}{c}\exists s \in ]t-|\Lambda|,t], \forall C'\in \mathcal{C}_s\\ d(e,C')<\kappa_1 c\ln|\Lambda|\end{array}&\overline{D}_1\end{array}\right)
\leqslant\\
\sum_{s\in]t-|\Lambda|,t]}\sum_{f:d(e,f)< \kappa_1 c\ln|\Lambda|}P_\mu \left(\begin{array}{c|c}\exists C_s\in \mathcal{C}_s, f\in C_s&\overline{D}_1\end{array}\right)\leqslant\\
\sum_{s\in]t-|\Lambda|,t]}\sum_{f:d(e,f)<\kappa_1 c\ln|\Lambda|}\frac{1}{|\Lambda|^{c+3}}\leqslant \frac{1}{|\Lambda|^{c}}\leqslant \frac{1}{2}.
\end{multline*}
Let $\kappa \geqslant \kappa_1$, the probability~\eqref{new.A-B} is less than 
$$2(1-p)^{|A\setminus B|}P_\mu\Big(\mathrm{reset}(B,A)\Big).
$$
We bound from above each term of~\eqref{new.pMk} and we obtain an upper bound for~\eqref{new.pMA}:
$$\sum_{\mbox{\scriptsize$\begin{array}{c}A\subset \mathrm{support}(\Gamma)\\ |A| = k \end{array}$}}\kern-20pt P_\mu\left(\kern-5pt\begin{array}{c|c}Y_t \in M(A)&D_1,\overline{D}_1\end{array}\kern-5pt\right)\leqslant \kern -20 pt\sum_{\mbox{\scriptsize$\begin{array}{c}A\subset \mathrm{support}(\Gamma)\\ |A| = k, B\subset A\end{array}$}}\kern -25pt 2(1-p)^{k}P_\mu\Big(\mathrm{reset}(B,A)\Big).
$$
For each set $A$ fixed, we have $$
\sum_{B\subset A}P_\mu\Big(\mathrm{reset}(B,A)\Big) = 1.
$$
Therefore, we obtain 
\begin{equation}P_\mu\left(\kern-5pt\begin{array}{c|c}Y_t \in M(A)&D_1,\overline{D}_1\end{array}\kern-5pt\right)\leqslant 2(1-p)^{|A|}\label{new.G1m}
\end{equation}
$$\sum_{\mbox{\scriptsize$\begin{array}{c}A\subset \mathrm{support}(\Gamma)\\ |A| = k \end{array}$}}\kern-20pt P_\mu\left(\kern-5pt\begin{array}{c|c}Y_t \in M(A)&D_1,\overline{D}_1\end{array}\kern-5pt\right)\leqslant 2\binom{|\mathrm{support}(\Gamma)|}{k}(1-p)^k.
$$
Finally, combined with~\eqref{new.pED}, we obtain an upper bound for~\eqref{new.n=1} which is
\begin{multline*}2(1-p)^{(\kappa c\ln|\Lambda|)/2d}\sum_{0\leqslant k \leqslant (\kappa c\ln|\Lambda|)/2d}\binom{|\mathrm{support}(\Gamma)|}{k}\left(\frac{s}{|\Lambda|}\right)^{(\kappa c\ln|\Lambda|)/2d-k}\\ \leqslant 2\left((1-p)\left(1+\frac{s}{|\Lambda|}\right)\right)^{(\kappa c\ln|\Lambda|)/2d}.
\end{multline*}
We sum over the possible choices for the path $\Gamma$, by the lemma~\ref{new.count}, the sum in~\eqref{new.n=1} is less than
$$2|\Lambda|\left(\alpha(d)(1-p)\left(1+s/|\Lambda|\right)\right)^{(\kappa c\ln|\Lambda|)/2d}.
$$
There is a $\kappa>0$, such that for $p\geqslant \tilde{p}$, such that this term is less than 
$$
\frac{1}{|\Lambda|^c}.
$$
We obtain the result stated in the lemma.
\end{proof}
We next show a generalisation of proposition~\ref{new.vconst} and corollary~\ref{new.vcconst} which is an essential ingredient for the proof of theorem~\ref{new.main1}.
\begin{prop}\label{new.lemrec}
We have the following estimate:
\begin{multline*}
\exists \tilde{p}<1\quad \exists \kappa>1 \quad \forall p \geqslant \tilde{p}\quad\forall c \geqslant 2 \quad \forall \Lambda \quad |\Lambda|\geqslant 12(2\kappa d)^d \\ \forall e\subset \Lambda \quad d(e,\Lambda^c)\geqslant \kappa c^2\ln^2|\Lambda|\\ \forall n \geqslant 1\quad n\leqslant c\ln|\Lambda| \quad \forall m\geqslant 0\quad n+m\leqslant c\ln|\Lambda|\\ \forall s\in \{1,\dots,|\Lambda|\} \quad  \forall t\geqslant n|\Lambda|\\
P_\mu\left(\begin{array}{c|c}\kern-8pt\begin{array}{c}\exists C \in \mathcal{C}_{t+s}\\d(e,C)\leqslant m\kappa c \ln|\Lambda|\end{array}\kern-5pt&\kern-5pt\begin{array}{c}\forall k\in \oset 1,\dots,n\cset\\
\forall r\in ]t-k|\Lambda|,t-(k-1)|\Lambda|]\\
\exists C_r\in \mathcal{C}_r\quad d(e,C_r)\geqslant (k+m)\kappa c\ln|\Lambda|\\
\exists C'\in \mathcal{C}_{t-n|\Lambda|}\quad d(e,C')\geqslant (n+m+1)\kappa c\ln|\Lambda| \end{array}\end{array}\kern-8pt\right)\\ \leqslant \frac{1}{|\Lambda|^c}.
\end{multline*}
\end{prop}
\begin{proof}
Notice that for the case $n = 1$, this proposition corresponds to the lemma~\ref{new.recinit}. Let $\kappa$ be a constant which will be determined at the end of the proof. We start by introducing some notations. For $m\in \mathbb{N}$ and  $k\geqslant 1$, we define $D_{k,m}$ to be the event
$$ D_{k,m}= \left\lbrace
\begin{array}{c}\forall r\in]t-k|\Lambda|,t-(k-1)|\Lambda|]\\
\exists C_r\in \mathcal{C}_r\quad d(e,C_r)\geqslant (k+m)\kappa c\ln|\Lambda|\end{array}
\right\rbrace
$$
and $\overline{D}_{k,m}$ the event
$$\overline{D}_{k,m} = \oset\exists C\in \mathcal{C}_{t-k|\Lambda|}\quad d(e,C)\geqslant (k+m+1)\kappa c\ln|\Lambda|\cset.
$$
For $\kappa \geqslant 1$, $c\geqslant 2$, $k\in \mathbb{N}^*$, $e\subset \Lambda$ and $t,s\in\mathbb{N}^*$, we denote by $(H_{k,m})$ the following inequality:
$$
\begin{array}{cc}
(H_{k,m}):&
P_\mu\left(\begin{array}{c|c}\kern-8pt\begin{array}{c}\exists C\in \mathcal{C}_{t+s}\\ d(e,C)\leqslant m\kappa c \ln|\Lambda| \end{array}\kern-5pt&
D_{1,m},\dots,D_{k,m},
\overline{D}_{k,m}\end{array}\kern-5pt\right)\leqslant \frac{1}{|\Lambda|^c}.
\end{array}
$$
Our goal is to show that there exist $\tilde{p}<1$ and $\kappa>1$ such that for $p\geqslant \tilde{p}$, $c\geqslant 2$, $e\subset \Lambda$ at distance larger than $\kappa c^2\ln^2|\Lambda|$ from $\Lambda^c$, $\smash{s\in \oset 1,\dots,|\Lambda|\cset}$, the inequality $(H_{k,m})$ holds for any $1\leqslant k+m\leqslant c\ln|\Lambda|$ and $t\geqslant (k+m)|\Lambda|$. In particular, the inequality stated in the proposition corresponds to the case $(k,m) = (n,0)$. In order to show this proposition by induction on $k$, we introduce an auxiliary inequality $(G_{k,m})$ for $A\subset \Lambda, d(e,A)\leqslant (\kappa c\ln|\Lambda|)/2$:
$$\begin{array}{cc}
(G_{k,m}):&
\kern-8ptP_\mu\Big(\begin{array}{c|c}\kern-5pt\forall f\in A\quad Y_t(f)  = 0&D_{1,m},\dots,D_{k+m},\overline{D}_{k+m}\end{array}\kern-5pt\Big)\leqslant 2^{k}(1-p)^{|A|}.
\end{array}
$$ 
By lemma~\ref{new.recinit}, there exist $\tilde{p}<1$ and $\kappa>1$ such that for $p\geqslant \tilde{p}$, $c\geqslant 2$, $e\subset \Lambda$ at distance larger than $\kappa c^2\ln^2|\Lambda|$ from $\Lambda^c$ and $t\geqslant c|\Lambda|\ln|\Lambda|$, $s\in \oset 1,\dots,|\Lambda|\cset$, the inequalities $(H_{1,m})$ hold for all $m\leqslant c\ln|\Lambda|-1$, meanwhile, the inequalities $(G_{1,m})$ was also proved in \eqref{new.G1m}. For this $\tilde{p}$, there exists a $\kappa>0$ such that, for any $c\geqslant 2$, we have
$$\left(\alpha(d)2^{1+2d/\kappa}(1-\tilde{p})\right)^{(\kappa c\ln|\Lambda|)/2d}\leqslant \frac{1}{|\Lambda|^c}.
$$
Notice that for this $\kappa$, the inequality in lemma~\ref{new.recinit} is also satisfied. Let us fix $c\geqslant 2$ and let us show the inequalities by induction on the integer $k$. Let $k<c\ln|\Lambda|$, we suppose that the inequalities $(H_{k,m})$ and $(G_{k,m})$ hold for all $m\in\mathbb{N}$ such that $k+m\leqslant c\ln|\Lambda|$. Let us prove first the inequality $(G_{k+1,m})$ for a $m\in\{0,\dots,\lfloor c\ln|\Lambda|\rfloor-k-1\}$. We reuse the notations $\mathrm{reset}(I,A,B)$ and $\overline{M}(A)$ defined for a subset $B$ of $A$ and a time interval $I$:
\begin{align*}
\mathrm{reset}(I,A,B) & = \left\lbrace\begin{array}{c} \forall e \in B,\, \exists r\in I, E_r = e\\ \forall r\in I, E_r\notin A\setminus B \end{array}\right\rbrace,\\
\overline{M}(A) &= \oset \omega\,:\,
\forall f \in A\quad \omega(f) = 0
\cset.
\end{align*}
We denote by $I_1$ the interval $]t-|\Lambda|,t]$. We rewrite the probability $(G_{k+1,m})$ as in the proof of lemma~\ref{new.recinit}:
\begin{multline*}
P_\mu\Big(\begin{array}{c|c}\forall f\in A\quad Y_t(f)  = 0&D_{1,m},\dots,D_{k+1,m},\overline{D}_{k+1,m}\end{array}\Big)\\ = \sum_{B\subset A} P_\mu\Big(\kern-5pt\begin{array}{c|c}\forall f\in A\quad Y_t(f)  = 0, \, \mathrm{reset}(I_1,A,B)&D_{1,m},\dots,D_{k+1,m},\overline{D}_{k+1,m}\end{array}\kern-5pt\Big).
\end{multline*}
For each $B\subset A$, we have 
\begin{multline*}P_\mu\Big(\begin{array}{c|c}\forall f\in A\quad Y_t(f)  = 0, \, \mathrm{reset}(I_1,A,B)&D_{1,m},\dots,D_{k+1,m},\overline{D}_{k+1,m}\end{array}\Big)\\ = P_\mu\left(\kern-5pt\begin{array}{c|c}
\begin{array}{c} \forall f \in B,\, Y_t(f) = 0\\
\forall f\in A\setminus B, Y_{t-|\Lambda|}(f) = 0\\
\mathrm{reset}(I_1,A,B)\end{array}&D_{1,m},\dots,D_{k+1,m},\overline{D}_{k+1,m}\end{array}\kern-5pt\right).
\end{multline*}
We use the same arguments as in the inequality \eqref{new.1-pB} of the lemma~\ref{new.recinit} to obtain a factor $1-p$ for each edge where the event $\mathrm{reset}$ is realised. We have 
\begin{multline*}P_\mu\left(\kern-5pt\begin{array}{c|c}
\begin{array}{c} \forall f \in B,\, Y_t(f) = 0\\
\forall f\in A\setminus B, Y_{t-|\Lambda|}(f) = 0\\
\mathrm{reset}(I_1,A,B)\end{array}&D_{1,m},\dots,D_{k+1,m},\overline{D}_{k+1,m}\end{array}\kern-5pt\right)\\
\leqslant (1-p)^{|B|}P_\mu\Big(Y_{t-|\Lambda|}\in \overline{M}(A\setminus B),\mathrm{reset}(I_1,A,B)\,\big|\,D_{1,m},\dots,D_{k+1,m},\overline{D}_{k+1,m}\Big).
\end{multline*}
The event $\mathrm{reset}(I_1,A,B)$ is independent of what happens before and until $t-|\Lambda|$ and of $D_{2,m},\dots,D_{k+1,m},\overline{D}_{k+1,m}$. Therefore, this last probability is less than or equal to
\begin{multline*}
\frac{P_\mu\Big(Y_{t-|\Lambda|}\in \overline{M}(A\setminus B),\mathrm{reset}(I_1,A,B)\,\big|\,D_{2,m},\dots ,D_{k+1,m}\overline{D}_{k+1,m}\Big)}{P_\mu\Big(D_{1,m}\,\big|\,D_{2,m},\dots,D_{k+1,m},\overline{D}_{k+1,m}\Big)}=\\
\frac{P_\mu\Big(\mathrm{reset}(I_1,A,B)\Big)}{P_\mu\Big(D_{1,m}\,\big|\,D_{2,m},\dots,D_{k+1,m},\overline{D}_{k+1,m}\Big)}\\
\times P_\mu\Big(Y_{t-|\Lambda|}\in \overline{M}(A\setminus B)\,\big|\,D_{2,m},\dots,D_{k+1,m},\overline{D}_{k+1,m}\Big).
\end{multline*}
We apply the inequality $(G_{k,m+1})$, at time $t-|\Lambda|$. The last probability is less or equal than 
$$2^{k}(1-p)^{|A\setminus B|}.
$$
For the denominator, we apply $(H_{k,m+1})$ at time $t-1$ and we obtain
\begin{multline*}P_\mu\Big(D_{1,m}\,\big|\,D_{2,m},\dots,D_{k+1,m},\overline{D}_{k+1,m}\Big)\\
\geqslant 1-P_\mu\left(\kern-10pt\begin{array}{c|c}\begin{array}{c}\exists r\in]t-|\Lambda|,t]\quad\exists C_r\in \mathcal{C}_r\\d(e, C_r)\leqslant m\kappa c\ln|\Lambda|\end{array}\kern-5pt&D_{2,m},\dots,D_{k+1,m},\overline{D}_{k+1,m}\end{array}\kern-5pt\right)
\\\geqslant 1-\frac{|\Lambda|}{|\Lambda|^{c}}.
\end{multline*}
Therefore, for $|\Lambda|\geqslant 2$, we have 
$$\frac{1}{|\Lambda|^{c-1}}\leqslant \frac{1}{2}.
$$
Therefore, we have for the denominator
$$P_\mu\Big(D_{1,m}\,\big|\,D_{2,m},\dots,D_{k+1,m},\overline{D}_{k+1,m}\Big)\geqslant \frac{1}{2}.
$$
We obtain $(G_{k+1,m})$ by summing over the choices of $B$:
\begin{multline*}
P_\mu\Big(\begin{array}{c|c}\forall f\in A\quad Y_t(f)  = 0&D_{1,m},\dots,D_{k+1,m},\overline{D}_{k+1,m}\end{array}\Big)
\\ \leqslant \frac{2^{k}(1-p)^{|A|}}{1/2}\sum_{B\subset A}P_\mu\Big(\mathrm{reset}(I_1,A,B)\Big) = 2^{k+1}(1-p)^{|A|}.
\end{multline*}
In order to obtain $(H_{k+1,m})$, we will study the STP obtained as in the corollary~\ref{new.vcconst}. We recall that this STP is of length at least $(\kappa c\ln|\Lambda|)/2d$. We fix first the space projection of the STP, which we denote by $\Gamma$. As in the proof of lemma~\ref{new.recinit} and proposition~\ref{new.vconst}, we study separately the edges that close after the time $t$ and the edges which are closed at time $t$ by conditioning the probability by the configuration $Y_t$. For the edges which become closed after $t$, we apply proposition~\ref{new.BKFermeutre} and we obtain that the probability for obtaining a simple closed decreasing STP $\gamma$ between $t$ and $t+s$ satisfying $\mathrm{Space}(\gamma) = \Gamma$ is less than
\begin{multline}\sum_{0\leqslant j \leqslant \mathrm{support}(\Gamma)}\sum_{A\subset\mathrm{support}(\Gamma) : |A| = j}\left(\frac{s}{|\Lambda|}(1-p)\right)^{(\kappa c\ln|\Lambda|)/2d-j}\times \\P\Big(Y_t\in \overline{M}(A)\,|\, D_{1,m},\dots,D_{k+1,m},\overline{D}_{k+1,m}\Big).\label{new.Hn}
\end{multline}
We apply the inequality $(G_{k+1,m})$ for the last probability and we have
$$P\Big(Y_t\in \overline{M}(A)\,|\, D_{1,m},\dots,D_{k+1,m},\overline{D}_{k+1,m}\Big) \leqslant 2^{k+1}(1-p)^j.
$$
Therefore, the sum in~\eqref{new.Hn} is less than 
\begin{multline*}
\sum_{0\leqslant j\leqslant \mathrm{support}(\Gamma)}\binom{\mathrm{support}(\Gamma)}{j}\left(\frac{s}{|\Lambda|}\right)^j 2^{k+1}(1-p)^{(\kappa c\ln|\Lambda|)/2d}
\leqslant \\ 2^{k+1}\Big(\left(1+s/|\Lambda|\right)(1-p)\Big)^{(\kappa c\ln|\Lambda|)/2d}.
\end{multline*}
For $|\Lambda|\geqslant 2d$, $k+1\leqslant c\ln |\Lambda|$ and $s\leqslant |\Lambda|$, we have $$2^{k+1}\Big(\left(1+s/|\Lambda|\right)(1-p)\Big)^{(\kappa c\ln|\Lambda|)/2d}\leqslant \left(2^{1+2d/\kappa}(1-p)\right)^{(\kappa c\ln|\Lambda|)/2d}.
$$
We sum over the choices for $\Gamma$ by using the lemma~\ref{new.count}, and we have
\begin{multline*}P_\mu\left(\kern-8pt\begin{array}{c|c}\begin{array}{c}\exists C\in \mathcal{C}_{t+s}\\
d(e,C)\leqslant m\kappa c\ln|\Lambda|\end{array}\kern-5pt&\kern-5pt\begin{array}{c}D_{1,m},\dots,D_{k+1,m},\overline{D}_{k+1,m} \end{array}\end{array}\kern-8pt\right)\\\leqslant |\Lambda|\left(\alpha(d)2^{1+2d/\kappa}(1-p)\right)^{(\kappa c\ln|\Lambda|)/2d}.
\end{multline*}
For $p\geqslant\tilde{p}$ and the $\kappa$ chosen at the beginning of the proof, for any $c\geqslant 2$, we have
$$|\Lambda|\left(\alpha(d)2^{1+2d/\kappa}(1-p)\right)^{(\kappa c\ln|\Lambda|)/2d}\leqslant \frac{1}{|\Lambda|^c}.
$$
Notice that the constant $\kappa$ doesn't depend on $k$. Therefore, the inequalities $\oset(H_{k,m}),(G_{k,m})\,:\, 1\leqslant k+m\leqslant c\ln|\Lambda| \cset$ are all satisfied for $p\geqslant\tilde{p}$ and this $\kappa$. This concludes the induction.
\end{proof}
\section{The law of an edge far from a cut} \label{new.th2}
We now show the theorem~\ref{new.main1} with the help of propositions~\ref{new.lemrec}. 

\begin{proof}[Proof of theorem~\ref{new.main1}]
Since $\mu$ is the stationary distribution of the process $(X_t,Y_t)_{t\in\mathbb{N}}$, we can choose a time $t$ and show the result for the configuration $(X_t,Y_t)$. For a time $r\in \mathbb{N}$ and a distance $\ell>0$, we introduce the events 
$$ \overline{D}(r,\ell) = \oset \exists C\in \mathcal{C}_r,d(e,C)\geqslant \ell\cset$$
and 
$$D(r,\ell) = \oset \forall \theta \in ]r,r+|\Lambda|], \, \exists C_\theta \in \mathcal{C}_\theta,d(e,C_\theta)\geqslant \ell \cset.
$$
We have to estimate the probability 
\begin{equation}
\label{new.probmain}P_\mu\left( e\in \mathcal{I}_t\,\Big|\,\overline{D}(t,\kappa'c^2\ln^2|\Lambda|)\right),
\end{equation}
where $\kappa'$ is a constant which will be determined later. For the moment, we can simply consider a large $\kappa'$.
We notice first that, on the event $\overline{D}(t,\kappa'c^2\ln^2|\Lambda|)$, there is a cut which is disjoint from $e$, so the edge $e$ cannot be pivotal, thus $$P_\mu\left( e\in \mathcal{I}_t\,\Big |\,\overline{D}(t,\kappa'c^2\ln^2|\Lambda|)\right) = P_\mu\left(e\in \mathcal{I}_t\setminus \mathcal{P}_t\,\Big |\, \overline{D}(t,\kappa'c^2\ln^2|\Lambda|)\right).$$ 
We consider the last time when $e$ is pivotal, i.e., the time $t-s$ defined by 
$$s = \inf \,\big\{\,r\geqslant 0 \,:\, e\in \mathcal{P}_{t-r}\,\big\}.
$$
On the interval $]t-s,t]$, the edge $e$ is not pivotal and it remains in the interface. Therefore, this edge is not modified during this interval, so we have 
$$P_\mu\left( e\in \mathcal{I}_t\,\Big |\,\overline{D}(t,\kappa'c^2\ln^2|\Lambda|)\right)\leqslant P_\mu\left(\kern-5pt\begin{array}{c|c}\kern-5pt\begin{array}{c}\exists s \geqslant 0,e\in \mathcal{P}_{t-s}\\ \forall r\in ]t-s,t] \\ e\notin \mathcal{P}_{r}
, E_r\neq e
\end{array}\kern-5pt&\kern-1pt\overline{D}(t,\kappa'c^2\ln^2|\Lambda|)\end{array}\kern-6pt\right).
$$
The events appearing in this probability concern only the process $(E_t)_{t\in\mathbb{N}}$ and the process $(Y_t)_{t\in\mathbb{N}}$. These processes are both reversible. By reversing the time, we obtain that
\begin{multline*}P_\mu\left( \begin{array}{c|c}\begin{array}{c}\exists s \geqslant 0,e\in \mathcal{P}_{t-s}\\ \forall r\in ]t-s,t] \\ e\notin \mathcal{P}_{r}
, E_r\neq e
\end{array}&\overline{D}(t,\kappa'c^2\ln^2|\Lambda|)\end{array}\right)=\\ P_\mu\left( \begin{array}{c|c}\begin{array}{c}\exists s \geqslant 0,e\in \mathcal{P}_{t+s}\\ \forall r\in ]t,t+s] \\ e\notin \mathcal{P}_{r}
, E_r\neq e
\end{array}&\overline{D}(t,\kappa'c^2\ln^2|\Lambda|)\end{array}\right).
\end{multline*}
Notice that the sequence $(E_r)_{t< r\leqslant t+s}$ is independent of the configuration $Y_t$. We estimate first the probability that the interval $]t,t+s]$ is too large. More precisely, we will show that $s$ is at most of order $|\Lambda|\ln |\Lambda|$. Let $c\geqslant 1$ be a constant. We have
\begin{multline}\label{new.psbig}
P_\mu\left( \begin{array}{c|c}\begin{array}{c}\exists s \geqslant c|\Lambda|\ln|\Lambda|,e\in \mathcal{P}_{t+s}\\ \forall r\in ]t,t+s] \\ e\notin \mathcal{P}_{r}
, E_r\neq e
\end{array}&\overline{D}(t,\kappa'c^2\ln^2|\Lambda|)\end{array}\right)
\leqslant \\P_\mu\left(\begin{array}{c} \forall r\in ]t,t+c|\Lambda|\ln|\Lambda|] \\  E_r\neq e
\end{array}\right)\leqslant \left(1-\frac{1}{|\Lambda|}\right)^{c|\Lambda|\ln|\Lambda|} \leqslant \frac{1}{|\Lambda|^c}.
\end{multline}
We now consider the case where $s<c |\Lambda|\ln |\Lambda|$. We split the interval $[t,t+s]$ into subintervals of length $|\Lambda|$. We set, for $0\leqslant i< s/|\Lambda|$, $$t_i = t+i|\Lambda|.$$
Let us distinguish two cases according to the positions of the cuts during the time interval $]t,t_1]$. We consider a constant $\kappa>0$ which will be chosen later. If the event $D(t,\kappa'c^2\ln^2|\Lambda|-\kappa c\ln|\Lambda|)$ doesn't occur, then there exists a time $\tau\in ]t,t_1]$ and a cut of $\mathcal{C}_\tau$ which visits at least an edge $f$ at distance less than $\kappa'c^2\ln^2|\Lambda|-\kappa c\ln|\Lambda|$ from $e$. Therefore, for a $s<c|\Lambda|\ln|\Lambda|$ fixed, we have
\begin{multline*}P_\mu\left(\begin{array}{c|c}e\in \mathcal{P}_{t+s}&\overline{D}(t,\kappa'c^2\ln^2|\Lambda|)\end{array}\right)\leqslant\\
P_\mu\left(\begin{array}{c|c}\begin{array}{c} \exists \tau\in]t,t_1],\exists C_\tau\in \mathcal{C}_\tau, \exists f\in C_\tau \\ d(e,f)\leqslant \kappa'c^2\ln^2|\Lambda|-\kappa c\ln |\Lambda|\end{array} & \overline{D}\left(t,\kappa'c^2\ln^2|\Lambda|\right)\end{array}\right)+ \\
P_\mu\left(\begin{array}{c|c}\kern-10pt\begin{array}{c} e\in \mathcal{P}_{t+s}\\D\left( t,\kappa'c^2\ln^2|\Lambda|-\kappa c\ln |\Lambda|\right)\end{array}\kern-5pt &\kern-2pt \overline{D}\left(t,\kappa'c^2\ln^2|\Lambda|\right)\end{array}\kern-5pt\right).
\end{multline*}
We estimate the first probability with the help of corollary~\ref{new.vcconst}. This case is illustrated in figure~\ref{new.fig:interval} but this time with the radius of the circles taken to be $\kappa'c^2\ln^2|\Lambda|$ and $\kappa'c^2\ln^2|\Lambda|-\kappa c\ln|\Lambda|$. There is a $\tilde{p}<1$, such that, for $p\geqslant \tilde{p}$ and $\kappa>0$, for any $c\geqslant 2$, $0<\tau\leqslant|\Lambda|$ and an edge $f$ at distance less than $\kappa'c^2\ln^2|\Lambda|-\kappa c\ln |\Lambda|$ from $e$, we have 
$$P_\mu\left(\kern-8pt\begin{array}{c|c}\begin{array}{c}\exists C_\tau \in \mathcal{C}_\tau, f\in C_\tau\end{array}\kern-8pt&\kern-1pt \overline{D}(t,\kappa'c^2\ln^2|\Lambda|)\end{array}\kern-5pt\right)\leqslant \frac{1}{|\Lambda|^{c+2}}.
$$
Therefore, the following inequality holds:
\begin{multline*}P_\mu\left(\begin{array}{c|c}\begin{array}{c} \exists \tau\in]t,t_1],\exists C_\tau\in \mathcal{C}_\tau, \exists f\in C_\tau \\ d(e,f)\leqslant \kappa'c^2\ln^2|\Lambda|-\kappa c\ln |\Lambda|\end{array} & \overline{D}\left(t,\kappa'c^2\ln^2|\Lambda|\right)\end{array}\right)\leqslant \\
\sum_{\tau\in ]t,t_1]}P_\mu\left(\kern-8pt\begin{array}{c|c}\begin{array}{c}\exists C_\tau \in \mathcal{C}_\tau,d(e,C_\tau)\leqslant \kappa'c^2\ln^2|\Lambda|-\kappa c\ln |\Lambda|\end{array}\kern-8pt&\kern-1pt \overline{D}(t,\kappa'c^2\ln^2|\Lambda|)\end{array}\kern-5pt\right)\leqslant  \frac{2d}{|\Lambda|^c}.
\end{multline*}
We then obtain 
\begin{multline*}
P_\mu\left(\kern-10pt\begin{array}{c|c}\begin{array}{c}e\in \mathcal{P}_{t+s}\end{array}&\overline{D}(t,\kappa'\ln^2|\Lambda|)\end{array}\kern-5pt\right)\leqslant\\\frac{2d}{|\Lambda|^c}+P_\mu\left(\kern-10pt\begin{array}{c|c} \begin{array}{c}e\in \mathcal{P}_{t+s}\end{array} \kern-5pt&\kern-5pt \begin{array}{c} D\left( t,\kappa'\ln^2|\Lambda|-\kappa\ln |\Lambda|\right)\\ \overline{D}\left(t,\kappa'\ln^2|\Lambda|\right)\end{array}\end{array}\kern-10pt\right).
\end{multline*}
Starting from this inequality, we apply proposition~\ref{new.lemrec} and repeat the previous argument at the times $t_i$, $0 \leqslant i<s/|\Lambda|$. By iteration, we obtain that, for any $n<s/|\Lambda|$ and $|\Lambda|\geqslant 12(2\kappa d)^d$,
\begin{multline}\label{new.forec}
P_\mu\left(\kern-10pt\begin{array}{c|c}\begin{array}{c}e\in \mathcal{P}_{t+s}\end{array}&\overline{D}(t,\kappa'c^2\ln^2|\Lambda|)\end{array}\kern-5pt\right)\leqslant\\\frac{2dn}{|\Lambda|^c}+P_\mu\left(\kern-10pt\begin{array}{c|c} \begin{array}{c}e\in \mathcal{P}_{t+s}\end{array} \kern-5pt&\kern-5pt \begin{array}{c} \bigcap_{1\leqslant i\leqslant n}\kern-5pt D\left( t_i,\kappa'c^2\ln^2|\Lambda|-i\kappa c\ln |\Lambda|\right)\\ \overline{D}\left(t,\kappa'c^2\ln^2|\Lambda|\right)\end{array}\end{array}\kern-10pt\right).
\end{multline}
We consider this inequality with $n = \lfloor s/|\Lambda|\rfloor<c\ln|\Lambda|$:
\begin{multline*}
P_\mu\left(\kern-10pt\begin{array}{c|c}\begin{array}{c}e\in \mathcal{P}_{t+s}\end{array}&\overline{D}(t,\kappa'c^2\ln^2|\Lambda|)\end{array}\kern-5pt\right)\leqslant\\ \frac{2dc\ln|\Lambda|}{|\Lambda|^c}+
P_\mu\left(\kern-10pt\begin{array}{c|c} \begin{array}{c}e\in \mathcal{P}_{t+s}\end{array} \kern-5pt&\kern-5pt \begin{array}{c} \bigcap_{1\leqslant i< s/|\Lambda|}\kern-10pt D\left( t_i,\kappa'c^2\ln^2|\Lambda|-i\kappa c\ln |\Lambda|\right)\\ \overline{D}\left(t,\kappa'c^2\ln^2|\Lambda|\right)\end{array}\end{array}\kern-10pt\right).
\end{multline*}
We notice that $s-|\Lambda|\lfloor s/|\Lambda|\rfloor<|\Lambda|$ and there exists a $ \kappa'>0$ such that 
$$\kappa'c^2\ln^2|\Lambda| - \kappa c\ln|\Lambda|\lfloor s/|\Lambda|\rfloor \geqslant \kappa c\ln|\Lambda|.
$$ 
We can apply again proposition~\ref{new.lemrec} at time $t_n$ and we get 
$$P_\mu\left(\kern-10pt\begin{array}{c|c} \begin{array}{c}e\in \mathcal{P}_{t+s}\end{array} \kern-5pt&\kern-5pt \begin{array}{c} \bigcap_{1\leqslant i< c\ln|\Lambda|}\kern-10pt D\left( t_i,\kappa'c^2\ln^2|\Lambda|-i\kappa c\ln |\Lambda|\right)\\ \overline{D}\left(t,\kappa'c^2\ln^2|\Lambda|\right)\end{array}\end{array}\kern-10pt\right)\leqslant \frac{1}{|\Lambda|^c}.
$$
Finally, we obtain the following upper bound for~\eqref{new.forec}:
$$P_\mu\left(\kern-10pt\begin{array}{c|c}\begin{array}{c}e\in \mathcal{P}_{t+s}\end{array}&\overline{D}(t,\kappa'c^2\ln^2|\Lambda|)\end{array}\kern-5pt\right)\leqslant \frac{2dc\ln|\Lambda|+1}{|\Lambda|^c}.
$$
We sum over the choices of $s< c|\Lambda|\ln|\Lambda|$ and we combine with~\eqref{new.psbig}. We obtain
\begin{multline*} P_\mu\left(\kern-3pt \begin{array}{c|c}\begin{array}{c}\exists s \geqslant 0,e\in \mathcal{P}_{t+s}\\ \forall r\in ]t,t+s] \\ e\notin \mathcal{P}_{r}
, E_r\neq e
\end{array}&\overline{D}(t,\kappa'c^2\ln^2|\Lambda|)\end{array}\kern-3pt\right)\leqslant \\\frac{1+c|\Lambda|\ln|\Lambda|+2d|\Lambda|c^2\ln^2|\Lambda|}{|\Lambda|^c}.
\end{multline*}
For $|\Lambda|> 4+c+2dc^2+12(2\kappa d)^d$, we have $\ln|\Lambda|\leqslant |\Lambda|$ and thus
$$
\frac{1+c|\Lambda|\ln|\Lambda|+2|\Lambda|dc^2\ln^2|\Lambda|}{|\Lambda|^c}\leqslant \frac{1+c+2dc^2}{|\Lambda|^{c-3}} \leqslant \frac{1}{|\Lambda|^{c-4}}.
$$
Therefore, there exists a $\tilde{p}<1$ and a $\kappa'>0$ such that for $p\geqslant \tilde{p}$, for any $c\geqslant 2$, we have
$$P_\mu\left(e\in \mathcal{I}_t\setminus \mathcal{P}_t\,\Big |\, \overline{D}(t,\kappa'c^2\ln^2|\Lambda|)\right)\leqslant \frac{1}{|\Lambda|^{c-4}}.
$$
Since $(c+4)^2/c^2\leqslant 25$ for $c\geqslant 1$, by replacing $\kappa'$ by $25\kappa'$ in the probability, we can replace $1/|\Lambda|^{c-4}$ by $1/|\Lambda|^c$. Hence the desired result.
\end{proof}
\def\cP{\mathcal{P}}
\def\cI{\mathcal{I}}

\section{The construction of the impatient STP}\label{eps.constp}
In order to improve the control of the speed of the pivotal edges, we
will construct a new STP which connects an edge $e\in\mathcal{P}_{t}$ at time $t$ and the set $\mathcal{P}_s\cup\mathcal{I}_s$ at time $s<t$. Before starting the construction, we define first some relevant properties of a STP, which will be enjoyed by our construction.
\begin{defi}
A STP $(e_1,t_1),\dots,(e_n,t_n)$ is impatient if every time-change is ended by an edge which is updated, i.e.,
$$\forall i\in \{1,\dots,n-2\}\qquad e_i = e_{i+1} \Rightarrow E_{t_{i+1}+1} = e_{i+1}.
$$
\end{defi}
\begin{defi}
A STP $(e_1,t_1),\dots,(e_n,t_n)$ is called $X$-closed-moving (resp. $Y$-closed-moving) if all the edges which are not time-change edges are closed in $X$ (resp. in $Y$), i.e.,
$$\forall i\in \{1,\dots,n-1\} \quad e_i \neq e_{i+1} \Rightarrow X_{t_i}(e_i) = 0 \quad(\text{resp. }Y_{t_i}(e_i) = 0).
$$
\end{defi}
\noindent We now construct a specific STP satisfying some of these properties.
\begin{prop}\label{eps.consstp}
Let $s<t$ be two times and $e\in \mathcal{P}_t$. There exists a decreasing simple impatient STP which connects the time-edge $(e,t)$ to an edge of the set $\mathcal{P}_s\cup\mathcal{I}_s\setminus\{e\}$ at time $s$ or an edge $f$ intersecting the boundary of $\Lambda$ after time $s$. Moreover this STP is $X$-closed-moving except on the edge $e$.
\end{prop}
\begin{proof}
The proof of this proposition is done in two steps. The first step is to construct a STP which connects certain edges. In the second step, we modify the STP obtained in the first step to get a simple and impatient STP.
\paragraph{Step 1.}
At time $t$, the edge $e$ belongs to a cut. Therefore, there exists a path $\gamma_1$ which connects $e$ to the boundary of $\Lambda$. We start at the edge $e$ and we follow the path $\gamma_1$. If the path $\gamma_1$ does not encounter an edge $f\in \mathcal{I}_t\cup \mathcal{P}_{t-1}\setminus\{e\}$ then the STP $$(e,t),(\gamma_1,t)$$ connects $e$ to the boundary of $\Lambda$, where the notation $(\rho,t)$, for a path $\rho= (e_i)_{1\leqslant i\leqslant n}$ and a time $t$, means the sequence of time-edges $(e_i,t)_{1\leqslant i\leqslant n}$. Suppose next that there exists an edge of $\mathcal{I}_t\cup\mathcal{P}_{t-1}\setminus\{e\}$ in $\gamma_1$. We enumerate the edges of $\gamma_1$ in the order they are visited when starting from $e$ and we consider the first edge $e_1$ in $\gamma_1$ which belongs to the set $\mathcal{I}_t\cup\mathcal{P}_{t-1}$. We denote by $\rho_1$ the sub-path of $\gamma_1$ visited between $e$ and $e_1$. We then consider the time $\eta(t)$ defined as follows:
$$ \eta(t) = \max\oset r<t\,:\, X_r(e_1) =Y_r(e_1) \cset.
$$
Since $e_1\in \mathcal{I}_t$, the time $\eta(t) $ when it becomes an edge of the interface is strictly less than $t$ and if $e_1\in\mathcal{P}_{t-1}\setminus \mathcal{I}_t$, we have $\eta(t)\leqslant t-1$. In both cases, we have $\eta(t)< t$ and the edge $e_1$ is closed in $X_{\eta(t)}$. If the time $\eta(t)$ is before the time $s$ then, at time $s$, the edge $e_1$ belongs to the set $\mathcal{I}_s\cup\mathcal{P}_s\setminus \{e\}$. Therefore the STP
$$(e,t),(\rho,t),(e_1,t),(e_1,s)
$$
satisfies the conditions in the proposition. If we have $\eta(t)> s$, then we repeat the above argument starting from the edge $e_1$ at time $\eta(t)$. We obtain either a path $\gamma_2$ which connects $e_1$ to the boundary of $\Lambda$ at time $\eta(t)$ or a path $\rho_2$ which connects $e_1$ to an edge $e_2\in \mathcal{I}_{\eta(t)}\cup \mathcal{P}_{\eta(t)-1}\setminus\{e\}$ and a time $\eta^2(t)<\eta(t)$. We proceed in this way until we reach a time edge $(e_k,\eta^k(t))$ with $\eta^k(t)\leqslant s$. Since $\eta(t)<t$, the sequence of times $$
\eta(t),\eta^2(t),\dots,\eta^k(t)
$$decreases strictly through this procedure and this procedure terminates after a finite number of iterations. The concatenation of the paths obtained at the end of the procedure,
$$(e,t),(\rho_1,t),(e_1,\eta(t)),\dots,(\rho^k,\eta^{k-1}(t)),(f_k,s),
$$
connects $e$ to an edge of $\mathcal{P}_s\cup\mathcal{I}_s\setminus\{e\}$. Since the sequence $(\eta^i(t))_{1\leqslant i\leqslant k}$ is decreasing, this is a decreasing STP. Each time when the STP meets an edge of $\mathcal{I}$ which is different from $e$, there is a time change to the time before it opened in $X$, therefore each movement in space except on the edge $e$ is done through a closed edge in $X$ and the STP is $X$-closed-moving. 
\paragraph{Step 2.} We use two iterative procedures to transform the STP in the step 1 into a simple and impatient STP. To get a simple STP, we use the same procedure as in the proof of proposition 4.4 in~\cite{new}. Let us denote by $(e_i,t_i)_{0\leqslant i\leqslant N}$ the STP obtained previously. Starting with the edge $e_0$, we examine the rest of the edges one by one. Let $i\in \oset 0,\dots,N\cset$. Suppose that the edges $e_0,\dots,e_{i-1}$ have been examined and let us focus on $e_i$. We encounter three cases:
\begin{itemize}[leftmargin = 0.4cm]
\item For every index $j\in \{i+1,\dots,N\}$, we have $e_j\neq e_i$. Then, we don't modify anything and we start examining the edge $e_{i+1}$.
\item There is an index $j\in\{i+1,\dots,N\}$ such that $e_i = e_j$, but for the first index $k>i+1$ such that $e_i = e_k$, there is a time $\alpha\in ]t_k,t_i[$ when $X_\alpha(e_i) = 1$. Then we don't modify anything and we start examining the next edge $e_{i+1}$.
\item There is an index $j\in\{i+1,\dots,N\}$ such that $e_i = e_j$ and for the first index $k> i+1$ such that $e_i = e_k$, we have $X_\alpha(e_i) = 0$ for all $\alpha\in ]t_k,t_i[$. In this case, we remove all the time-edges whose indices are strictly between $i$ and $k$. We then have a simple time change between $t_i$ and $t_k$ on the edge $e_i$. We continue the procedure from the index $k$.
\end{itemize}
The STP becomes strictly shorter after every modification (we remove systematically the consecutive time changes if there is any), and the procedure will end after a finite number of modifications. We obtain in the end a simple path in $X$. Since the procedure doesn't change the order of the times $t_i$, we still have a decreasing path. In order to obtain an impatient STP, we modify the simple decreasing STP obtained above and we use another iterative procedure as follows. We denote again by $(e_i,t_i)_{0\leqslant i \leqslant n}$ the simple STP obtained above. We start by examining the time-edge $(e_0,t_0)$ and then the rest of the time edges of the STP one by one as illustrated in the figure~\ref{eps.fig:imp}. Suppose that we have examined the indices $i< k$ and that we are checking the index $k$. If the edge $e_{k+1}$ is different from $e_k$, we don't modify the STP at this stage and we continue the procedure from $(e_{k+1},t_{k+1})$.
\begin{figure}
\centering
\resizebox{120mm}{!}{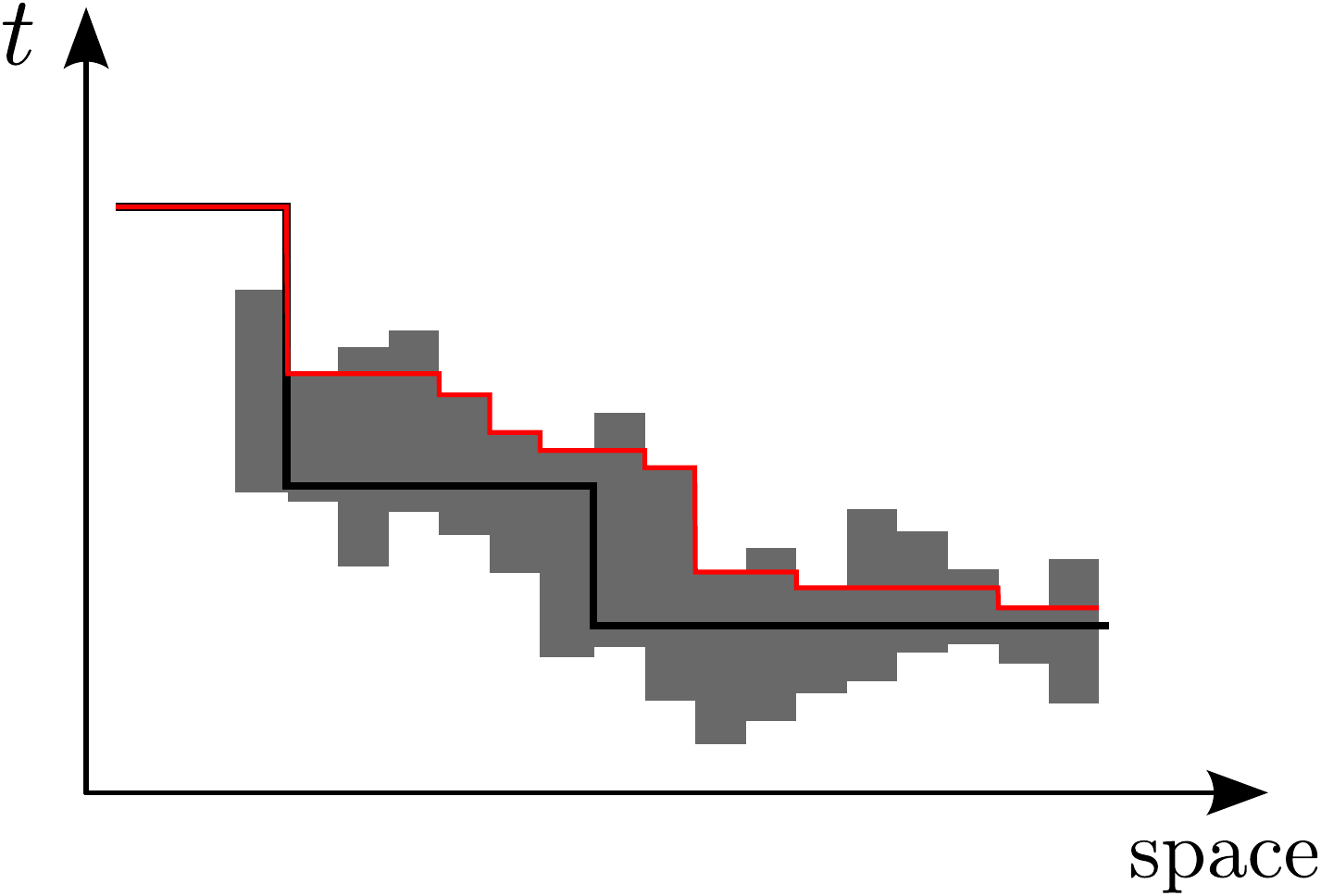}
\caption{An impatient modification (in red) of a STP (in black) according to the intervals when each edge is closed (in gray)}
\label{eps.fig:imp}
\end{figure}
If the edge $e_{k+1}$ is equal to $e_k$, then the time-edge $(e_k,t_k)$ belongs to a time change. Since the STP is $X$-closed-moving, then the edge $e_{k+1}$ is closed at time $t_{k+1}$. Let $[\alpha,\beta]$ be the biggest interval containing $t_{k+1}$ during which the edge $e_{k+1}$ is closed in $X$. If $\beta\geqslant t_k$ and $e_{k+2}\neq e_{k+3}$, we replace the sub-sequence $$(e_k,t_k),(e_{k+1},t_{k+1}),(e_{k+2},t_{k+2})$$
by 
$$(e_k,t_k),(e_{k+2},t_k),(e_{k+2},t_{k+2}),
$$
and we continue the modification from the time-edge $(e_{k+2},t_k)$.
If $\beta\geqslant t_k$ and $e_{k+2}= e_{k+3}$ we replace the sequence
$$(e_k,t_k),(e_{k+1},t_{k+1}),(e_{k+2},t_{k+2}),(e_{k+3},t_{k+3})$$ by
$$(e_k,t_k),(e_{k+2},t_k),(e_{k+2},t_{k+3}),
$$
and we continue the modification from the time-edge $(e_{k+2},t_k)$.
If $\beta<t_k$ and $e_{k+2}\neq e_{k+3}$, we replace $(e_k,t_k),(e_{k+1},t_{k+1})$ by $$(e_k,t_k),(e_k,\beta),(e_{k+2},\beta).$$
If $\beta<t_k$ and $e_{k+2}= e_{k+3}$, we replace 
$(e_k,t_k),(e_{k+1},t_{k+1}),(e_{k+2},t_{k+2})$ by $$(e_k,t_k),(e_k,\beta),(e_{k+2},\beta),$$
and we continue the STP at the time-edge $(e_{k+2},\beta)$. The STP obtained after the modification procedure is decreasing, $X$-closed-moving and impatient. Moreover, between two consecutive visits of an edge $f$ of the STP, there exists a time when the edge $f$ is open. Therefore, this STP is also simple. 
\end{proof}

\section{Exponential decay of the new STP}\label{eps.speedcontrol}
We show here that the set $\mathcal{P}\cup\mathcal{I}$ cannot move too fast. Typically, during an interval of size $|\Lambda|\ln|\Lambda|$, the set $\mathcal{P}\cup\mathcal{I}$ can at most move a distance of order $\ln|\Lambda|$. This result relies on an estimate for the STP constructed in proposition~\ref{eps.consstp} which we state in the following proposition.
\begin{lem}\label{eps.lemstp}
Let $e$ be an edge in $\Lambda$ and $\ell\in\mathbb{N}^*$. Let $(\varepsilon_1,\dots,\varepsilon_n)$ be a sequence of edges such that $|\mathrm{support}(\varepsilon_1,\dots,\varepsilon_n)| = \ell$. We have the following inequality:
\begin{multline*}\exists \tilde{p}<1 \quad \forall p\geqslant \tilde{p} \quad \forall s,t \quad 0<t-s\leqslant \ell|\Lambda| \\
P_\mu\left(\begin{array}{c}
\exists \gamma \text{ decreasing simple impatient}\\X\text{-closed-moving STP except on }e,\\
\gamma \text{ starts from }(e,t)\text{ and ends after }s,\\
\mathrm{space}(\gamma) = (\varepsilon_1,\dots,\varepsilon_n)
\end{array}\right)\leqslant \left(1+\frac{1}{|\Lambda|}\right)^{\ell|\Lambda|}(4-4p)^n.
\end{multline*}
\end{lem}
\begin{proof}
Let us fix a STP $\gamma$ satisfying the conditions stated in the probability. We denote by $(e_i,t_i)_{i\in I}$ the sequence of the time-edges of $\gamma$. We denote by $k$ the number of the time changes in $\gamma$ and by $T$ the set of the indices of the time changes, i.e.,
$$T = \oset i\in I \,:\, e_i = e_{i+1},\quad t_i \neq
 t_{i+1}\cset.
$$ 
We shall obtain an upper bound of the probability 
\begin{equation}
P \left(\begin{array}{c}
(e_i,t_i)_{i\in I}\text{ is a decreasing simple impatient}\\X\text{-closed-moving STP except on }e
\end{array}\right),
\label{eps.sumpb}
\end{equation}
which depends only upon the integer $n$ and the number of time changes $k$. In order to bound the probability appearing in the lemma, we shall sum over the choices of the set of the $k$ times, denoted by $K$, in the interval $\{s,\dots,t\}$, over the choices of set of the $k$ edges, denoted by $A$, where the time changes occur and the number $k$ from $1$ to $n$. The probability appearing in the lemma is less or equal than
\begin{equation}\label{eps.sum6.1}\sum_{1\leqslant k\leqslant n}\kern-5pt\sum_{\footnotesize\begin{array}{c}A\subset\{1,\dots,n\}\\|A|=k\end{array}}\kern-10pt\sum_{\footnotesize\begin{array}{c}K\subset\{s,\dots,t\}\\|K| = k\end{array}}\kern-20ptP \left(\begin{array}{c}
(e_i,t_i)_{i\in I}\text{ is a decreasing simple impatient}\\X\text{-closed-moving STP except on }e
\end{array}\kern-2pt\right).
\end{equation}
Let us obtain an upper bound for this probability. The STP is impatient and $X$-closed-moving, therefore for any $i\in T$, the edge $e_{i+1}$ becomes open at time $t_{i+1}+1$. Moreover, the STP is simple, thus for any pair of indices $(p,q)\in I\setminus T$, if $e_p = e_q$ and $t_p>t_q$, there exists a time $r\in ]t_q,t_p[$, such that the edge $e_p$ is open at time $r$. We can rewrite the probability inside the sum as 
\begin{equation}P_\mu\left(\begin{array}{c}
\forall i \in T \quad E_{t_{i+1}+1} = e_{i+1}\\
\forall i\in I\setminus T\quad X_{t_i}(e_i) = 0\\
\forall p,q\in I\setminus T\text{ s.t. } e_p = e_q, t_p>t_q\\
\exists r\in ]t_q,t_p[\quad X_r(e_p) = 1
\end{array}\right).\label{eps.probstc}
\end{equation}
Since the times $t_i$ are fixed, this probability can be factorised as a product over the edges. In fact, the event in the probability depends only on the process $(X_t)_{t\in\mathbb{N}}$. We introduce, for an edge $f\subset\Lambda$, the subset $J(f)$ of $I$: 
$$J(f) = \oset i\in I\,:\, e_i = f\cset.
$$ Let us denote by $S$ the set $\mathrm{support}(\gamma)$. The previous probability is less or equal than
\begin{equation}\prod_{f\in S\setminus\{e\}}P_\mu\left(\begin{array}{c}
\forall i \in J(f)\cap (T+1) \quad E_{t_i+1} = f\\
\forall i\in J(f)\setminus T\quad X_{t_i}(f) = 0\\
\forall p,q\in J(f)\setminus T \text{ s.t. }p <q\\
\exists r\in ]t_q,t_p[\quad X_r(f) = 1
\end{array}\right).\label{eps.pedge}
\end{equation}
Let us consider one term of the product. For a fixed edge $f$, we can order the set $\oset t_i\,:\,i\in J(f)\setminus T\cset$ in an increasing sequence $( \tau_i)_{1\leqslant i\leqslant m_f}$, where $m_f = |J(f)\setminus T|$. Let us denote by $T(f)$ the set of the indices among $\{1,\dots,m_f\}$ which correspond to the end of a time change, i.e., the set corresponding to $J(f)\cap (T+1)$ before the reordering.
Since the STP is simple, between two consecutive visits at times $\tau_i$ and $\tau_{i+1}$ of $f$, there is a time $\theta_i$ when $f$ is open. Moreover the STP is impatient, so for each index $i\in T(f)$, the edge $f$ becomes open at time $\tau_i+1$. Therefore, each term of the product~\eqref{eps.pedge} is less or equal than 
\begin{equation}
P_\mu\left(\begin{array}{c}
\forall i\in \{1,\dots, m_f\}\quad  X_{\tau_i}(f)= 0\\
\forall i\in T(f) \quad X_{\tau_{i}+1}(f) = 1\\
\forall i\in \{1,\dots, m_f-1\}\quad \exists\theta_i\in]\tau_i,\tau_{i+1}[\quad X_{\theta_i}(f) =1\label{eps.pF}
\end{array}\right).
\end{equation}
In order to simplify the notations, we define, for a time $r$, the event 
$$\mathcal{E}(r) = \left\lbrace\begin{array}{c}
\forall i\in \{1,\dots, m_f\}\text{ such that } \tau_i\leqslant r\quad X_{\tau_i}(f)= 0\\
\forall i\in T(f)\text{ such that } \tau_i+1\leqslant r \quad X_{\tau_{i}+1}(f) = 1\\
\forall i\in \{1,\dots, m_f-1\}\text{ such that } \tau_i\leqslant r\\ \exists\theta_i\in]\tau_i,\tau_{i+1}[\quad X_{\theta_i}(f) =1
\end{array} \right\rbrace.
$$
The status of the edge $f$ in the process $(X_t)_{t\in\mathbb{N}}$ evolves according to a Markov chain on $\{0,1\}$. The sequence $(\tau_i)_{1\leqslant i \leqslant m_f}$ being fixed, if $m_f\in T(f)$, we condition \ref{eps.pF} by the events before time $\tau_{m_f}$, we have 
\begin{multline*}P_\mu\left(\begin{array}{c}
\forall i\in \{1,\dots, m_f\}\quad  X_{\tau_i}(f)= 0\\
\forall i\in T(f) \quad X_{\tau_{i}+1}(f) = 1,\quad m_f\in T(f)\\
\forall i\in \{1,\dots, m_f-1\}\quad \exists\theta_i\in]\tau_i,\tau_{i+1}[\quad X_{\theta_i}(f) =1
\end{array}\right)=\\ P_\mu\left(\kern-5pt\begin{array}{c|c}X_{\tau_{m_f}+1}(f) = 1&\mathcal{E}(\tau_{m_f})\end{array}\right)P_\mu\big(\mathcal{E}(\tau_{m_f})\big)\leqslant \frac{P_\mu\big(\mathcal{E}(\tau_{m_f})\big)}{|\Lambda|}.
\end{multline*}
If $m_f\notin T(f)$, the probability
$$P_\mu\left(\begin{array}{c}
\forall i\in \{1,\dots, m_f\}\quad  X_{\tau_i}(f)= 0\\
\forall i\in T(f) \quad X_{\tau_{i}+1}(f) = 1,\quad m_f\notin T(f)\\
\forall i\in \{1,\dots, m_f-1\}\quad \exists\theta_i\in]\tau_i,\tau_{i+1}[\quad X_{\theta_i}(f) =1
\end{array}\right)
$$
is equal to $P_\mu\big(\mathcal{E}(\tau_{m_f})\big)$. We then condition $P_\mu\big(\mathcal{E}(\tau_{m_f})\big)$ by the events before time $\tau_{m_f-1}$. We shall distinguish two cases according to whether $m_f-1$ belongs to $T(f)$ or not. If $m_f-1\in T(f)$, we have 
$$P_\mu\big(\mathcal{E}(\tau_{m_f})\big) = P_\mu\left(\begin{array}{c|c}\begin{array}{c}X_{\tau_{m_f}}(f)=0\\X_{\tau_{m_f-1}+1}(f) = 1\end{array}&\mathcal{E}(\tau_{m_f-1})\end{array}\right)P_\mu\big(\mathcal{E}(\tau_{m_f-1})\big),
$$
and if $m_f-1\notin T(f)$, we have 
$$P_\mu\big(\mathcal{E}(\tau_{m_f})\big)\\ = P_\mu\left(\begin{array}{c|c}\begin{array}{c}X_{\tau_{m_f}}(f)=0\\\exists\theta_{m_f}\in]\tau_{m_f-1},\tau_{m_f}[\\ X_{\theta_{m_f}}(f) =1\end{array}&\mathcal{E}(\tau_{m_f-1})\end{array}\right)P_\mu\big(\mathcal{E}(\tau_{m_f-1})\big).
$$
We condition successively the event $P_\mu\big(\mathcal{E}(\tau_{i})\big)$ by $\mathcal{E}(\tau_{i-1})$, we obtain
\begin{multline}
P_\mu\big(\mathcal{E}(\tau_{m_f})\big) =  P_\mu\big(\mathcal{E}(\tau_1)\big)\prod_{1\leqslant i<m_f,i\in T(f)}P_\mu\left(\begin{array}{c|c}\begin{array}{c}X_{\tau_{i+1}}(f)=0\\X_{\tau_{i}+1}(f) = 1\end{array}&\mathcal{E}(\tau_i)\end{array}\right)\\ \times
\prod_{1\leqslant i<m_f, i\notin T(f)}P_\mu\left(\begin{array}{c|c}\begin{array}{c}X_{\tau_{i+1}}(f)=0\\\exists \theta_{i+1}\in ]\tau_{i},\tau_{i+1}[\\X_{\theta_i}(f) = 1\end{array}&\mathcal{E}(\tau_i)\end{array}\right).\label{eps.pMarkov}
\end{multline}
By the Markov property, each term in the second product is equal to 
$$P_\mu\left(\begin{array}{c|c}\begin{array}{c}X_{\tau_{i+1}}(f)=0\\\exists \theta_i\in ]\tau_{i},\tau_{i+1}[\\X_{\theta_{i+1}}(f) = 1\end{array}&X_{\tau_i}(f) =0\end{array}\right).
$$
Since this probability is invariant by translation in time, it is equal to 
$$P_0\left(\begin{array}{c}X_{\tau'}(f)=0\\\exists \theta\in ]0,\tau'[\\X_\theta(f) = 1\end{array}\right),
$$
where we have set $\tau' = \tau_{i+1}-\tau_i$ and $P_0$ is the law of the Markov chain $(X_t(f))_{t\in\mathbb{N}}$ starting from a closed edge. By considering the stopping time $\theta'$ defined as the first time after $0$ when $f$ is open, we have by the strong Markov property
$$P_0\left(\begin{array}{c}X_{\tau'}(f)=0\\\exists \theta\in ]0,\tau'[\\X_\theta(f) = 1\end{array}\right)\leqslant P_0\big(X_{\tau'}(f) =0\,\big|\, X_{\theta'}(f) =1\big) = P_1\big(X_{\tau'-\theta'}(f) =0\big).
$$
Notice that for $r\geqslant 1$, we have 
$$P_1\big(X_{r}(f)=0\big)\leqslant P_\mu\big(X_{r}(f)=0\big) = 1-p.
$$
Therefore we have
$$P_0\left(\begin{array}{c}X_{\tau'}(f)=0\\\exists \theta\in ]0,\tau'[\\X_\theta(f) = 1\end{array}\right)\leqslant 1-p.
$$
As for the probabilities in the first product of \eqref{eps.pMarkov}, we can also replace $\mathcal{E}(\tau_i)$ by $\{X_{\tau_i}(f)=0\}$ in the conditioning. The difference between the previous case is that we have directly $\theta' = 1$, since $X_{\tau_{i}+1}(f) = 1$. We have
$$P_\mu\left(\begin{array}{c|c}\begin{array}{c}X_{\tau_{i+1}}(f)=0\\X_{\tau_{i}+1}(f) = 1\end{array}&\mathcal{E}(\tau_i)\end{array}\right)\leqslant P_1\big(X_{\tau'-1}(f) =0\big)P_0\big(X_1(f) =1\big)\leqslant\kern-2pt \frac{1-p}{|\Lambda|}.
$$
Combining the upper bounds for each term of the product above, we have the following upper bound for $P_\mu\big(\mathcal{E}(\tau_{m_f})\big)$:
$$P_\mu\big(\mathcal{E}(\tau_{m_f})\big)\leqslant\frac{(1-p)^{m_f}}{|\Lambda|^{|T(f)\cap\{1,\dots,m_f-1\}|}},
$$
where
$$|T(f)\cap\{1,\dots,m_f-1\}|= \left\lbrace \begin{array}{cl}|J(f)\cap (T+1)| &\text{if } m_f\notin T(f)\\
|J(f)\cap (T+1)|-1 &\text{if }m_f\in T(f)\end{array}\right..
$$
In both cases, we have the following upper bound for \eqref{eps.pF}:
$$P_\mu\left(\kern-2pt\begin{array}{c}
\forall i\in \{1,\dots, m_f\}\quad  X_{\tau_i}(f)= 0\\
\forall i\in T(f) \quad X_{\tau_{i}+1}(f) = 1\\
\forall i\in \{1,\dots, m_f-1\}\quad \exists\theta_i\in]\tau_i,\tau_{i+1}[\quad X_{\theta_i}(f) =1
\end{array}\kern-2pt\right)\leqslant \frac{2(1-p)^{m_f}}{|\Lambda|^{|J(f)\cap (T+1)|}}.
$$
We obtain an upper bound for~\eqref{eps.probstc} by multiplying this inequality over the edges $f$ in $\mathrm{support}(\gamma)$:
\begin{multline}P_\mu\left(\begin{array}{c}
\forall i \in T \quad E_{t_{i+1}} = e_{i+1}\\
\forall i\in I\setminus T\quad X_{t_i}(e_i) = 0\\
\forall p,q\in I\setminus T\text{ s.t. } e_p = e_q, t_p>t_q\\
\exists r\in ]t_q,t_p[\quad X_r(e_p) = 1
\end{array}\right)\\
\leqslant \frac{2^{|S|}(1-p)^{\sum_{f\in S}m_f}}{|\Lambda|^{\sum_f|J(f)\cap(T+1)|}}\leqslant \frac{2^{|S|}(1-p)^{|I|-k}}{|\Lambda|^k}.\label{eps.ptifixed}
\end{multline}
Since $|I|-k\geqslant n$, and $|S|\leqslant n$, for $k$ fixed and $(t_i)_{i\in I}$ fixed, we have the following upper bound for \eqref{eps.sumpb},
$$P \left(\begin{array}{c}
(e_i,t_i)_{i\in I}\text{ is a decreasing simple impatient}\\X\text{-closed-moving STP except on }e
\end{array}\right)\leqslant \frac{(2-2p)^{n}}{|\Lambda|^k}.
$$ 
Finally, we use this upper bound in \eqref{eps.sum6.1} and we have
\begin{multline*}
P_\mu\left(\begin{array}{c}
\exists \gamma \text{ decreasing simple impatient}\\X\text{-closed-moving STP except on }e,\\
\gamma \text{ starts from }(e,t)\text{ and ends after }s,\\
\mathrm{space}(\gamma) = (\varepsilon_1,\dots,\varepsilon_n)
\end{array}\right)\\
\leqslant
\sum_{1\leqslant k\leqslant n}\sum_{A\subset\{1,\dots,n\},|A|=k}\sum_{K\subset\{s,\dots,t\},|K| = k}\kern-5pt\frac{(2-2p)^{n}}{|\Lambda|^k}\\\leqslant \sum_{1\leqslant k\leqslant n} \binom{n}{k}\binom{\ell|\Lambda|}{k}\frac{(2-2p)^{n}}{|\Lambda|^k}
\leqslant \sum_{1\leqslant k\leqslant n} \binom{\ell|\Lambda|}{k}\frac{(4-4p)^{n}}{|\Lambda|^k}\\\leqslant \left(1+\frac{1}{|\Lambda|}\right)^{\ell|\Lambda|}(4-4p)^n.
\end{multline*}
This yields the desired result.
\end{proof}
We use next proposition~\ref{eps.consstp} and lemma~\ref{eps.lemstp} to show that the pivotal edges cannot move too fast. 
\begin{prop}\label{eps.speed}
There exists $\tilde{p}<1$, such that for $p\geqslant \tilde{p}$, for $\ell\geqslant 1$, $t\in\mathbb{N}$, $s\in \mathbb{N}$, $s\leqslant \ell|\Lambda|$ and any edge $e$ at distance at least $\ell$ from the boundary of $\Lambda$,
$$P_\mu\Big(e\in \mathcal{P}_{t+s},\, d(e,\mathcal{P}_t\cup\mathcal{I}_t\setminus\{e\})\geqslant \ell\Big)\leqslant \exp(-\ell).
$$
\end{prop}
\begin{proof}
By proposition~\ref{eps.consstp}, there exists a STP which is decreasing simple impatient and $X$-closed-moving except on $e$ which starts from the edge $e$ at time $t+s$ and ends at an edge of $\mathcal{P}_t\cup\mathcal{I}_t\setminus\{e\}$ or an edge intersecting the boundary of $\Lambda$ after the time $t$.  In both cases, this STP has a length at least $\ell$. Therefore, we have the inequality
\begin{multline*}P_\mu\Big(e\in \mathcal{P}_{t+s},\, d(e,\mathcal{P}_t\cup\mathcal{I}_t\setminus\{e\})\geqslant \ell\Big)\\ \leqslant P_\mu \left(\kern-5pt\begin{array}{c}
\exists \gamma \text{ decreasing simple impatient}\\X\text{-closed-moving STP except on }e\\
\gamma \text{ starts from }(e,t+s)\text{ and ends after }t\\
|\mathrm{length}(\gamma)|\geqslant \ell
\end{array}\kern-5pt\right).
\end{multline*}
Let us fix a path $(e_1,\dots,e_n)$ with $n = \ell$ starting from $e$. By lemma~\ref{eps.lemstp}, for $\ell\geqslant 1$, we have
$$
P_\mu\left(\begin{array}{c}
\exists \gamma \text{ decreasing simple impatient}\\X\text{-closed-moving STP except on }e,\\
\gamma \text{ starts from }(e,t+s)\text{ and ends after }t,\\
\mathrm{space}(\gamma) = (e_1,\dots,e_\ell)
\end{array}\right)\leqslant \left(1+\frac{1}{|\Lambda|}\right)^{\ell|\Lambda|}(4-4p)^\ell.
$$
We sum over the number of the choices for the path $(e_1,\dots,e_\ell)$ and we obtain
$$P_\mu \left(\kern-5pt\begin{array}{c}
\exists \gamma \text{ decreasing simple impatient}\\X\text{-closed-moving STP except on }e\\
\gamma \text{ starts from }(e,t+s)\text{ and ends after }t\\
|\mathrm{length}(\gamma)|\geqslant \ell
\end{array}\kern-5pt\right)\leqslant \left(1+\frac{1}{|\Lambda|}\right)^{\ell|\Lambda|}\kern-9pt\alpha(d)^\ell(4-4p)^\ell,
$$
where $\alpha(d)$ is the number of the $*$-neighbours of an edge in dimension $d$.
There exists a $\tilde{p}<1$ such that for $p\geqslant \tilde{p}$, we have 
$$\forall \ell\geqslant 1\quad \left(1+\frac{1}{|\Lambda|}\right)^{\ell|\Lambda|}\alpha(d)^\ell(4-4p)^\ell\leqslant e^{-\ell}.
$$ 
This gives the desired upper bound.
\end{proof}

\section{The proof of theorem \ref{eps.main}}\label{eps.proofmain}
We now prove theorem \ref{eps.main} with the help of proposition~\ref{eps.speed} and the observation that an edge of the interface cannot survive a time more than $O(|\Lambda|\ln|\Lambda|)$.

\begin{proof}[Proof of theorem \ref{eps.main}]
Let $c$ be a constant bigger than 1. We define two sets $\mathfrak{P}_t^-$ and $\mathfrak{P}_t^+$ as
\begin{align*}
\mathfrak{P}_t^- = \bigcup_{r\in[t-2dc|\Lambda|\ln|\Lambda|,t]}\cP_r\\ \mathfrak{P}_t^+ = \bigcup_{s\in [t,t+2dc|\Lambda|\ln|\Lambda|]}\mathcal{P}_s.
\end{align*}
By the definition of $d_H^\ell$, we have
\begin{multline*}P_\mu\left( d_H^{2d c\ln|\Lambda|}\big(\mathfrak{P}_t^+,\mathfrak{P}_t^-\big)\geqslant 2d c\ln|\Lambda|\right)\\ = P_\mu\left(\begin{array}{c}
\exists s\in [0,2dc|\Lambda|\ln|\Lambda|]\quad \exists e\in \mathcal{P}_{t+s}\\d\Big(e,\Lambda^c\cup \mathfrak{P}_t^-\Big)\geqslant 2d c\ln|\Lambda|
\end{array}\right) \\+ P_\mu\left(\begin{array}{c}
\exists s\in [0,2dc|\Lambda|\ln|\Lambda|]\quad \exists e\in \mathcal{P}_{t-s}\\d\Big(e,\Lambda^c\cup \mathfrak{P}_t^+\Big)\geqslant 2d c\ln|\Lambda|
\end{array}\right).
\end{multline*}
Since the probability concerns only the process $(Y_t)_{t\in\mathbb{N}}$, which is reversible, we have
\begin{multline*}P_\mu\left(\begin{array}{c}
\exists s\in [0,2dc|\Lambda|\ln|\Lambda|]\quad \exists e\in \mathcal{P}_{t+s}\\d\Big(e,\Lambda^c\cup \mathfrak{P}_t^-\Big)\geqslant 2d c\ln|\Lambda|
\end{array}\right) \\= P_\mu\left(\begin{array}{c}
\exists s\in [0,2dc|\Lambda|\ln|\Lambda|]\quad \exists e\in \mathcal{P}_{t-s}\\d\Big(e,\Lambda^c\cup \mathfrak{P}_t^+\Big)\geqslant 2d c\ln|\Lambda|
\end{array}\right).
\end{multline*}
Therefore, we can concentrate on the following probability
$$P_\mu\left(\begin{array}{c}
\exists s\in [0,2dc|\Lambda|\ln|\Lambda|]\quad \exists e\in \mathcal{P}_{s}\\d\Big(e,\Lambda^c\cup \mathfrak{P}_t^-\Big)\geqslant 2d c\ln|\Lambda|
\end{array}\right).
$$
Let us fix an edge $e$ in $\Lambda$ at distance at least $2d c\ln|\Lambda|
$ from $\mathfrak{P}_t^-$ and a $s\in [0,2dc|\Lambda|\ln|\Lambda|]$. We distinguish two cases. If the set $\mathcal{P}_t\cup\mathcal{I}_t$ is at distance more than $2dc\ln|\Lambda|$ from the edge $e$, then by proposition \ref{eps.speed}, there exists a $p_2<1$ such that, for $p\geqslant p_2$ and $c\geqslant 1$, we have 
\begin{equation}P_\mu\big(e\in\mathcal{P}_{t+s},d(e,\mathcal{P}_t\cup\mathcal{I}_t)\geqslant 2dc\ln|\Lambda|\big)\leqslant \exp(-2dc\ln|\Lambda|).\label{eps.PIloin}
\end{equation}
If there exists an edge $f\in\mathcal{P}_t\cup \mathcal{I}_t$, which is at distance less than $2dc\ln|\Lambda|$ from $e$, we consider the last time when $f$ was pivotal before $t$ and we define the random integer $\tau$ such that 
$$\tau = \inf\oset r\geqslant 0\,:\, f\in \mathcal{P}_{t-r}\cset.
$$
Since $f\notin\mathfrak{P}_t^-$, we must have $\tau\geqslant 2dc|\Lambda|\ln|\Lambda|$. The edge $f$ is not pivotal during the time interval $[t-\tau+1,t]$ and it belongs to the interface. Moreover, it cannot be chosen to be modified during this interval since it must remain different in the two processes. Therefore, for any $r\in [t-2dc|\Lambda|\ln|\Lambda|+1,t]$, we have $E_r \neq f$. However, this event is unlikely because the sequence $(E_t)_{t\in\mathbb{N}}$ is a sequence of i.i.d. random edges chosen uniformly in $\Lambda$. More precisely, we have the following inequality:
\begin{multline*}
P_\mu\big(\tau \geqslant 2dc|\Lambda|\ln|\Lambda|\big)\leqslant P\Big(\forall r\in [t-2dc|\Lambda|\ln|\Lambda|+1,t], E_r\neq f\Big)\\ \leqslant\left(1-\frac{1}{2d|\Lambda|}\right)^{2dc|\Lambda|\ln|\Lambda|}\leqslant \frac{1}{|\Lambda|^c}.
\end{multline*}
We obtain the following inequality:
\begin{equation}P_\mu\left(\begin{array}{c}
e\in\mathcal{P}_{t+s}\quad\exists f\in\mathcal{P}_t\cup\mathcal{I}_t\\
d(e,f)< 2dc\ln|\Lambda|\\
d(e,\Lambda\cup\mathfrak{P}_t^-)\geqslant 2dc\ln|\Lambda|
\end{array}\right)\leqslant \frac{\lambda(d)(2dc\ln|\Lambda|)^d}{|\Lambda|^c},\label{eps.sbig}
\end{equation}
where $\lambda(d)$ is a constant depending only on the dimension.
We combine the two cases \eqref{eps.PIloin} and \eqref{eps.sbig}, we obtain
$$P_\mu\big(e\in\mathcal{P}_{t+s}, d(e,\Lambda\cup\mathfrak{P}_t^-)\geqslant 2dc\ln|\Lambda|\big)\leqslant \frac{\lambda(d)(2dc\ln|\Lambda|)^d}{|\Lambda|^c}+\frac{1}{|\Lambda|^{2dc}}.
$$
We then sum over the number of the choices for the edge $e$ and of the number $s$ from $1$ to $2dc|\Lambda|\ln|\Lambda|$. We obtain 
\begin{equation}\label{eps.ssmall}
P_\mu\left(\begin{array}{c}
\exists s\in [t,t+2dc|\Lambda|\ln|\Lambda|]\quad \exists e\in \mathcal{P}_{t+s}\\d\Big(e,\Lambda^c\cup \mathfrak{P}_t^-\Big)\geqslant 2d c\ln|\Lambda|
\end{array}\right)\leqslant \frac{\lambda(d)(2dc\ln|\Lambda|)^{d+1}}{|\Lambda|^{c-2}}+\frac{2dc\ln|\Lambda|}{|\Lambda|^{2dc-2}}.
\end{equation}
In other words, we have
$$P_\mu\left(\begin{array}{c}
\mathfrak{P}_t^+\nsubseteq\mathcal{V}\big(\Lambda^c\cup \mathfrak{P}_t^-,2d c\ln|\Lambda|\big)
\end{array}\right) \leqslant \frac{\lambda(d)(2dc\ln|\Lambda|)^{d+1}}{|\Lambda|^{c-2}}+\frac{2dc\ln|\Lambda|}{|\Lambda|^{2dc-2}}.
$$
By the reversibility of the process $(Y_t)_{t\in\mathbb{N}}$, we also have
$$P_\mu\left(\begin{array}{c}
\mathfrak{P}_t^-\nsubseteq\mathcal{V}\big(\Lambda^c\cup \mathfrak{P}_t^+,2d c\ln|\Lambda|\big)
\end{array}\right)\\ \leqslant \frac{\lambda(d)(2dc\ln|\Lambda|)^{d+1}}{|\Lambda|^{c-2}}+\frac{2dc\ln|\Lambda|}{|\Lambda|^{2dc-2}}.
$$
Combining the two previous inequalities, we have
$$P_\mu\left(
d_H^{2dc\ln|\Lambda|}\big(\mathfrak{P}_t^-,\mathfrak{P}_t^+\big)\geqslant 2dc \ln|\Lambda|
\right)\\
 \leqslant \frac{2\lambda(d)(2dc\ln|\Lambda|)^{d+1}}{|\Lambda|^{c-2}}+\frac{4dc\ln|\Lambda|}{|\Lambda|^{2dc-2}}.
$$
For $|\Lambda|\geqslant e^{2d^2c}$, we have 
$$\frac{2\lambda(d)(2dc\ln|\Lambda|)^{d+1}}{|\Lambda|^{c-2}}+\frac{4dc\ln|\Lambda|}{|\Lambda|^{2dc-2}}\leqslant \frac{1}{|\Lambda|^{c-3}}.
$$
This yields the desired result.
\end{proof}

\bibliographystyle{halpha}
\bibliography{intperco}
\end{document}